\documentclass[12pt,a4paper,reqno, final]{amsart}

\usepackage[utf8]{inputenc}
\usepackage[english]{babel}
\usepackage{amsmath, amsfonts, amsthm, amssymb, amscd}
\usepackage[foot]{amsaddr}
\usepackage{graphicx}
\usepackage{lmodern} 
\usepackage{esint}
\PassOptionsToPackage{dvipsnames}{xcolor}
\usepackage{tikz}
\usepackage{epsfig}
\usepackage{floatflt}

\usepackage{comment}
\usepackage{showkeys}

\usepackage{todonotes}

\usepackage{comment}
\usepackage[T1]{fontenc}        
\usepackage{url}
\usepackage{setspace}
\usepackage{nicefrac}
\usepackage[colorlinks, linkcolor = black, citecolor = black, filecolor = black, urlcolor = blue]{hyperref}
\usepackage{fancyref}
\usepackage{commath}
\usepackage{bookmark}

\newtheorem{theorem}{Theorem}[section]
\newtheorem{proposition}[theorem]{Proposition}
\newtheorem{lemma}[theorem]{Lemma}

\newtheorem{definition}[theorem]{Definition}
\newtheorem{remark}[theorem]{Remark}

\newtheorem*{theorem*}{Theorem}
\numberwithin{equation}{section}
\def\Xint#1{\mathchoice
{\XXint\displaystyle\textstyle{#1}}%
{\XXint\textstyle\scriptstyle{#1}}%
{\XXint\scriptstyle\scriptscriptstyle{#1}}%
{\XXint\scriptscriptstyle\scriptscriptstyle{#1}}%
\!\int}
\def\XXint#1#2#3{{\setbox0=\hbox{$#1{#2#3}{\int}$ }
\vcenter{\hbox{$#2#3$ }}\kern-.6\wd0}}

\def\dashint{\Xint-}

\newcommand{\supp}{\operatorname{supp}}

\newcommand{\essup}{\operatorname{esssup}}

\newcommand{\R}{{\mathbb R}}
\newcommand{\N}{{\mathbb N}}
\newcommand{\Z}{\mathbb{Z}}

\newcommand{\expect}{{\mathbb E}}

\newcommand{\ort}{{\mathcal O}}
\newcommand{\intort}{{\int_\ort}}
\renewcommand{\norm}[2]{\left\|{#1}\right\|_{#2}}

\renewcommand{\set}[1]{{\left\{#1\right\}}}

\newcommand{\probab}{{\mathbb P}}
\newcommand{\tprobab}{{\tilde\probab}}

\definecolor{JF1}{rgb}{0.85,0.1,0.1}
\definecolor{JF1i}{rgb}{0.1,0.8,0.2}
\definecolor{JF1e}{rgb}{0.95,0.7,0.1}

\newcommand{\WienerProcessSpace}{Q^{1/2}L^2(\domain)}

\newcommand{\E}{\mathbb{E}}

\newcommand{\Basis}{(\Omega,\mathcal{F},(\mathcal{F}_t)_{t\geq 0},\mathbb{P})}

\newcommand{\Basislimit}{(\tilde \Omega,\tilde{\mathcal{F}},(\tilde{\mathcal{F}}_t)_{t\geq 0},\tprobab)}

\newcommand{\uspace}{C^{\gamma,\gamma/4}(\domain\times [0,T_{max}])}
\newcommand{\pspace}{L^2([0,T_{max}];H^1_{per}(\domain))}
\newcommand{\Jspace}{L^2(\domain \times [0,T_{max}])}
\newcommand{\wspace}{C([0,T_{max}];L^2(\ort))}

\newcommand{\Ih}{\mathcal{I}_h}
\newcommand{\domain}{\mathcal{O}}

\newcommand{\tJ}{\tilde{J}}
\newcommand{\tJh}{\tilde{J}^h}
\newcommand{\tp}{\tilde{p}}
\newcommand{\tu}{\tilde{u}}
\newcommand{\tph}{\tilde{p}^h}
\newcommand{\tuh}{\tilde{u}^h}
\newcommand{\chitth}{\chi_{T_h}}

\newcommand{\uh}{u^h}

\newcommand{\Hper}[1]{H_{\text{per}}^#1(\domain)}
\newcommand{\Mh}{{\mathcal{M}_{h,v}}}
\newcommand{\tMh}{\widetilde{\mathcal{M}}_{h,v}}
\newcommand{\tM}{\widetilde{\mathcal{M}_v}}
\newcommand{\tMhq}{\widetilde{\mathcal{M}}_{h,v}^2}

\newcommand{\tMq}{\widetilde{\mathcal{M}_v}^2}
\newcommand{\Ph}{{\mathcal{P}_h}}
\newcommand{\Lh}{{L_h}}

\newcommand{\crossl}{\langle\langle}
\newcommand{\crossr}{\rangle\rangle}

\newcommand{\sqmob}{{M_2^h}}
\newcommand{\porad}{{A_\Delta^h}}

\newcommand{\porb}{{B_\Delta^h}}
\newcommand{\discint}{{h\sum_{i=1}^{L_h}}}

\newcommand{\tThInt}{{\int_0^{t\wedge T_h}}}
\newcommand{\Li}{{\mathcal L}}

\newcommand{\ui}{{u^{h}_{i}}}
\newcommand{\uip}{{u^{h}_{i+1}}}
\newcommand{\uim}{{u^{h}_{i-1}}}
\newcommand{\tui}{{{\tilde {u}}^h_i}}
\newcommand{\tuip}{{{\tilde {u}}^h_{i+1}}}
\newcommand{\tuim}{{{\tilde {u}}^h_{i-1}}}

\allowdisplaybreaks[2]

\numberwithin{equation}{section}

\newcommand{\cO}{\mathcal{O}}

\newcommand{\dx}{\dif x}

\newcommand{\dt}{\dif t}
\newcommand{\ds}{\dif s}

\newcommand{\ti}[1]{\tilde{#1}}

\newcommand{\un}{u_{0}}

\newcommand{\Cstr}{C_{\text{Strat}}}

\newcommand{\intT}{{\int_0^{T}}}

\begin{document}
\title[]{Existence of nonnegative energy-dissipating solutions to a class of stochastic thin-film equations under weak slippage: Part I -- positive solutions}
\date{June 11, 2024}

\author[G.~Gr\"un]{G\"unther Gr\"un}
\address[G.~Gr\"un]{Friedrich--Alexander--Universität Erlangen--Nürnberg, Cauerstraße 11, 91058 Erlangen, Germany}
\email{gruen@math.fau.de}

\author[L.~Klein]{Lorenz Klein}
\address[L.~Klein]{Friedrich--Alexander--Universität Erlangen--Nürnberg, Cauerstraße 11, 91058~Erlangen, Germany}
\email{lorenz.klein@fau.de}


\makeatletter
\@namedef{subjclassname@2020}{%
	\textup{2020} Mathematics Subject Classification}
\makeatother 


\begin{abstract}
For mobility exponents $n \in (2,3)$, we prove existence of strictly positive solutions to stochastic thin-film equations with singular effective interface potential and Stratonovich-type lower order terms. With the perspective of using these solutions in \cite{GruenKlein24Limit} to construct surface-tension-energy dissipating solutions to stochastic thin-film equations with compactly supported initial data, for which finite speed of propagation is shown in \cite{GruenKlein24FSOP}, we establish decay estimates on the sum of surface-tension  energy and effective interface potential -- without relying on further functionals involving initial data. 
Besides an extension of earlier techniques used in the case $n=2$ and a refinement of oscillation estimates for discrete solutions, the main analytical novelty of this paper is a discretization method  which shows nonnegativity for a finite-element counterpart of the integral $\intort (u^{n-2}u_{x})_x u_{xx} \dx$ under periodic boundary conditions in the parameter regime $n \in (2,3)$. 
This nonnegativity property serves to control It\^o-correction terms in the estimate for the decay of the surface-tension energy. This way, it is the key to obtain the desired decay estimates for the sum of surface-tension  energy and effective interface potential which permit to establish the singular limit of vanishing effective interface potential in \cite{GruenKlein24Limit}. 
\end{abstract}

\maketitle

\tableofcontents

\section{Introduction}
This is the first paper in a series of three papers which are concerned with the proof of existence and of finite speed of propagation for energy-dissipating solutions of stochastic thin-film equations of the class
\begin{align}\label{eq:stfeComp}
	du = -(u^{n} u_{xxx})_{x} \dt + (\Cstr+S) (u^{n-2} u_{x})_{x} \dt + (u^{n/2} dW_{Q}(t))_{x}
\end{align}
in the parameter regime $n \in (2,3)$.
For $S=0$, $\Cstr$ as in (H5) of Section~\ref{sec:discr}, a standard $Q$-Wiener process $W_{Q}(x,t) = \sum_{\ell \in \Z} \lambda_{\ell} g_{\ell}(x) \beta_{\ell}(t)$, $(g_{\ell})_{\ell \in \Z}$ the usual sine ($\ell >0$) and cosine ($\ell <0$) basis functions, and frequency balancing $\lambda_{\ell} = \lambda_{-\ell} \, \, \forall \ell \in \N$, this is exactly the surface-tension driven thin-film equation 
\begin{align}
	du = - (u^{n} u_{xxx})_{x} \dt + (u^{n/2}\circ dW_{Q}(t))_{x} \, ,
\end{align}
i.e. the basic stochastic thin-film equation with Stratonovich noise, while for 
$S=\Cstr$, it is the stochastic thin-film equation with Hänggi-Klimontovich (or backward It\^o) noise, see for example \cite{Karatzas2014} or \cite{Kuo2006}.
Note that \eqref{eq:stfeComp} can equivalently be written as
\begin{align}\label{eq:Enlog}
	du = -(u^{n}(u_{xx} + (\Cstr +S)u^{-1})_{x})_{x} \dt + (u^{n/2}dW_{Q}(t))_{x} \, .
\end{align} 
This suggests that the corresponding energy is given by $\mathcal{E}_{log}[u]:= \intort \frac{1}{2}u_{x}^{2} -\log  u \dx$ which would become singular for compactly supported initial data. 
Formal computations, however, show that \eqref{eq:stfeComp} dissipates the surface-tension energy $\intort \frac{1}{2} u_{x}^{2}\dx$ for appropriate choices of $\Cstr$, $S$, and $n\in(2,3)$ -- see Section~\ref{sec:discr} for more details.

Hence, it is the scope of this paper and the subsequent publication \cite{GruenKlein24Limit} to establish the existence of energy-dissipating solutions to \eqref{eq:stfeComp} subject to compactly supported initial data, i.e. martingale solutions which satisfy the energy estimate
\begin{align}\label{eq:EnEnt} \notag
	&\mathbb{E} \Bigg[ \sup_{t \in [0,T]} \left(\frac{1}{2} \intort u_{x}^{2}(t) \dx \right)^{q} \Bigg]
	+
	\mathbb{E} \Bigg[ \left(\intT \intort \left|\left(u^{\frac{n+2}{6}}\right)_{x}\right|^{6} \dx \ds \right)^{q} \Bigg]
	\\ 
	&+
	\mathbb{E} \Bigg[ \left(\intT \intort \left|\left(u^{\frac{n+2}{2}}\right)_{xxx}\right|^{2} \dx \ds \right)^{q} \Bigg]
	+
	\mathbb{E} \Bigg[ \left(\intT \intort \left|\left(u^{\frac{n}{4}}\right)_{x}\right|^{4} \dx \ds \right)^{q} \Bigg]
	\le C(\un,q,T) \, .
\end{align}
Weighted versions of this estimate will be a crucial tool for the proof of finite speed of propagation in \cite{GruenKlein24FSOP}.
To establish this existence result for energy-dissipating solutions, we apply a three-step approximation procedure. 
In the first step, we prove the existence of $\mathbb{P}$-almost surely strictly positive solutions to the equation
\begin{align}\label{eq:stfeSingPot}
	du=-(u^{n} (u_{xx} - F'(u))_{x})_x \dt+ (\Cstr+S) (u^{n-2} u_{x})_{x} \dt +\sum_{\ell \in \Z} \big(\lambda_\ell u^{n/2} g_\ell\big)_x \, d\beta_\ell(t)
\end{align}
with an effective interface potential $F$ generically given as $F = c_{F}u^{-p}$, $c_{F}>0$, $p>n$.
This is the topic of the present paper. 
In \cite{GruenKlein24Limit}, we show that for $c_{F} \to 0$ solutions to \eqref{eq:stfeSingPot} with fixed strictly positive initial data approach solutions of \eqref{eq:stfeComp}. 
By means of a decay estimate for the entropy $\int_\ort u^{2-n} \dx$, sufficient regularity is provided to apply Bernis inequalities \cite{BernisInequ1996, IntermediateScaling, GiacomelliShishkov}
to control third order derivatives of appropriate powers of solutions. The last step is about the passage to the limit from  strictly positive to merely nonnegative initial data -- resulting in an existence result for  energy-dissipating solutions subject to compactly supported initial data, satisfying in particular \eqref{eq:EnEnt}. 
\\
There is a vast literature on analysis, numerics, and modelling aspects of deterministic thin-film equations.
For weak existence theory see \cite{BernisFriedman, BerettaBertschDalPasso, BertozziPugh1996, Otto, BertschGiacomelliKarali, Mellet} and in higher dimensions \cite{DGGG, GruenCauchy} and the references therein. A corresponding classical theory has been developed in \cite{GiacomelliGnannKnuepferOtto, GiacomelliKnuepfer, GiacomelliKnuepferOtto, Gnann2015, Gnann2016, GIM2019} for zero contact angles and in \cite{Knuepfer, KnuepferMasmoudi2013, Knuepfer2015, KnuepferMasmoudi2015} for nonzero contact angles. Qualitative properties such as propagation rates or the occurrence of waiting times have been examined for example in \cite{HulshofShishkov, ThinViscous, WaitingTime,GruenOptimalRatePropagation,Grun2001b, LowerBounds, SupportPropagationThinFilm, UpperBoundsThinFilmWeakSlippage}. For numerical studies, we refer to \cite{Grun2000, ZhoBe,NumericsGruen} and the references therein.
\\
Stochastic thin-film equations have been derived by \cite{StoneMoro}  and \cite{MeckeRauscher}  to model the impact of thermal fluctuations on spreading or dewetting of very thin liquid films, see also \cite{Dur_n_Olivencia_2019} for numerical investigations. 
For a recent thermodynamically inspired approach to introduce fluctuations in thin-film flow using Gibbs measures related to surface-tension energies, we refer to \cite{GessOtto2023}.
\\
The analysis of $\eqref{eq:stfeSingPot}$ started with the work of \cite{FischerGruen2018} who established existence of $\mathbb{P}$-almost surely strictly positive solutions for $\Cstr+S=0$ and $n=2$ which corresponds to the assumption of a Navier-slip condition at the liquid-solid interface and conservative linear multiplicative noise.
In \cite{MetzgerGruen2022}, this result has been generalized to the two-dimensional setting, see also \cite{SauerbreyAgresti2024wellposedness} for recent results for arbitrary $n$ and the higher-dimensional case. 
To keep the presentation concise, we refer to \cite{DGGG} and \cite{Dareiotis2023solutions}  for an overview of the literature in the case of the stochastic thin-film equation without singular effective interface potential.
\\
Mathematically, this paper is inspired by the ideas of \cite{FischerGruen2018} for the case of $n = 2$ and linear multiplicative noise. 
To derive basic integral estimates, we discretize in space using linear finite elements and we obtain an energy-entropy estimate using It\^o's formula and an appropriate stopping time argument. 
The crucial new ingredient is the discretization of the porous-medium term $(u^{n-2}u_{x})_{x}$ which has to be consistent with the integration-by-parts formula
\begin{align}\label{eq:IntPartsCont}
	\intort (u^{n-2}u_{x})_{x} u_{xx} \dx = \intort u^{n-2} u_{xx}^{2} \dx - \frac{(n-2)(n-3)}{3} \intort u^{n-4} u_{x}^{4} \dx \, .
\end{align} 
For this, we discretize $-u^{n-2} u_{xx}$ and $-(n-2) u^{n-3} |u_{x}|^{2}$ separately (see \eqref{eq:definePorad} and \eqref{eq:definePorb}), however, at the expense that in the discrete setting mass is no longer conserved. It turns out that the new discretization ansatz generates positivity for discrete version of $-\int_\ort \left(u^{n-2}u_x\right)_xu^{-\alpha} \dx$ with positive parameters $\alpha$, too.
Therefore, it provides the nonnegativity properties one would expect from the continuous setting.
Concerning discrete masses, the deviation from the mass of initial data in the continuous setting turns out to be of order $O(h)T$ in expectation on the time interval $[0,T]$.

Let us give the outline of the paper.  
In Section~\ref{sec:discr}, we fix  details on the discretization and we formulate general assumptions on the interface potential, on initial data, noise, and parameters, see (H1)-(H5) below. 
Subsequently, in \eqref{eq:1} and \eqref{eq:2} we introduce the semi-discrete scheme to be used for approximation purposes. 
In particular, the definitions of stopping times to control energy and mass of discrete solutions are given there as well as our approach to  discretize  the porous-media-type terms, cf. \eqref{eq:definePorad} and \eqref{eq:definePorb}. 
\\
Solution concept and  main existence result are made precise in Section~\ref{sec:MainResult}.
\\
Section~\ref{sec:apriori} is mainly devoted to the derivation of the energy-entropy estimate.  Similarly as in \cite{FischerGruen2018}, we use a finite-element discretization which we prefer compared to other approaches not the least due to its perspective of practical numerics. 
As a preparation for the passage to the  limit $c_{F} \to 0$, we refine the Oscillation Lemma~4.1 of \cite{FischerGruen2018} to get optimal results on the ratio $\frac{\ui}{\uip}$ in the limit $h \to 0$.
Proposition~\ref{prop:integral1} states the discrete energy-entropy estimate. The strategy here is to use It\^o's formula in combination with a stopping time argument and a Gronwall argument. The discrete counterpart of
\eqref{eq:IntPartsCont} provides exactly the terms involving fourth powers of $u_x$ which are needed to absorb corresponding It\^o-terms\footnote{Note that these terms do not occur in the case of linear multiplicative noise -- corresponding to $n=2$.}.
Control of arbitrary moments is achieved by means of the Burkholder-Davis-Gundy inequality.  After the derivation of uniform H\"older-estimates for discrete solutions and improved regularity results on the pressure, we close this section with the aforementioned estimation of the deviation of discrete masses from the mass of initial data -- see 
Lemma~\ref{lem:EstDiscMass}. 
\\
We begin Section~\ref{sec:convergence} with convergence results for the scheme -- based on an application of  Jakubowski's theorem \cite{Jakubowski},  see Proposition~\ref{prop:conv1}. The remainder of this section is devoted to the limit passage in the deterministic terms and finally to the identification of the stochastic integral -- for the latter using 
methods presented and/or developed in \cite{BreitFeireislHofmanova,BrzezniakOndrejat,DebusscheHofmanovaVovelle,HofmanovaSeidler}. Finally, we provide the proof of the main result. 

In Section~\ref{sec:CalcPorMed}, we show that our approach of finite-element discretizations for $u^{n-2}u_{xx}$ and for $u^{n-3}u_x^2$ yields indeed the desired positivity results when tested with discrete versions of $-u_{xx}$ and of $u^{-\alpha}$, $\alpha>0$, respectively. The analysis in this section benefits a lot from the aforementioned estimates on oscillations of discrete solutions.   
\\
In the appendix, we provide some  technical results, including estimates for the It\^o-terms with respect to energy and entropy which occur in the derivation of the a-priori estimates, cf. Lemma~\ref{lem:appenHilfeSiebenA},  Lemma~\ref{lem:poroussquare}, and Lemma~\ref{lem:nineAnew}. As the regularization parameter $S$ only serves to control additional contributions in the It\^o-term related to the surface-tension energy, we invest some time to obtain optimal estimates on 
$$\sum_{\ell \in \mathbb{Z}} \lambda_{\ell}^{2} \int_{0}^{t\wedge T_{h}} \int_\ort \left|\partial_h^-((\sqmob(u^h)g_\ell)_x)\right|^2 \, dx \, ds $$
to get good lower bounds on $S$.
\\
\\
{\bf Notation.}
Throughout the paper, we use standard notation for Sobolev spaces and from stochastic analysis. The spatial domain $\ort$ is given by the interval $(0,L)$, and we abbreviate $I:=[0,T].$ The notation $a\wedge b$ stands for the minimum of $a$ and $b$, and $L_2(X,Y)$ denotes the set of Hilbert-Schmidt operators from $X$ to $Y$. For values of $\gamma_s, \gamma_t\in (0,1)$, $C^{\gamma_s,\gamma_t}(\bar\ort\times[0,T])$ denotes the space of continuous functions on $\bar\ort\times[0,T]$ which are H\"older-continuous with exponent $\gamma_s$ (respectively $\gamma_t$) with respect to space (respectively time). In particular, the exponent $\gamma$ will exclusively be used for H\"older properties related to the martingale solution for the stochastic thin-film equation. 
For a stopping time $T$, we write $\chi_T$ to denote the ($\omega$-dependent) characteristic function of the time interval $[0,T]$. 
The abbreviation $\overline{v}$ is used for the mean value of a function $v$ over a domain $\ort$. 
Further notation related to the discretization will be introduced in Section~\ref{sec:discr}.

\section{Preliminaries on the discretization}
\label{sec:discr}

In this section, we will introduce a semi-discrete scheme which will serve to obtain spatially discrete approximate solutions to the stochastic thin-film equation. Existence of those approximate solutions will be established in Section~\ref{sec:apriori} applying a stopping time argument to solutions of an appropriate system of ordinary stochastic differential equations.
\begin{itemize}
\item Given an integer fraction $h$ of a real number $L>0$, by $X_h$ we denote the space of periodic linear finite elements, i.e. the space of periodic continuous functions on $[0,L]$ that are linear on each of the intervals $[0,h]$, $[h,2h]$, $\dots$, $[L-h,L]$. 
\item By $e_i$, we denote the function in $X_h$ that equals $1$ at $x=ih$ and that vanishes for all other $x=kh$, $k\neq i$. 
\item Let $C_{per}([0,L])$ be the space of periodic continuous functions on $[0,L]$. By $\mathcal{I}_h:C_{per}([0,L])\rightarrow X_h$, we  denote the nodal interpolation operator uniquely defined by $(\Ih[\psi])(ih):=\psi(ih)$ for all $i\in\{1,\cdots, \Lh\}$ where $\Lh:=Lh^{-1}$ is the dimension of $X_h$.
\item On the Hilbert space $X_h$, we introduce the scalar product
\begin{align*}
(\phi^h,\psi^h)_h:=\sum_{i=1}^\Lh h\phi^h(ih)\psi^h(ih)
\end{align*}
and the corresponding norm
\begin{align*}
||\psi^h||_h:=\left(\sum_{i=1}^\Lh h |\phi^h(ih)|^2 \right)^{1/2}.
\end{align*}
Note that the norm $||\cdot ||_h$ is equivalent to the $L^2(\domain)$-norm on $X_h$, uniformly in $h$. With a slight misuse of notation, we will frequently abbreviate 
$(\Ih[\phi], \psi_h)_h$ for functions $\phi\in C_{per}([0,L])$ and $\psi_h\in X_h$ by $(\phi, \psi_h)_h$.
\item By $\partial_h^{+}$ and $\partial_h^{-}$, we denote the forward and backward difference quotient, respectively, i.e. $\partial_h^{+}f(x):=h^{-1}(f(x+h)-f(x))$ (with $f$ extended outside of $[0,L]$ by periodicity). 
\item
The discrete Laplacian $\Delta_h v^h$ of a function $v^h\in X_h$ is defined by
the variational formulation 
$$ (\Delta_h v^h, \psi^h)_h:=-\int_\ort v_x^h\cdot \psi_x^h dx \qquad \forall \psi^h\in X_h.$$
We note the identity $\Delta_h v^h=\partial_h^{+}(\partial_h^{-}u).$ 
\item Sometimes, we abbreviate $v_i:=v(ih)$ for functions $v\in C^0(\ort)$ and $i=1, \ldots, L_h.$
\end{itemize}
Now we are in the position to formulate the general assumptions on the data.

  \begin{itemize}
\item[(H1)] The mobility is given by $m(u)=u^n$ with $n \in (2,3)$.
\item[(H2)]
The effective interface potential $F(u)$ satisfies for $u > 0$, $c_{F} > 0 \,$, and some $p>n$
\begin{align*}
	F(u) = c_{F}u^{-p} \, .
\end{align*}
For non-positive $u$, we define $F(u):=+\infty.$
\item[(H3)] Let $\Lambda$ be a probability measure on $H^1_{per}(\ort)$ equipped with the Borel $\sigma$-algebra which is supported on the subset of strictly positive functions such that 
there is a positive constant $C$ with the property that
\begin{align*}
\essup_{v\in\supp\Lambda}\left\{E[\Ih[v]]+\left(\int_\ort v\,dx\right)+\left(\int_\ort v \,dx\right)^{-1}\right\}\leq C
\end{align*} 
for any $h>0$ with $E[u]:= \intort \frac{1}{2} |u_{x}|^{2} + F(u) \dx$. 
\item[(H4)] Let $(\Omega,\mathcal F, (\mathcal F_t)_{t\geq 0}, \probab)$ be a stochastic basis with a complete, right-continuous filtration such that
\begin{itemize}
\item $W$ is a $Q$-Wiener process on $\Omega$ adapted to $(\mathcal F_t)_{t\geq 0}$ which admits a decomposition of the form $W=\sum_{\ell \in \Z} \lambda_\ell g_\ell \beta_\ell$ for a sequence of independent standard Brownian motions $\beta_\ell$ and nonnegative real numbers $(\lambda_\ell)_{\ell\in\N}$.
Here $(g_{\ell})_{\ell \in \mathbb{Z}}$ are eigenfunctions of the Laplacian on $\ort$ subject to periodic boundary conditions,
	\begin{align}\label{def:L2-Basis}
		g_{\ell}(x) = 
		\begin{cases}
			\sqrt{\frac{2}{L}} \sin(\frac{2 \pi \ell x }{L})\quad &\ell >0,
			\\
			\frac{1}{\sqrt{L}}\quad &\ell=0, 
			\\
			\sqrt{\frac{2}{L}} \cos(\frac{2 \pi \ell x }{L})\quad &\ell<0.
		\end{cases}
	\end{align} 
\item the noise $W$ is colored in the sense that $\sum_{\ell\in \Z} \ell^4 \lambda_\ell^2 <\infty$, 
\item there exists an $\mathcal F_0$-measurable random variable $u_0$ such that $\Lambda=\probab\circ u_0^{-1}.$
\end{itemize}
\item[(H5)] The regularization parameter $S$ satisfies 
$$S > C_{Strat} \frac{3C_{osc}^{4-n}}{(1+C_{osc})^{n-4}}    \frac{n-2}{3-n} + C_{Strat} (C_{osc}^{n-2}-1)$$
with $C_{osc} = 1 + \sqrt{2c_{F}}$, cf. Lemma~\ref{lem:lowerbound}, and $C_{Strat} = \frac{1}{2} \frac{n^{2}}{4} \left(\frac{\lambda_{0}^{2}}{L} + \sum_{\ell=1}^{\infty} \frac{2\lambda_{\ell}^{2}}{L}\right)$.
\end{itemize}

\begin{remark}
The lower bound of $S$ tends to zero for $n \to 2$ and to infinity for $n \to 3$ which is in accordance with the behavior of the formally derived lower bound in the continuous setting $S_{opt} := C_{Strat} \frac{9}{4} \frac{(n-2)^{2}}{(3-n)(2n-3)}$ (cf. formula (3.33) in \cite{GruenKlein24Limit}). Moreover, for $c_{F} \to 0$, i.e. $C_{osc} \to 1$, we see that the second summand of the lower bound tends to zero.
Note that for values of $n$ sufficiently close to $2$, $S$ can be chosen equal to $\Cstr$. 
Hence, the case of a surface-tension driven stochastic thin-film equation with Hänggi-Klimontovich noise is covered.
\end{remark}

Let us define our scheme for approximation.
On a stochastic basis satisfying (H4), given a positive time $T_{max}$ and introducing $E_{max,h}:=\tfrac12  c_{F}h^{-(p-2)/(p+2)},$ we consider solutions
\begin{align*}
u^h &\in L^2(\Omega;C([0,T_{max}];X_h)),
\\
p^h &\in L^2(\Omega;L^{2}((0,T_{max});X_h))
\end{align*}
with $|| \essup_{t \in (0,T_{max})} ||p^{h}(t)||_{X_{h}}  ||_{L^{2}(\Omega)} < \infty$
to the system of stochastic differential equations
\begin{subequations}
\label{FaedoGalerkinScheme}
\begin{align}
\label{eq:1}
(u^h(T), \phi^h)_h
&=
(u_0, \phi^h)_h
-\int_0^{T\wedge T_h}\int_\domain M_h(u^h)p^h_x \phi^h_x \,dx\,dt
\\& \nonumber
\quad -(C_{Strat}+S)\int_0^{T\wedge T_h}\porad(u^h, \phi^h)dt
-(C_{Strat}+S)\int_0^{T\wedge T_h}\porb(u^h, \phi^h)dt
\\&
\nonumber
\quad -\sum_{|\ell| \le N_h} \int_0^{T\wedge T_h} \int_\domain \lambda_\ell \sqmob(u^h) g_\ell \phi^h_x \,dx\,d\beta_\ell
\qquad
\forall \phi^h\in X_h,
\\
\label{eq:2}
(p^h, \phi^h)_h &= \chi_{T_h} \int_\domain u^h_x \phi^h_xdx +\chi_{T_h} (F'(u^h), \phi^h)_h
\qquad
\forall \phi^h\in X_h \, .
\end{align}
\end{subequations}
Here $\chi_{T_h}=\chi_{T_h}(t)$ is an abbreviation for the characteristic function of the time interval $[0,T_h]$, where $T_h$ is the stopping time defined by 
$T_h:= T^{E}_{h} \wedge T^{M}_{h}$ with 
\begin{align*}
	T^{E}_{h} :=T_{max}\wedge\inf\{t\in [0,\infty):E_h[u^h(\cdot,t)]\geq E_{max,h}\}
\end{align*}
and
\begin{align*}
	T_{h}^{M} := T_{max}\wedge\inf\Big\{t\in [0,\infty):|\overline{\uh}(\cdot,t)-\overline{\uh}(\cdot,0)|\geq \frac{\overline{\uh}(\cdot,0)}{2}\Big\} \, ,
\end{align*}  
for $\overline{\uh}(t) = \overline{\uh}(\omega,t) := \fint_{\ort} \uh(\omega,t,x) \, dx$ for $(\omega,t) \in \Omega \times [0,T_{max}]$.
Moreover, $N_h\in \mathbb{N}$ is a cutoff for the noise for the purpose of discretization, subject only to the condition $N_h\rightarrow \infty$ for $h\rightarrow 0$. Furthermore, $M_h(u^h)$ and $\sqmob(u^h)$ are suitable modifications of the pointwise mobility $m(u^h)$ and of its root $\sqrt{m(u^h)}$, respectively, (see below).

The operators 
\begin{align}
\label{eq:definePorad}
\begin{split}
 \porad(u^h, v^h):=& -\tfrac{n-2}{6}\discint (\ui)^{n-3}\left(\left|\tfrac{\ui-\uim}{h}\right|^2+\left|\tfrac{\uip-\ui}{h}\right|^2\right)v_i^{h}\\
&-\tfrac{n-2}{6}\discint\bigg\{\left((\ui)^{n-3}+(\uip)^{n-3}\right)\tfrac{\uip-\ui}{h}
\\ & \qquad \qquad \qquad +
\left((\uim)^{n-3}+(\ui)^{n-3}\right)\tfrac{\ui-\uim}{h}\bigg\}
\tfrac{\uip-\uim}{2h} v^{h}_i
\end{split}
\end{align}
and
\begin{align}
\label{eq:definePorb}
\begin{split}
	\porb(u^h, v^h):=& -\left(\Ih[(u^h)^{n-2}]\Delta_h u^h, v^h\right)_h\\
	=& -\discint (\ui)^{n-2}\left(\tfrac{\uip-2 \ui+\uim}{h^2}\right)v_i^{h}
\end{split}
\end{align}
are discrete weak versions of $-(n-2)u^{n-3}|u_x|^2$ and $-u^{n-2}u_{xx}$, respectively. Their sum is a discrete weak version of $-(u^{n-2}u_x)_x$. 
  
Discrete initial data are computed by the  formula $u_0^h(\omega):=\Ih[u_0(\omega)].$ 
The following estimate can be established in a standard way.
\begin{gather}
\label{mean2}
\left|\overline{u_0^h}(\omega)-\overline{u_0}(\omega)\right|\leq C h \left(\int_\ort|(u_0)_x(\omega)|^2\right)^{1/2} \, ,
\end{gather}
where $\overline{v}$ denotes the mean value of a function $v$ over $\ort$.

The discrete mobility $M_h(u^h)$ is defined as follows: Choose $\sigma:=\tfrac12h^{2/(p+2)}$ and consider the shifted mobility $m_\sigma(u):=m(\max(\sigma,u)).$ Then, for an element $v^h\in X_h$, the discrete mobility $M_h(v^h)$ is given as the elementwise constant function defined by 
\begin{equation}
\label{eq:defmobility}
M_h(v^h)|_{(ih,(i+1)h]}:=\begin{cases}
m_\sigma(v_i^h) & \text{ if } v^h_i=v^h_{i+1}\\
\left(\fint_{v^h_i}^{v^h_{i+1}}\tfrac1{m_\sigma(s)}ds\right)^{-1} & \text{ if } v^h_i\neq v^h_{i+1}.
\end{cases}
\end{equation}
Its square-root $\sqrt{m(u)}$ is replaced on the discrete level by
\begin{equation}
  \label{eq:mobilityroot}
  \sqmob(u^h):= \Ih[|u^h|^{n/2}].
  \end{equation}
We note that by Lemma~\ref{lem:lowerbound} and our choice of $\sigma$, for $\uh$ satisfying \eqref{eq:assump} the discrete mobility $M_{h}(u^h)\left|_{(ih, (i+1)h)}\right.$  is given by
$(n-1)\tfrac{\uip -\ui}{(\ui)^{1-n}-(\uip)^{1-n}}$ whenever $\uip \neq \ui$.
Related to the discrete mobility $M_h(\cdot),$ we introduce the nonnegative discrete entropy density \begin{equation}
\label{eq:entropydef}
G_h(s):=\int_1^s\int_1^\mu \frac{1}{m_\sigma(\tau)}\,d\tau\,d\mu
\end{equation} 
and similarly $g_h(s):=G_h'(s)$.  For further reference, we note the identity
\begin{equation}
\label{eq:entropconsist}
\int_\ort M_h(u^h)p_x^h\partial_x\left(\Ih[g_h(u^h)]\right) \,dx
= \int_\ort p_x^h u_x^h \,dx
\end{equation}
which is commonly referred to as {\it entropy consistency of the discrete mobility}, cf. \cite{Grun2000} and \cite{ZhoBe}. 

\section{Solution concept and main result}
\label{sec:MainResult}
\begin{definition}
\label{DefinitionMartingaleSolution}
Let $\Lambda$ be a probability measure on $H^1_{per}(\ort)$. We will call a triple
\newline
 $(\Basislimit, \tu, \tilde W)$ a weak martingale solution to the stochastic thin-film equation \eqref{eq:stfeSingPot} with initial data $\Lambda$ on the time interval $[0,T]$ provided
\begin{itemize}
\item[i)] $\Basislimit$ is a stochastic basis with a complete, right-continuous filtration, 
\item[ii)] $\tilde W$ satisfies (H4) with respect to $\Basislimit$, 
\item[iii)]
the continuous $L^{2}(\cO)$-valued, $\ti{\mathcal{F}}_{t}$-adapted process
$\tu\in L^2(\tilde \Omega;L^2((0,T); H^2_{per}(\ort)))\cap L^2(\tilde\Omega; C^{\gamma, \gamma/4}(\bar{\ort}\times[0,T]))$ with $\gamma\in (0,1/2)$ is 
positive $\tprobab$-almost surely
and $\tu_{xxx} \in L^{2}([0,T] \times \ort)$ $\tprobab$-almost surely, \color{black}
\item[iv)] there exists an $\tilde{\mathcal F_0}$-measurable $H^1_{per}(\ort)$-valued random variable $\tu_0$ such that $\Lambda=\tprobab\circ \tu_0^{-1}$, and the equation
\begin{align}
\label{eq:weakformula}
\begin{split}
\int_\domain \tu(t) \phi\,dx =&\int_\domain \tu_0 \phi\,dx
+\int_0^t \int_\domain m(\tu) (\tu_{xx}-F'(\tu))_x \phi_x \, dx \,ds
\\&
- \int_{0}^{t} \intort (\tu^{n-2}\tu_{x}) \phi_{x} \, dx \, ds
-\sum_{\ell\in \Z} \lambda_\ell\int_0^t \int_\domain \sqrt{m(\tu)}g_\ell \phi_x \,dx \,d\beta_\ell
\end{split}
\end{align}
holds $\tprobab$-almost surely for all $t\in [0,T]$ and all $\phi\in H^1_{per}(\domain)$.
\end{itemize}
\end{definition}

We are going to establish the existence of a weak martingale solution via approximation by solutions to the semi-discrete scheme \eqref{FaedoGalerkinScheme}.
\begin{theorem}
\label{MainResult}
Let the assumptions (H1)-(H5) be satisfied and let $T_{max}>0$ be given. 
Assume $u^h$, $p^h$ (where $h\rightarrow 0$) to be a sequence of solutions to the Faedo-Galerkin scheme \eqref{FaedoGalerkinScheme} for the stochastic thin-film equation \eqref{eq:stfeSingPot} with $E_{max,h}=\tfrac{1}{2} c_{F} h^{-(p-2)/(p+2)}$. Let $0<\gamma<1/2$ be given.

Then there exist a stochastic basis $\Basislimit$ as well as processes ${\tilde u}^h$, ${\tilde p}^h$, and $\tu\in L^2(\tilde\Omega;L^2([0,T];H^1_{per}(\mathcal{O})))$ such that the following holds:
The processes $\tilde u^h$, $\tilde p^h$ have the same law as the processes $u^h$, $p^h$ and for a subsequence we $\tprobab$-almost surely have the convergence ${\tilde u}^h \rightarrow \tu$ strongly in $C^{\gamma,\gamma/4}(\domain\times [0,T_{max}])$ and $\sqrt{M_h(\tuh)}\tph_x\rightharpoonup -\tu^{\frac{n}{2}}\left(\tu_{xx}-F'(\tu)\right)_x$ weakly in $L^2(\ort\times[0,T_{max}]).$ 
Furthermore, $\tu$ is a weak martingale solution to the stochastic thin-film equation in the sense of Definition~\ref{DefinitionMartingaleSolution} satisfying the additional bound
\begin{align}\label{mainresultest}
	\nonumber
	&\mathbb{E}\Bigg[ \sup_{t\in [0,T_{max}]} E[\tu(t)]^{\bar p}
	+ \left(\int_0^{T_{max}}  \int_\domain \tu^{n} | \tp_x|^2\,dx \,ds \right)^{\bar{p}} 
	\\& \nonumber
	+\left(\int_0^{T_{max}}  \intort |\Delta\tu|^2 \, dx \,ds\right)^{\bar{p}} 
	+\left(\int_0^{T_{max}}  \intort \tu^{-p-2}  |\tu_x|^2  \,dx \,ds\right)^{\bar{p}} 
	\\& \nonumber
	+\left(\int_0^{T_{max}} \intort \tu^{n-4}\left|\tu_{x}\right|^4\, dx \,ds\right)^{\bar{p}}
	+\left(\int_0^{T_{max}} \intort \tu^{n-2}\left|\Delta \tu\right|^2\, dx \,ds\right)^{\bar{p}}
	\\&\nonumber
	+\left(\int_0^{T_{max}} \intort \tu^{n-p-4}\left|\tu_{x}\right|^2\, dx \,ds\right)^{\bar{p}}
	+\left(\int_0^{T_{max}} \intort \tu^{-2}\left|\tu_{x}\right|^2\, dx \,ds\right)^{\bar{p}}\Bigg]
	\\&
	\le C(\bar{p}, \tu_{0}, T_{max}) \, .
\end{align}
for any $\bar p\geq 1$. In particular, $\tu$ is positive $\tprobab$-almost surely.
\end{theorem}

\section{Analysis of the discrete scheme}
\label{sec:apriori}
\subsection{Stopping times, oscillation lemma, and existence of discrete solutions}

To establish a discrete counterpart of the combined energy-entropy estimate, we introduce discretization-adapted variants of the energy and  of the entropy. 
We set
\begin{align}
\label{DefinitionEh}
E_h[v]:=\int_\domain \frac{1}{2} |v_x|^2 + \mathcal{I}_h[F(v)] \,dx
\end{align}
and
\begin{align}
S_h[v]:=\int_\domain \mathcal{I}_h[G_h(v)] \,dx. \color{blue}
\end{align}
A bound on the energy $E_h[v]$ entails a lower bound on the thickness of the thin film as well as on its oscillation.
\begin{lemma}
\label{lem:lowerbound}
Let $u^h\in X_h$ be strictly positive. We then have the estimate
\begin{align}
\label{eq:lowerboundbyenergy}
\sup_{x\in \domain}~ (u^h)^{-1}
\leq C \left(\fint_\domain u^h \,dx\right)^{-1} +C E_h[u^h]^{2/(p-2)}.
\end{align}
If in addition the bound
\begin{equation}
  \label{eq:assump}E_h[u^h]\leq c_{F}\, h^{-\tfrac{p-2}{p+2}}
\end{equation}
holds, then we have the estimates
\begin{equation}
\label{eq:lowerbound}
\min_{x\in\ort} u^h\geq h^{\tfrac2{p+2}}
\end{equation}
and
\begin{equation}
\label{eq:lowerboundmean}
\fint_\domain u^h \,dx \geq h^{\tfrac{2}{p+2}} \, .
\end{equation}
Moreover, we have
\begin{equation}
\label{eq:oscillation}
\max\left\{
\frac{u^h(ih)}{u^h((i+1)h)}, \frac{u^h(ih)}{u^h((i-1)h)}\right\}\leq
 1 + \sqrt{2c_{F}} =: C_{osc}
\end{equation}
for all $i=1,\cdots, \Lh$.
\end{lemma}
\begin{proof}
To establish \eqref{eq:lowerboundbyenergy}, we estimate
\begin{align*}
\int_\domain |((u^h)^{-p/2+1})_x| \,dx
&=C(p)\int_\domain (u^h)^{-p/2} |u^h_x| \,dx
\\&
\leq 2C(p)\left(h\sum_{i=1}^{L_h} (u^h(ih))^{-p} \right)^{1/2} E_h[u^h]^{1/2}
\leq C(p) E_h[u^h]
\end{align*}
(where in the last step we have used Hypothesis (H2)), which implies
\begin{align*}
\sup_{x\in \domain} ~ (u^h)^{-1}
&\leq C\inf_{x\in \domain} ~(u^h)^{-1} + C E_h[u^h]^{2/(p-2)}
\\&
\leq C \left(\fint_\domain u^h \,dx\right)^{-1} + C E_h[u^h]^{2/(p-2)}.
\end{align*} 
Having established \eqref{eq:lowerboundbyenergy}, we may prove \eqref{eq:lowerbound} as follows.
If \eqref{eq:lowerbound} were not true, then there would exist $i_{0} \in \{1,\dots, L_{h}\}$ with $u_{i_{0}}^{h} < h^{\frac{2}{2+p}}$. Then
\begin{align*}
	E_{h}[u^{h}] = \frac{1}{2} \intort |u^{h}_{x}|^{2}\, dx + \discint c_{F} (\ui)^{-p} > c_{F} h (u_{i_{0}}^{h})^{-p} >  c_{F} hh^{-\frac{2p}{2+p}} = c_{F} h^{-\frac{p-2}{2+p}} \, 
\end{align*}  
which would contradict \eqref{eq:assump}.
\eqref{eq:lowerboundmean} easily follows from \eqref{eq:lowerbound}.
To prove \eqref{eq:oscillation}, w.l.o.g. it is sufficient to prove the inequality for the first term. 
For $x,y \in \ort$ we have by Hölder's inequality
\begin{align*}
	|\uh(x) - \uh(y)| = \left|\int_{y}^{x} \uh_{x}(s) \, ds\right| \le |x-y|^{\frac{1}{2}} \left(\int_{\ort} |\uh_{x}|^{2}\right)^{\frac{1}{2}} \, dx \, .
\end{align*}
Thus, we infer
\begin{align*}
	\sup_{x,y \in \ort } \frac{|\uh(x) - \uh(y)|}{|x-y|^{\frac{1}{2}}} \le \sqrt{2 E_{h}[\uh]} \, .
\end{align*}
Hence, 
\begin{align*}
	\left|\frac{u^h(ih)}{u^h((i+1)h)}-1\right|
	=\frac{|u^h(ih)-u^h((i+1)h)|}{u^h((i+1)h)}
	\leq  \frac{h^{1/2} \sqrt{2E_h[u^h]}}{h^{\tfrac{2}{2+p}}}
	\leq \sqrt{2c_{F}}.
\end{align*}
\color{black}
\end{proof}

\begin{remark}
	The positive solutions we aim to construct in Theorem~\ref{MainResult} will be used in \cite{GruenKlein24Limit} as approximations to establish the existence of nonnegative solutions in the limit $c_{F} \to 0$. Thus, we do not discuss an estimate like \eqref{eq:oscillation} with an $h$-dependent upper bound and its behavior for $h \to 0$.  
\end{remark}

Adapting the results of Lemma~4.3 of \cite{FischerGruen2018}, we get an existence result for discrete solutions. 

\begin{lemma}
	\label{ExistenceDiscrete}
	Let $T_{max}$ be  a positive real number and $E_{max,h}=\tfrac12 c_{F}  h^{-(p-2)/(p+2)}$. Then there exist a stopping time $T_{h}$ and stochastic processes $u^h\in L^2(\Omega;C([0,T_{max}];X_h))$, $p^h\in L^2(\Omega;L^{\infty}((0,T_{max});X_h))$ with the following properties:
	\begin{itemize}
		\item Almost surely, we have $T_{h}=T^{E}_{h} \wedge T^{M}_{h}$.
		\item Almost surely, the process $p^h$ solves \eqref{eq:2} for $t\leq T_{max}$ and is contained in $C([0,T_{h}];X_h)$.
		\item Almost surely, the process $u^h$ solves \eqref{eq:1} for $t\leq T_{h}$ and is constant for $t\in [T_{h},T_{max}]$ (and thus solves \eqref{eq:1} for $t\leq T_{max}$).
		\item 
		We have $\overline{\uh}(t) \in \left[\frac{\overline{\uh}(0)}{2}, \frac{3 \overline{\uh}(0)}{2}\right]$ for $t \le T_{h}$. 
	\end{itemize}
\end{lemma}

We note that by \eqref{mean2} and the bounds in $(H3)$ almost surely and for $t \in [0,T_{h}]$ the mean value of a solution $\uh$ satisfies  $\overline{\uh}(\omega,t) \in [c,C]\, ,$ 
where $0<c<C$ are constants independent of $h$, $\omega$ and $t$. As $\uh$ is constant on $(T_{h}, T_{max})$, the same holds true for all $t \in [0,T_{max}]$.

\subsection{Energy and entropy estimates}

We now demonstrate that our spatial semi-discretization preserves the combined energy-entropy estimate as long as the energy remains below a critical threshold and the  mass stays near the initial mass.
As before, we choose $E_{max,h}=\tfrac12 c_{F} h^{-(p-2)/(p+2)}$ to be the threshold energy. In particular, it becomes infinite in the limit $h\to 0.$

Writing $u^h(x,t)=\sum_{i=1}^\Lh a_i(t)e_i(x)$, we first note that \eqref{eq:1} may be rewritten as
\begin{align}
\label{eq:integral3}
da_i=\frac{1}{h} L_i(s) \,ds + \frac{1}{h} \Li_i(s)ds +  \sum_{|\ell|\le N_h} Z_i( \lambda_\ell g_\ell) d\beta_\ell,
\end{align}
where we have introduced
\begin{align}
\label{eq:integral4}
L_i(s):=-\chi_{T_h}(s)\int_\domain M_h(u^h(s)) p^h_x(s) (e_i^h)_x \,dx, 
\end{align}
\begin{equation}
\label{eq:defLi}
\Li_i(s):=-\chi_{T_h}(s) (C_{Strat}+S) \left\{\porad(u^h(s), e_i)+\porb(u^h(s), e_i)\right\},
\end{equation}
and $Z_i:L^2(\domain)\rightarrow \mathbb{R}$ defined by
\begin{align}
\label{eq:integral5}
Z_i(w):=\chi_{T_h}\frac{1}{h}\sum_{\ell \in \Z} \int_\domain \big( (g_\ell,w)_{L^2(\domain)} \sqmob(u^h) g_\ell \big)_x e_i^h \,dx \, .
\end{align}
For a given positive parameter $\kappa$, we consider the integral quantity
\begin{align}
\label{eq:integral13}
R(\kappa,h,s):=E_h[u^h(s)]+\kappa S_h[u^h(s)]\, .
\end{align}
Using It\^o's formula, we derive the following integral estimates.
\begin{proposition}
\label{prop:integral1}
Let $\bar p\geq 1$ be arbitrary and let $u^h, p^h$ be a solution to
 \eqref{eq:1} and \eqref{eq:2} for a parameter $0<h<1$.
Then, for sufficiently large 
$\kappa$ depending only on $(\lambda_\ell)_\ell$ and on $\bar p$, there exist a positive constant $C_1$ independent of $h$ and independent of initial data as well as a positive constant $\bar{C}$ depending only on $\bar p$, $(\lambda_\ell)_{\ell\in\N}$, $T_{max}$, and initial data such that for all $t\in [0,T_{max}]$ the following inequality holds:
\begin{align}\label{eq:proposition1}
	\nonumber
	&\mathbb{E}[\sup_{t\in [0,T_{max}]} R(t\wedge T_h)^{\bar p}]
	+ C_1
	\mathbb{E}\Bigg[\bigg(\int_0^{t\wedge T_h}  \int_\domain M_h(u^h(s)) | p_x^h(s)|^2\,dx \,ds \bigg)^{\bar{p}} \Bigg]
	\\& \nonumber
	+C_1 \kappa
	\mathbb{E}\Bigg[\bigg(\int_0^{t\wedge T_h}  \sum_{i=1}^\Lh \dashint_{u^h_{i-1}}^{u^h_i} |\tau|^{-p-2} \,d\tau \int_{(i-1)h}^{ih} |u^h_x(s)|^2 \,dx \,ds\bigg)^{\bar{p}} \Bigg]
	\\& \nonumber
	+C_1 C(S)\expect\Bigg[\bigg(\tThInt \discint(\ui)^{n-4}\left|\tfrac{\uip-\ui}{h}\right|^4(s)ds\bigg)^{\bar{p}}\Bigg]
	\\& \nonumber
	+C_1 \kappa
	\mathbb{E}\Bigg[\bigg(\int_0^{t\wedge T_h}  ||\Delta_hu^h(s)||_h^2  \,ds\bigg)^{\bar{p}} \Bigg]
	+C_1S\expect\Bigg[\bigg(\tThInt \discint(\ui)^{n-2}\left|(\Delta_h u^h)_i\right|^2(s)ds\bigg)^{\bar{p}}\Bigg]
	\\&\nonumber
	+C_1 S\expect\Bigg[\bigg(\tThInt \discint (\ui)^{n-p-4}\left|\tfrac{\uip-\ui}{h}\right|^2(s)ds\bigg)^{\bar{p}}\Bigg]
	\\&
	+C_1\kappa S\expect\Bigg[\bigg(\tThInt \discint (\ui)^{-2}\left|\tfrac{\uip-\ui}{h}\right|^2(s)ds\bigg)^{\bar{p}} \Bigg]
	\le
	\bar{C}(\bar{p},(\lambda_{\ell})_{\ell \in \Z}, u_{0}, T_{max}) \, .
\end{align}
\end{proposition}
\begin{proof}
Using the notation 
\begin{equation}
\label{eq:integral6}
\varphi(h,s):=\frac{1}{h}\sum_{i=1}^\Lh L_i(s) e_i, 
\end{equation}
\begin{equation}
\label{eq:integral601}
\psi(h,s):=\frac{1}{h}\sum_{i=1}^{L_h}\Li_i(s)e_i, 
\end{equation}
and
\begin{equation}
\label{eq:integral7}
\Phi(h,s)(w):=\sum_{i=1}^\Lh Z_i(w)e_i,
\end{equation}
we may rewrite \eqref{FaedoGalerkinScheme} as 
\begin{equation}
\label{eq:integral8}
du^h=\varphi(h,s)ds + \psi(h,s)ds+ \Phi(h,s)(dW^h_Q)
\end{equation}
with 
\begin{equation}
\label{eq:integral8a}
W_Q^h:=\sum_{|\ell|\le N_h} \lambda_\ell g_\ell \beta_\ell.
\end{equation}
By It\^o's formula, we deduce
\begin{align}
\nonumber
R(t\wedge T_h)
&= R(0)+ \kappa \int_0^{t\wedge T_h}  \frac{1}{h}\Big(g_h(u^h),\sum_{i=1}^\Lh L_i(s)e_i \Big)_h \,ds
\\&\nonumber
\quad+ \kappa \sum_{|\ell|\le N_h} \int_0^{t\wedge T_h}  \Big(g_h(u^h),\sum_{i=1}^\Lh Z_i(\lambda_\ell g_\ell)e_i \Big)_h \,d\beta_\ell
\\&\nonumber
\quad+  \int_0^{t\wedge T_h} \frac{1}{h}\Big(F'(u^h),\sum_{i=1}^\Lh L_i(s)e_i \Big)_h \,ds
\\&\nonumber
\quad+  \sum_{|\ell|\le N_h} \int_0^{t\wedge T_h} \Big(F'(u^h),\sum_{i=1}^\Lh Z_i(\lambda_\ell g_\ell)e_i \Big)_h \,d\beta_\ell
\\&\nonumber
\quad+ \int_0^{t\wedge T_h} \frac{1}{h} \int_\domain u^h_x \sum_{i=1}^\Lh L_i(s) (e_i)_x \,dx \,ds
\\&\nonumber
\quad+ \sum_{|\ell|\le N_h} \int_0^{t\wedge T_h} \int_\domain u^h_x \sum_{i=1}^\Lh Z_i(\lambda_\ell g_\ell) (e_i)_x \,dx \,d\beta_\ell
\\&\nonumber
\quad+\frac{\kappa}{2} \sum_{|\ell|\le N_h} \int_0^{t\wedge T_h}  \bigg(\frac{1}{m_\sigma(u^h)},\Big(\sum_{i=1}^\Lh Z_i(\lambda_\ell g_\ell)e_i\Big)^2\bigg)_h \,ds
\\&\nonumber
\quad+\frac{1}{2} \sum_{|\ell|\le N_h} \int_0^{t\wedge T_h} 
\bigg(F''(u^h),\Big(\sum_{i=1}^\Lh Z_i(\lambda_\ell g_\ell)e_i\Big)^2\bigg)_h \,ds
\\&\nonumber
\quad+\frac{1}{2} \sum_{|\ell|\le N_h} \int_0^{t\wedge T_h} 
\int_\domain \Big|\sum_{i=1}^\Lh Z_i(\lambda_\ell g_\ell) (e_i)_x\Big|^2 \,dx \,ds
\\&\nonumber
\quad+ \kappa\tThInt\frac1h\bigg(g_h(u^h), \sum_{i=1}^{L_h}\Li_i(s)e_i\bigg)_h ds\\ &\nonumber
\quad+\tThInt \frac1h\bigg(F'(u^h), \sum_{i=1}^{L_h}\Li_i(s) e_i\bigg)_h ds
\\&\nonumber
\quad + \tThInt\frac1h\int_\ort u_x ^h\sum_{i=1}^{L_h}\Li_i(s) (e_i)_x \, dx \, ds
\\ \label{eq:prop2}
&\quad =:R(0)+I_1+\ldots+I_{12}.
\end{align}

In the sequel, we will frequently take advantage of the identities
\begin{align}
\label{eq:prop3}
\Big(w,\sum_{i=1}^\Lh Z_i(g_\ell) e_i\Big)_h
=\chi_{T_h}\int_\domain (\sqmob(u^h) g_\ell)_x \mathcal{I}_h[w] \,dx\, ,
\end{align}
\begin{equation}
\label{eq:prop333}
\frac{1}{h}\Big(w,\sum_{i=1}^{L_h}\Li_i(s)e_i\Big)_h
=
-\chi_{T_h}(s) (C_{Strat}+S)\left\{\porad(u^h, \Ih [w])+\porb(u^h, \Ih [w])\right\} \, ,
\end{equation}
and
\begin{align}
\label{eq:prop4}
\frac{1}{h}\Big(w,\sum_{i=1}^\Lh L_i(s) e_i\Big)_h
=
-\chi_{T_h}(s)  \int_\domain M_h(u^h) p^h_x \mathcal{I}_h[w]_x \,dx
\end{align}
for all $w\in C^0_{per}(\ort)$, which can easily be deduced from \eqref{eq:integral4}, \eqref{eq:integral5} together with the fact that the $h^{-1/2}e_i$ form an orthonormal basis with respect to the scalar product $(\cdot,\cdot)_h$. 

For $I_{12}$, we compute
\begin{equation}
\label{eq:neupme231}
\begin{split}
\frac1h&\int_\ort u_x^h\sum_{i=1}^{L_h} \Li_i(s)(e_i)_x \, dx 
\\ &= 
 \chi_{T_h}(s) (C_{Strat}+S)
 \left\{\porad(u^h(s), \Delta_h u^h(s)) 
 + \porb(u^h(s), \Delta_h u^h(s)) \right\}.
\end{split}
\end{equation}
Indeed, observing
$$
\sum_{j=1}^{L_h} \Li_j(s) (e_j)_x\bigg|_{(ih, (i+1)h)}=\tfrac{\Li_{i+1}(s)-\Li_i(s)}{h}, 
$$
we find
$$\int_\ort u_x^h\sum_{j=1}^{L_h}\Li_j(s) \left(e_j\right)_xdx=\discint \tfrac{u_{i+1}^h-u_i^h}{h} \, \tfrac{\Li_{i+1}(s)-\Li_i(s)}{h}$$
which gives the result using discrete integration by parts and the definition of $\Li_i$.

Therefore, 
$$ I_{12}=  (C_{Strat}+S)\tThInt \left\{\porad(u^h(s), \Delta_h u^h(s))+\porb(u^h(s), \Delta_h u^h(s))\right\}ds.$$
Using Lemmas~\ref{lem:discgradlapl} and \ref{lem:disclaplap},  we find
\begin{align}
	\label{eq:discPmeSign}
	\begin{split}
		I_{12} 
		\leq 
		& - (C_{Strat}+S)\tfrac{|(n-2)(n-3)|}{3} \tfrac{(1+C_{osc})^{n-4}}{2} 
		\\
		&\qquad \cdot \tThInt \discint  (\ui)^{n-4} \left( \left|\tfrac{\uip -\ui}{h} \right|^{2} \left|\tfrac{\ui -\uim}{h} \right|^{2} + \left|\tfrac{\uip -\ui}{h} \right|^{4} \right) \, ds
		\\
		& - (C_{Strat}+S) \tThInt \discint (\ui)^{n-2}\left|(\Delta_h u^h)_i\right|^2 \, ds
	\end{split}
\end{align}
which  both are terms with a good sign.

Using \eqref{eq:prop333}, we may rewrite 
\begin{equation} I_{11}=- (C_{Strat}+S) \tThInt \left\{\porad(u^h(s), \Ih[F'(u^h(s))])+\porb(u^h(s), \Ih[F'(u^h(s))])\right\}ds.
\end{equation}
By Lemma~\ref{lem:discsingular}, we infer 
\begin{equation}
\label{eq:est1I12}
-\porad(u^h(s), \Ih[F'(u^h(s))])\leq -c_p(n-2)\discint(\ui)^{n-p-4}\left|\tfrac{\uip-\ui}{h}\right|^2(s).
\end{equation}
By Lemma~\ref{lem:discLapPow}, we get
$$
-\porb(u^h(s), -\Ih[(u^h(s))^{-p-1}])\leq -C\discint (\ui)^{n-p-4}\left|\tfrac{\uip-\ui}{h}\right|^2(s).$$
Hence, 
\begin{equation}
\label{eq:est2I12}
I_{11}\leq -C \tThInt \discint (\ui)^{n-p-4}\left|\tfrac{\uip-\ui}{h}\right|^2(s)\,ds, 
\end{equation}
which again is a good term.
In a similar spirit, we identify 
$$ I_{10}=- (C_{Strat}+S) \tThInt \left\{\porad(u^h, \Ih [g_h(u^h)])+\porb(u^h, \Ih [g_h(u^h)])\right\}(s) \,ds, $$
and we estimate using Lemma~\ref{lem:lemma68}
\begin{align}
\label{eq:est1I11}
\begin{split}
& -\porad(u^h(s), \Ih [g_h(u^h(s))])\\
& \qquad\qquad\leq c_{ent}\bigg(-\discint (\ui)^{-2}\left|\tfrac{\uip-\ui}{h}\right|^2+\varepsilon \discint (\ui)^{n-4}\left|\tfrac{\uip-\ui}{h}\right|^4 \\
&\qquad\qquad\qquad + h^2\sum_{i=1}^{L_h}(\ui)^{n-4}\left|\tfrac{\uip-\ui}{h}\right|^4+h^2\sum_{i=1}^{L_h}(\ui)^{n-4}\left|\tfrac{\uip-\ui}{h}\right|^2+C_{\varepsilon}\bigg)
\end{split}
\end{align}
for arbitrary $\varepsilon>0.$ Note that the first term on the right-hand side has a good sign while the other ones are either constant or may be -- for $\varepsilon$ and $h$ sufficiently small -- absorbed. 
Combining Lemma~\ref{lem:discLapPow} and Remark~\ref{lem:remark}, we find that 
$$ - \tThInt\porb(u^h(s), \Ih [g_h(u^h(s))])\, ds$$
produces terms which can be either absorbed or are of Gronwall-type.

We continue with estimates for the terms $I_1$ to $I_{9}.$ By \eqref{eq:prop4}, we infer
\begin{align*}
\frac{1}{h} \Big(g_h(u^h), \sum_{i=1}^\Lh L_i(s)e_i\Big)_h
=-\int_\domain M_h(u^h) p^h_x \mathcal{I}_h[g_h(u^h)]_x \,dx.
\end{align*}
Noting $I_{1}$ to be formally identical with the term $I_{1}$ in \cite{FischerGruen2018}, Proposition~4.4 for $\bar{p}=1$, we find
\begin{align}
	  \label{eq:prop5}
	  \begin{split}
	I_1
	&\leq - \kappa \int_0^{t\wedge T_h} ||\Delta_h u^h||_h^2 \,ds
	\\& ~~~
	-\kappa c_1 \int_0^{t\wedge T_h}  \sum_{i=1}^\Lh \dashint_{u^h_i}^{u^h_{i+1}} |s|^{-p-2}\,ds \int_{ih}^{(i+1)h} |u^h_x|^2 \,dx \,ds
	=:-I_{1a}+I_{1b}
	  \end{split}
\end{align}
Obviously, $I_{1a}$ and $I_{1b}$ have a good sign and may be used for absorption purposes.
Similarly, comparing with $I_{3}$ and $I_{5}$ in \cite{FischerGruen2018}, we obtain
\begin{align}
	\label{eq:prop6}
	I_3+I_5
	=-\int_0^{t\wedge T_h} \int_\domain M_h(u^h) |p^h_x|^2 \,dx \,ds.
\end{align}
\newline
Ad $I_7$: By Lemma~\ref{lem:aux1}, we find
\begin{align*}
	&\frac{\kappa}{2} \sum_{|\ell|\le N_h} \int_0^{t\wedge T_h}  \Big(\frac{1}{m_\sigma(u^h)},\Big(\sum_{i=1}^\Lh Z_i(\lambda_\ell g_\ell) e_i\Big)^2\Big)_h \,ds
	\\&
	\leq \kappa \sum_{|\ell|\le N_h} \lambda_\ell^2
	\int_0^{t\wedge T_h}
	\sum_{i=1}^\Lh \frac{1}{m_\sigma(u^h_i)}
	\int_{(i-1)h}^{ih} \Big|\frac{\sqmob(u^h)(x+h)-\sqmob(u^h)(x)}{h}\Big|^2 |g_\ell(x)|^2 \,dx\, ds\\
	& \quad + \kappa \sum_{|\ell|\le N_h} \lambda_\ell^2
	\int_0^{t\wedge T_h} \sum_{i=1}^\Lh \frac{1}{m_\sigma(u^h_i)} \int_{(i-1)h}^{ih} \left(\sqmob(u^h)(x+h)\right)^2 \Big|\frac{g_\ell(x+h)-g_\ell(x)}{h}\Big|^2 \,dx
	 \,ds
	\\
	&=: I_{7a} + I_{7b}.
\end{align*}
To discuss $I_{7a}$, we note that  convexity implies 
\begin{align*}
 & \int_{(i-1)h}^{ih} \Big|\frac{\sqmob(u^h)(x+h)-\sqmob(u^h)(x)}{h}\Big|^2 \,dx
 \\
 &\leq \frac1h\int_0^h\frac{x}{h}\left|\sqmob(u^h)(ih)-\sqmob(u^h)((i-1)h)\right|^2
\\ \notag 
&\qquad +\frac{h-x}{h}\left| \sqmob(u^h)((i+1)h))-\sqmob(u^h)(ih)\right|^2dx \\
  &=  \frac{h}{2}\Big(\Big|\frac{\sqmob(u^h)(ih)-\sqmob(u^h)((i-1)h)}{h}\Big|^2
+\Big|\frac{\sqmob(u^h)((i+1)h)-\sqmob(u^h)(ih)}{h}\Big|^2 \Big).
\end{align*}
Hence,
\begin{align}
  \label{eq:neusiebenA}
  \begin{split}
I_{7a}& \leq \kappa C \sum_{\ell \in \Z} \lambda_\ell^2
\int_0^{t\wedge T_h}
  h \sum_{i=1}^\Lh \frac{1}{2}\Big(\frac{1}{m_\sigma(u^h_i)}+\frac{1}{m_\sigma(u^h_{i-1})}\Big) \Big|\frac{\sqmob(u^h_i)-\sqmob(u^h_{i-1})}{h}\Big|^2 \,ds\\
&\qquad \leq C \delta  \kappa \sum_{\ell \in \Z} \lambda_\ell^2
\int_0^{t\wedge T_h} 
\sum_{i=1}^\Lh \dashint_{u^h_i}^{u^h_{i+1}} |\tau|^{-p-2} \,d\tau
\int_{ih}^{(i+1)h} |u^h_x|^2 \,dx
\,ds
\\
& 
\qquad\quad  +C C_\delta \kappa \sum_{\ell \in \Z} \ell^2 \lambda_\ell^2
\mathbb{E}\int_0^{t\wedge T_h} R(s)\, ds \, ,
  \end{split}
\end{align}
where we used Lemma~\ref{lem:appenHilfeSiebenA} in  the last step.

Using the estimate $\Big|\frac{g_\ell(x+h)-g_\ell(x)}{h}\Big|^2 \leq C\ell^2$ together with the Oscillation Lemma~\ref{lem:lowerbound} and the definition of $m_\sigma(\cdot)$, $I_{7b}$  is readily estimated by
\begin{align}
	\label{eq:neuSiebenA2}
	\begin{split}
	I_{7b}& \leq \kappa C \sum_{\ell \in \Z} \lambda_\ell^2\ell^2
	\int_0^{t\wedge T_h} h\sum_{i=1}^\Lh\frac{(\ui)^n}{m_\sigma(\ui)}\, ds\\
	& \leq  \kappa C \sum_{\ell \in \Z} \lambda_\ell^2\ell^2 \int_0^{t\wedge T_h} R(s)\, ds +1.
\end{split}
\end{align}
Hence, $I_7$ can be estimated against terms which can either be absorbed or which are of Gronwall-type.

Ad $I_8$: In the same spirit, using in particular inequality \eqref{eq:oscillation}, we have 
\begin{align}
\nonumber
&\frac{1}{2} \sum_{|\ell|\le N_h}
\int_0^{t\wedge T_h} \Big(F''(u^h),\Big(\sum_{i=1}^\Lh Z_i(\lambda_\ell g_\ell) e_i\Big)^2\Big)_h \,ds
\\
\nonumber
&\leq \delta \sum_{\ell \in \Z} \lambda_\ell^2
\int_0^{t\wedge T_h} \sum_{i=1}^\Lh \dashint_{u^h_i}^{u^h_{i+1}} \mathcal |\tau|^{-p-2} \,d\tau \int_{ih}^{(i+1)h} |u^h_x|^2 \,dx \,ds
\\&~~~
\label{eq:prop8}
+C_\delta \sum_{\ell \in \Z} \ell^2 \lambda_\ell^2 \int_0^{t\wedge T_h} R(s) \,ds =: I_{8a} + I_{8b}
\end{align}
with a constant $C_\delta(\bar p)$ depending on $\delta>0.$
For sufficiently large $\kappa$, $I_{8a}$ can be absorbed by $\mathbb{E}[I_{1a}]$, while $I_{8b}$ will become a Gronwall term.

Ad $I_9$: Using periodicity, the special form of the stiffness matrix, and \eqref{eq:integral5}, we obtain 
\begin{align}
	\nonumber
	&\frac{1}{2} \sum_{|\ell|\le N_h}
	\int_0^{t\wedge T_h} \int_\domain \Big(\sum_{i=1}^\Lh Z_i(\lambda_\ell g_\ell) (e_i)_x \Big)^2 \,dx \,ds
	\\& \nonumber
	=
	\frac{1}{2h}
	\sum_{|\ell|\le N_h} \lambda_\ell^2 \int_0^{t\wedge T_h}  \sum_{i=1}^\Lh \Big(\int_\domain \partial_h^- ((\sqmob(u^h) g_\ell)_x)e_{i+1} \,dx \Big)^2
	\,ds
	\\& 
	\leq
	\frac{1}{2h} \sum_{\ell \in \mathbb{Z}} \lambda_{\ell}^{2} \int_{0}^{t\wedge T_{h}} \sum_{i=1}^{L_{h}} \left( \int_\ort \partial_h^-((\sqmob(u^h)g_\ell)_x) e_{i+1} \dx \right)^2 \ds \, .
\end{align}
Applying Lemma~\ref{lem:nineAnew}, we infer for arbitrary positive numbers $\varepsilon$ and $\eta$
\begin{align}\notag
	&
	\frac{1}{2h} \sum_{\ell \in \mathbb{Z}} \lambda_{\ell}^{2} \int_{0}^{t\wedge T_{h}} \sum_{i=1}^{L_{h}} \left( \int_\ort \partial_h^-((\sqmob(u^h)g_\ell)_x) e_{i+1} \dx \right)^2 \ds
	\\ \notag
	&\le 
	(1+\eta) C_{Strat} \left(1 + \varepsilon \tfrac{n-2}{2} \right)
	C_{osc}^{n-2} \int_{0}^{t\wedge T_{h}} h \sum_{i=1}^{L_{h}}  (\ui)^{n-2}  
	|(\Delta_{h} \uh )_{i}|^{2} \, ds
	\\ \notag
	&\quad+(1+\eta) C_{Strat} \left(\tfrac{(n-2)^{2}}{4} + \tfrac{n-2}{2\varepsilon } \right)
	C_{osc}^{4-n} 
	\\ \notag
	&\quad \qquad \cdot \int_{0}^{t\wedge T_{h}} \discint (\ui)^{n-4}   \bigg\{ \left|\tfrac{u^{h}_{i+1}-u^{h}_{i}}{h}\right|^{4} +2 \left| \tfrac{u^{h}_{i+1}-u^{h}_{i}}{h} \right|^{2}  \left| \tfrac{u^{h}_{i}-u^{h}_{i-1}}{h} \right|^{2} + \left|\tfrac{u^{h}_{i}-u^{h}_{i-1}}{h}\right|^{4} \bigg\} \ds 
	\\  \notag
	&\quad+
	\tilde{C}_{\eta,1}
	\int_{0}^{t\wedge T_{h}}  \discint (u^{h}_{i})^{n-2} \left| \tfrac{u^{h}_{i+1} - u^{h}_{i}}{h} \right|^{2}  \, ds
	+
	\tilde{C}_{\eta,2}
	\int_{0}^{t\wedge T_{h}} 	\int_\ort (\uh)^{n}\, dx \, ds 
	\\
	&=: I_{9a} +\dots + I_{9d}\, .
\end{align} 
The terms $I_{9c}$ and $I_{9d} $ can be estimated, using Lemma~\ref{lem:poroussquare}
and Lemma~\ref{lem:uHochN},
and subsequently be absorbed or serve as Gronwall terms. To estimate $I_{9b}$ we note 
\begin{align}\notag \label{est:Absorber}
	&(1+\eta) C_{Strat}\left(\tfrac{(n-2)^{2}}{4} + \tfrac{n-2}{2\varepsilon } \right)
	C_{osc}^{4-n} 
	\\ \notag
	&\qquad \cdot\int_{0}^{t\wedge T_{h}} \discint (\ui)^{n-4}  \bigg\{ \left|\tfrac{u^{h}_{i+1}-u^{h}_{i}}{h}\right|^{4} +2 \left| \tfrac{u^{h}_{i+1}-u^{h}_{i}}{h} \right|^{2}  \left| \tfrac{u^{h}_{i}-u^{h}_{i-1}}{h} \right|^{2} + \left|\tfrac{u^{h}_{i}-u^{h}_{i-1}}{h}\right|^{4} \bigg\} \, ds  \, 
	\\ \notag
	&\le
	2(1+\eta) C_{Strat}\left(\tfrac{(n-2)^{2}}{4} + \tfrac{n-2}{2\varepsilon } \right)
	C_{osc}^{4-n}  \int_{0}^{t\wedge T_{h}} \discint (\ui)^{n-4} 
	\left| \tfrac{u^{h}_{i+1}-u^{h}_{i}}{h} \right|^{2}  \left| \tfrac{u^{h}_{i}-u^{h}_{i-1}}{h} \right|^{2}  \, ds  
	\\ 
	&\quad+(1+\eta) C_{Strat}\left(\tfrac{(n-2)^{2}}{4} + \tfrac{n-2}{2\varepsilon } \right)
	C_{osc}^{4-n} (1+C_{osc}^{n-4})\int_{0}^{t\wedge T_{h}}  \discint (\ui)^{n-4}  \left|\tfrac{u^{h}_{i+1}-u^{h}_{i}}{h}\right|^{4} \, ds \, ,
\end{align}
where we used 
\begin{align}\notag
	&\discint (\ui)^{n-4} \left(\left|\tfrac{u^{h}_{i+1}-u^{h}_{i}}{h}\right|^{4} + \left|\tfrac{u^{h}_{i}-u^{h}_{i-1}}{h}\right|^{4} \right)
	\\ \notag
	&\le \discint (\ui)^{n-4}\left|\tfrac{u^{h}_{i+1}-u^{h}_{i}}{h}\right|^{4} 
	+
	\discint C_{osc}^{n-4}(\uim)^{n-4} \discint \left|\tfrac{u^{h}_{i}-u^{h}_{i-1}}{h}\right|^{4}
	\\ 
	&= (1+ C_{osc}^{n-4})\discint(\ui)^{n-4}
	\left|\tfrac{u^{h}_{i+1}-u^{h}_{i}}{h}\right|^{4} \, .
\end{align}
According to \eqref{eq:discPmeSign}, we see that the terms in \eqref{est:Absorber} and $I_{9a}$ have to be absorbed in  
\begin{align}\notag
	&(C_{Strat}+S) \tfrac{(n-2)(3-n)}{3} \tfrac{(1+C_{osc})^{n-4}}{2} h \sum_{i=1}^{L_{h}}  (\ui)^{n-4} \left( \left|\tfrac{\uip -\ui}{h} \right|^{2} \left|\tfrac{\ui -\uim}{h} \right|^{2} + \left|\tfrac{\uip -\ui}{h} \right|^{4} \right)
	\\ 
	&+
	(C_{Strat}+S) \tThInt \discint (\ui)^{n-2}\left|(\Delta_h u^h)_i\right|^2 \, ds \, .
\end{align}
As $2> 1+C_{osc}^{n-4}$  
and $\eta$ may be chosen arbitrarily small, for absorption the conditions
\begin{align}\label{absorbcond1}
	C_{Strat}C_{osc}^{n-2} \left(1 + \varepsilon \tfrac{n-2}{2} \right) < (C_{Strat}+S) \,
	\Leftrightarrow \,
	C_{Strat} \left( C_{osc}^{n-2}  \left(1 + \varepsilon \tfrac{n-2}{2} \right) -1 \right) < S
\end{align} 
and
\begin{align}\label{absorbcond2} \notag
	&2C_{Strat}\left(\tfrac{(n-2)^{2}}{4} + \tfrac{n-2}{2\varepsilon } \right)
	C_{osc}^{4-n} 
	< (C_{Strat}+S)	
	\tfrac{(n-2)(3-n)}{3} \tfrac{(1+C_{osc})^{n-4}}{2}
	\\
	\Leftrightarrow \, 
	&C_{Strat}\left(\left(\tfrac{3(n-2)}{4(3-n)} + \tfrac{3}{2\varepsilon (3-n)} \right)
	\frac{4C_{osc}^{4-n} }{(1+C_{osc})^{n-4}} - 1\right)
	< S	
\end{align}
have to be met. 
Therefore, we need $S$ to satisfy
\begin{align}\label{AbsorbMax}
	C_{Strat}^{-1}S > \max \Big\{C_{osc}^{n-2} \left(1 + \varepsilon \tfrac{n-2}{2} \right) -1 , \left(\tfrac{3(n-2)}{4(3-n)} + \tfrac{3}{2\varepsilon (3-n)} \right)
	\frac{4C_{osc}^{4-n}}{(1+C_{osc})^{n-4}} - 1 \Big\} \,.
\end{align}
We note that $C_{osc}^{n-2} \left(1 + \varepsilon \tfrac{n-2}{2} \right)  $ and $\left(\tfrac{3(n-2)}{4(3-n)} + \tfrac{3}{2\varepsilon (3-n)} \right)
\frac{4C_{osc}^{4-n}}{(1+C_{osc})^{n-4}}$ are continuous, the former being monotone increasing  and the latter monotone decreasing as functions of $\varepsilon > 0$.
Therefore, setting $K_{1} := \frac{4 C_{osc}^{4-n} }{(1+C_{osc})^{n-4}}$ and $K_{2} := C_{osc}^{n-2}$,  
we find the minimum to be attained for
$
	\varepsilon = \frac{3}{2} \frac{K_{1}}{K_{2}} \frac{1}{3-n} \, .
$
Thus, we require
\begin{align}
	S > C_{Strat}  K_{1} \frac{3}{4}  \frac{n-2}{3-n} + C_{Strat} \left(K_{2}-1\right) \, 
\end{align}
which is our assumption on $S$, cf. (H5).

Next we take $\bar{p}$-th power, suprema, and expectation and arrive for arbitrary $t'\in [0,T_{max}]$ at the inequality
\begin{align}\label{eq:prop9}
\nonumber
&\mathbb{E}\bigg[\sup_{t\in [0,t' \wedge T_{h}]} R(t\wedge T_h)^{\bar p} \bigg]
\\& \nonumber
+ C_1
\mathbb{E}\Bigg[\bigg(\int_0^{t'\wedge T_h}  \int_\domain M_h(u^h(s)) | p_x^h(s)|^2\,dx \,ds \bigg)^{\bar{p}} \Bigg]
\\& \nonumber
+C_1 \kappa
\mathbb{E}\Bigg[\bigg(\int_0^{t'\wedge T_h}  ||\Delta_hu^h]||_h^2 \,ds\bigg)^{\bar{p}} \Bigg]
\\& \nonumber
+C_1 \kappa
\mathbb{E}\Bigg[\bigg(\int_0^{t'\wedge T_h}  \sum_{i=1}^\Lh \dashint_{u^h_{i-1}}^{u^h_i} |\tau|^{-p-2} \,d\tau \int_{(i-1)h}^{ih} |u^h_x(s)|^2 \,dx \,ds\bigg)^{\bar{p}} \Bigg]
\\& \nonumber
+C_1(C_{Strat}+S)\expect\Bigg[\bigg(\int_0^{t'\wedge T_h} \discint(\ui)^{n-4}\left|\tfrac{\uip-\ui}{h}\right|^4(s)\,ds\bigg)^{\bar{p}}\Bigg]
\\&\nonumber
+C_1(C_{Strat}+S)\expect\Bigg[\bigg(\int_0^{t'\wedge T_h} \discint(\ui)^{n-2}\left|(\Delta_h u^h)_i\right|^2(s)\,ds\bigg)^{\bar{p}}\Bigg]
\\&\nonumber
+C_1(C_{Strat}+S)\expect\Bigg[\bigg(\int_0^{t'\wedge T_h} \discint (\ui)^{n-p-4}\left|\tfrac{\uip-\ui}{h}\right|^2(s)\,ds\bigg)^{\bar{p}}\Bigg]
\\&\nonumber
+C_1\kappa(C_{Strat}+S)\expect\Bigg[\bigg(\int_0^{t'\wedge T_h} \discint (\ui)^{-2}\left|\tfrac{\uip-\ui}{h}\right|^2(s)\,ds\bigg)^{\bar{p}} \Bigg]
\\ \notag
&\le
\bar{C}(\bar{p},(\lambda_{\ell})_{\ell \in \Z}, u_{0}, T_{max}) \Bigg( 1 + \mathbb{E}\Bigg[\int_0^{t'\wedge T_h} R(s)^{\bar{p}}\,ds \Bigg] \Bigg)
\\& \qquad \qquad \qquad \qquad \quad
+
\mathbb{E}\bigg[\sup_{t\in [0,t'\wedge T_h]} (I_{4} + I_{6})^{\bar{p}} \bigg]
+
\mathbb{E}\bigg[\sup_{t\in [0,t'\wedge T_h]}I_{2}^{\bar{p}} \bigg] \, .
\end{align}
To establish \eqref{eq:proposition1}, we only have to estimate the expected values of the suprema with respect to time of the absolute $\bar{p}$-th moments of the stochastic integrals, i.e. the terms $I_2$, $I_4$, and $I_6$.
We note that $\expect[\sup_{s\in [0,T_h]} R(s)^{\bar p}]$ is finite due to the cut-off mechanism applied.

We begin with $I_4$ and $I_6$. 
Adapting the argumentation for $I_4+I_6$ of \cite{FischerGruen2018}, we get 
\begin{align}
I_4+I_6 
=
-\lambda_\ell \sum_{|\ell|\le N_h} \int_0^{t\wedge T_h} 
 \int_\domain \sqmob(u^h) p^h_x g_\ell \,dx \,d\beta_\ell.
\end{align}

By the Burkholder-Davis-Gundy inequality, we have for $\varepsilon > 0$ sufficiently small
\begin{align}
\nonumber
& \mathbb{E}\Bigg[\sup_{t\in [0,t'\wedge T_h]} \bigg|\int_0^{t\wedge T_h}   \sum_{|\ell|\le N_h} \lambda_\ell \int_\domain \sqmob(u^h) p^h_x g_\ell \,dx\,d\beta_\ell \bigg|^{\bar{p}}\Bigg]
\\&\le 
\varepsilon C_{BDG}^2 \sum_{\ell \in \Z} \lambda_\ell^2 \,
\mathbb{E}\Bigg[\bigg(\int_0^{t'\wedge T_h} 
 \int_\domain M_h(u^h) |p^h_x|^2 \,dx\,ds\bigg)^{\bar{p}}\Bigg] + C_{\varepsilon}\, .
\label{eq:prop16}
\end{align}
The first term can be absorbed by \eqref{eq:prop6}, the second one is independent of $h$.

Finally, let us discuss $I_2$.
Using \eqref{eq:prop3}, we get
\begin{align*}
I_2
&=
\kappa \sum_{|\ell|\le N_h} \int_0^{t\wedge T_h} 
\Big(g_h(u^h),\sum_{i=1}^\Lh Z_i(\lambda_\ell g_\ell) e_i \Big)_h \,d\beta_\ell
\\
&=
\kappa  \sum_{|\ell|\le N_h} \int_0^{t\wedge T_h} 
\lambda_\ell \int_\domain \big(\sqmob(u^h) g_\ell\big)_x \mathcal{I}_h[g_h(u^h)] \,dx \,d\beta_\ell.
\end{align*}
By Burkholder-Davis-Gundy, we get as above
\begin{align}
\nonumber
&
\kappa
\mathbb{E}\Bigg[ \sup_{t\in (0,t' \wedge T_h)} \bigg|\sum_{|\ell|\le N_h} \int_0^{t\wedge T_h} \Big(g_h(u^h),\sum_{i=1}^\Lh Z_i(\lambda_\ell g_\ell) e_i \Big)_h \,d\beta_\ell \bigg|^{\bar{p}} \Bigg]
\\ \nonumber
&\leq
\kappa C_{BDG} \mathbb{E}\Bigg[\bigg(
\int_0^{t'\wedge T_h} \sum_{\ell \in \Z} \lambda_\ell^2
\int_\domain |\sqmob(u^h)_x|^2|\mathcal{I}_h[g_h(u^h)]|^2  |g_\ell|^2 \,dx\,ds\bigg)^{\bar{p}/2}\Bigg]
\\ \label{eq:prop19}
&~~~
+\kappa C_{BDG} \mathbb{E}\Bigg[ \bigg(
\int_0^{t'\wedge T_h}  \sum_{\ell \in \Z} \lambda_\ell^2
\int_\domain |\sqmob(u^h) \mathcal{I}_h[g_h(u^h)]|^2 |(g_\ell)_x|^2 \,dx\,ds\bigg)^{\bar{p}/2}\Bigg]
&=:
I_{2a}+I_{2b}\, .
\end{align}

Now using the Oscillation Lemma~\ref{lem:lowerbound}, Lemma~\ref{lem:poroussquare}, and the calculus inequality $a^{-n}\leq\delta a^{-p-2}+C_\delta$ $ \forall a\in
\R^+ $, we estimate
\begin{align}
  \label{eq:newest101}
  \begin{split}
    |I_{2a}|\leq  \kappa C(\ort, \bar{p}, (\lambda_{\ell})_{\ell \in \Z})\Bigg\{ &\expect\Bigg[ \bigg( \delta\int_0^{t'\wedge T_h}\discint |\ui|^{n-4}\left|\tfrac{\uip-\ui}{h}\right|^4 ds \bigg)^{\bar{p}/2}\Bigg] \\
    + & \expect\Bigg[\bigg(\delta \int_0^{t'\wedge T_h} \discint \dashint_{\ui}^{\uip}|\tau|^{-p-2}d\tau\left|\tfrac{\uip-\ui}{h}\right|^2 ds \bigg)^{\bar{p}/2}\Bigg]\\
    + &C \expect \Bigg[\bigg(C_\delta \int_0^{t'\wedge T_h} R(s)\, ds \bigg)^{\bar{p}/2}\\
    &+ \bigg(C_\delta\bigg(\int_{0}^{t'\wedge T_{h}} \int_\ort u^h(s) \, dx \, ds\bigg)^{\tfrac{n+2}{4-n}}+C(T_{max})\bigg)^{\bar{p}/2}\Bigg]\Bigg\}\, .
  \end{split}
\end{align}
By Young's inequality terms in \eqref{eq:newest101} can either be absorbed, are independent of $h$, or will serve as Gronwall terms.
To control $I_{2b}$, we note the estimate
\begin{align}\notag
  \label{eq:neuest102}
    \int_\ort &\sqmob(u^h)^2|\Ih[g_h(u^h)]|^2 \left(g_\ell\right)_x^2 dx
    \\ \notag 
    &\leq  C\ell^2\discint \left((\ui)^n+ (\uim)^n\right)\left(1+(\ui)^{2-2n}+(\uip)^{2-2n}\right)\\
    & \leq C \ell^2\discint \left((\ui)^n+ (\ui)^{2-n}\right)
\end{align}
which comes due to the Oscillation Lemma~\ref{lem:lowerbound}. These terms can be absorbed similarly as the terms in $I_9$.

After absorption and using (H3) and the definition of the stopping time $T_{h}$,
we are left to find a bound of the the term $\mathbb{E}\left[\int_{0}^{t'\wedge T_{h}} R(s)^{\bar{p}} \, ds\right]$. From \eqref{eq:prop9} we infer
\begin{align}\notag \label{eq:prop20}
	\mathbb{E}\Bigg[\sup_{t\le t'} R(t\wedge T_{h})^{\bar{p}}\Bigg]
	&\le C(T_{max},\bar{p},u_{0}, (\lambda_{\ell})_{\ell \in \Z}) + C \mathbb{E}\Bigg[\int_{0}^{t' \wedge T_{h}} R(s)^{\bar{p}}ds\Bigg]
	\\ 
	&\le C(T_{max},\bar{p},u_{0}, (\lambda_{\ell})_{\ell \in \Z}) + C \mathbb{E}\Bigg[\int_{0}^{t'} \sup_{t\le s}R(t\wedge T_{h})^{\bar{p}}ds\Bigg] \,.
\end{align}
Owing to the definition of the stopping times $T_{h}$, we may apply Fubini's theorem in the last line of \eqref{eq:prop20}. Then, a Gronwall argument with respect to the mapping $t' \mapsto \mathbb{E} \left[ \sup_{t\le t'} R(t\wedge T_{h})^{\bar{p}} \right]$ entails the desired result,  since $t'\in [0,T_{max}]$ was chosen arbitrarily and $\uh(s)$ is constant for $s\in[T_{h},T_{max}]$.
\end{proof}

\subsection{Uniform H\"older estimates}

Let us prove that appropriate H\"older norms (with respect to space and time) of solutions to our semidiscrete scheme are square-integrable with respect to the probability measure.  
We adapt Lemma~4.7 of \cite{FischerGruen2018} to the case $n\neq 2$.
\begin{lemma}
\label{lem:Hoelder1}
Let $h\in (0,1]$, $T_{max}>0$, $\bar p>1$, $\alpha\in (0,\frac{1}{2})$.
Assume $u^h, p^h$ to be a solution  to \eqref{eq:1} and \eqref{eq:2} with initial data satisfying (H3) and (H4).

If $2\alpha\bar p>1$ holds, the stochastic integral
\begin{align}
\label{DefinitionIh}
I_h(t):=\sum_{i=1}^\Lh\sum_{|\ell|\le N_h}
\frac{1}{h} \int_0^{t\wedge T_h} \int_\domain \Big(\sqmob(u^h) \lambda_\ell g_\ell\Big)_x e_i \,dx \,d\beta_\ell(s) e_i
\end{align}
is contained in $L^{2\bar p}(\Omega;C^{\beta}([0,T_{max}];L^2(\domain)))$ with $\beta:=\alpha-\frac{1}{2\bar p}$ and there exists a constant $C_1$ independent of $h>0$ such that
\begin{align}
\label{eq:hoelder1}
||I_h||_{L^{2\bar p}(\Omega;C^{\beta}([0,T_{max}];L^2(\domain)))} \leq C_1
\end{align}
holds.
\end{lemma}
\begin{remark}
\label{rem:hoelbetter}
Choosing $\alpha$ and $\bar p$ sufficiently large, we infer the estimate 
\begin{equation}
\label{eq:hoelbetter}
\norm{I_h}{L^2(\Omega;C^{1/4}([0,T_{max}];L^2(\ort)))}\leq C_1
\end{equation}
with an $h$-independent positive constant $C_1$.
\end{remark}
\begin{proof}[Proof of Lemma~\ref{lem:Hoelder1}]
Similarly as in Lemma~4.7 of \cite{FischerGruen2018}, it is sufficient to show that
\begin{align*}
\tilde Z(s)(w)=\chi_{T_h}(s) \frac{1}{h} \sum_{i=1}^\Lh \int_\domain \Big(\sqmob(u^h) \sum_{|\ell|\le N_h} \lambda_\ell (g_\ell,w)_{L^2(\domain)} g_\ell \Big)_x e_i \,dx ~ e_i
\end{align*}
is progressively measurable and contained in
\begin{align*}
L^{2\bar p}(\Omega \times [0,T_{max}];L_2(L^2(\domain);L^2(\domain)))
\end{align*}
with a uniform bound in $h$. 

For the Hilbert-Schmidt norm of $\tilde Z$ we find
\begin{align*}
||\tilde Z(s)||_{L_2(L^2(\domain);L^2(\domain))}^2
&
\leq \frac{C}{h}\chi_{T_h}(s) \sum_{\ell \in \Z} \lambda_\ell^2
\sum_{i=1}^\Lh \int_{(i-1)h}^{ih} |(\sqmob(u^h) g_\ell)_x|^2 \,dx \int_\domain e_i^2  \,dx
\\
&
\leq C\chi_{T_h}(s)  \bigg(\sum_{\ell \in \Z} \ell^2 \lambda_\ell^2 \int_\domain |\sqmob(u^h)|^2 \,dx 
+\sum_{\ell \in \Z} \lambda_\ell^2 \int_\domain |\sqmob(u^h)_x|^2 \,dx
\bigg).
\end{align*}
Combining Lemma~\ref{lem:uHochN}, Lemma~\ref{lem:poroussquare}, and Lemma~\ref{lem:lowerbound} with Proposition~\ref{prop:integral1}  and (H4), we may prove 
\begin{align*}
\tilde Z(\cdot,\cdot)\in L^{2\bar p}(\Omega \times [0,T_{max}];L_2(L^2(\domain);L^2(\domain)))
\end{align*}
with a uniform bound in $h$. 
We have
\begin{align}\notag
	&||\tilde{Z}||^{2\bar{p}}_{L^{2\bar{p}}(\Omega \times [0,T_{max}];L_2(L^2(\domain);L^2(\domain)))}
	=
	\mathbb{E} \Bigg[\int_{0}^{T_{h}} ||\tilde{Z}(s)||^{2\bar{p}}_{L_2(L^2(\domain);L^2(\domain)))} \, ds \Bigg]
	\\ \notag
	&\le 
	C \mathbb{E} \Bigg[\int_{0}^{T_{h}}  \bigg( \int_\domain |\sqmob(u^h)|^2 \,dx \bigg)^{\bar{p}}\, ds \Bigg] 
	+C \mathbb{E} \Bigg[  \int_{0}^{T_{h}}\bigg(\int_\domain |\sqmob(u^h)_x|^2 \,dx
	\bigg)^{\bar{p}} \, ds \Bigg]
	\\ \notag
	&=: I_{1}+I_{2}\, .
\end{align} 

Ad $I_{1}$: By Lemmas~\ref{lem:uHochN} and \ref{lem:lowerbound} as well as the definition of the stopping times $T_{h}$ we have
\begin{align}\notag
	I_{1} &\le C \mathbb{E} \Bigg[\int_{0}^{T_{h}}  \bigg( \delta \int_\domain |\uh_{x}|^2 \,dx + C_{\delta} \left(\int \uh \, dx\right)^{\frac{n+2}{4-n}} +1 \bigg)^{\bar{p}}\, ds \Bigg] 
	\\ \notag 
	&\le C T_{max}\,
	\mathbb{E} \Bigg[ \delta^{\bar{p}} \bigg(  \sup_{t\in [0,T_{max}]} \int_\domain |\uh_{x}|^2 \,dx \bigg)^{\bar{p}} + C_{\delta}^{\bar{p}} \left(\sup_{t \in [0,T_{max}]}\intort \uh \, dx\right)^{\frac{\bar{p}(n+2)}{4-n}} +1  \Bigg] 
	\\ \notag
	&\le C(T_{max}, \bar{p}, n, u_{0})\,.
\end{align}

Ad $I_{2}$: We find
\begin{align*}
	I_{2} &\le C \mathbb{E} \Bigg[\int_{0}^{T_{h}} 
	\bigg( \discint (\ui)^{n-2} \bigg| \frac{\uip - \ui}{h}\bigg|^{2}	\bigg)^{\bar{p}} \, ds	
	\Bigg]
	\\
	&\le 
	C
	\mathbb{E} \Bigg[\int_{0}^{T_{h}} 
	\bigg( \sup_{x \in \ort} \uh(s)\bigg)^{\bar{p}(n-2)}
	\bigg( \discint  \bigg| \frac{\uip - \ui}{h}\bigg|^{2}	\bigg)^{\bar{p}} \, ds	
	\Bigg]
	\\
	&\le 
	C
	\mathbb{E} \Bigg[\int_{0}^{T_{h}} 
	\bigg( \intort \uh(s) \, dx + \left(\intort |\uh_{x}|^{2}(s) \, dx\right)^{\frac{1}{2}} \bigg)^{\bar{p}(n-2)}
	\left( \intort | \uh_{x} |^{2}(s) \dx	\right)^{\bar{p}} \, ds	
	\Bigg]
	\\
	&\le 
	C \, T_{max} \,
	\mathbb{E} \Bigg[ 
	\bigg( \sup_{s \in [0,T_{max}]} \intort \uh(s) \, dx \bigg)^{\bar{p}(n-2)} \bigg( \sup_{s \in [0,T_{max}]}  \intort | \uh_{x} |^{2}(s)\, dx	\bigg)^{\bar{p}} 
	\\
	& ~~~~~~~~~~~~~~~~+ \bigg(\sup_{s \in [0,T_{max}]} \intort |\uh_{x}|^{2}(s) \, dx\bigg)^{\frac{\bar{p}n}{2}} \Bigg]
	\\
	&\le 
	C \, T_{max} \,
	\Bigg\{
	\mathbb{E} 
	\Bigg[ 
	\bigg( \sup_{s \in [0,T_{max}]} \intort \uh(s) \, dx \bigg)^{2\bar{p}(n-2)} \Bigg]
	+
	\mathbb{E} \Bigg[ \bigg( \sup_{s \in [0,T_{max}]}  \intort | \uh_{x} |^{2}(s)\, dx	\bigg)^{2\bar{p}} \Bigg]
	\\
	& ~~~~~~~~~~~~~~~+
	\mathbb{E} \Bigg[ \bigg(\sup_{s \in [0,T_{max}]} \intort |\uh_{x}|^{2}(s) \, dx\bigg)^{\frac{\bar{p}n}{2}} \Bigg]
	\Bigg\}
	\\ 
	&\le 
	C(T_{max},\bar{p},n,u_{0}) \, ,
\end{align*}
which, altogether, gives the desired estimate.
Finally, we note that $\tilde Z$ is progressively measurable as it is a continuous composition of terms having this property. This gives the assertion of the lemma.
\end{proof}

\begin{lemma}
\label{lem:hoelder2}
Under the assumptions of Lemma~\ref{lem:Hoelder1} solutions $u^h$ of \eqref{eq:1} are contained  and uniformly bounded in the space
\begin{align*}
L^2(\Omega;C^{1/4}([0,T_{max}];L^2(\domain))).
\end{align*}
 In particular, a positive constant $C_2$ independent of $h>0$ exists such that
\begin{align}
\label{eq:hoelder3}
\mathbb{E}\bigg[\sup_{t_1,t_2\in [0,T_{max}]} \frac{||u^h(t_1)-u^h(t_2)||_h^2}{|t_1-t_2|^{2/4}}\bigg]
\leq C_2.
\end{align}
\end{lemma}
\begin{proof}
Following the lines of the proof in Lemma~4.10 of \cite{FischerGruen2018} we find the existence of a constant $C$ independent of $h$ such that
\begin{align}\label{eq:hoelder4}
  \begin{split}
||u^h(t_2)-u^h(t_1)||_{h}^2
&\leq C \left| \int_{t_1\wedge T_h}^{t_2\wedge T_h}\int_\domain M_h(u^h) p^h_x (u^h(t_2)-u^h(t_1))_x \,dx\,ds\right|
\\&~~~
+ C\left|\int_{t_1\wedge T_h}^{t_2\wedge T_h}\porad(u^h, u^h(t_2)-u^h(t_1))ds\right|
\\&~~~
+ C\left|\int_{t_1\wedge T_h}^{t_2\wedge T_h}\porb(u^h, u^h(t_2)-u^h(t_1))ds\right|
\\&~~~
+\left|\left|
I_h(t_2)-I_h(t_1)
\right|\right|_{L^2(\domain)}^2
\end{split}
\end{align}
is satisfied.
From Remark~\ref{rem:hoelbetter}, we infer the existence of a function $\mathcal{C}\in L^{2}(\Omega)$ such that
\begin{align}
||I_h(t_2)(\omega)-I_h(t_1)(\omega)||_{L^2(\domain)} \leq \mathcal{C}(\omega) |t_2-t_1|^{1/4}
\label{eq:hoelder5}
\end{align}
holds for all $t_1,t_2\in [0,T_{max}]$ with $t_1\leq t_2$, $\mathbb{P}$-almost surely.
For the first term in \eqref{eq:hoelder4} we find 
\begin{align*}
&\left| \int_{t_1\wedge T_h}^{t_2\wedge T_h}\int_\domain M_h(u^h) p^h_x (u^h(t_2)-u^h(t_1))_x \,dx\,ds \right|
\\ 
&\le
\left(\int_{t_1\wedge T_h}^{t_2\wedge T_h}\int_\domain M_h(u^h) (p^h_x)^{2} \, dx\, ds \right)^{\frac{1}{2}} 
\left(\int_{t_1\wedge T_h}^{t_2\wedge T_h}\int_\domain M_h(u^h)
(u^h(t_2)-u^h(t_1))^{2}_x \,dx\,ds\right)^{\frac{1}{2}} 
\\ &=: A_{1} A_{2} \, .
\end{align*}
For $A_{2}$ we have due to Poincar\'e's inequality
\begin{align*}
A_{2} &\le
C \left(\sup_{t \in [0,T_{h}]}\sup_{x \in \domain}M_h(u^h)(t,x) \sup_{t \in [0,T_{h}]}\int_\domain 
(u^h_x)^{2}(t) \,dx\right)^{\frac{1}{2}}
\sqrt{t_2-t_1}
\\&
\le
C\left(\sup_{t \in [0,T_{h}]}||\uh(t) ||^{n}_{H^{1}(\domain)} \sup_{t \in [0,T_{h}]}\int_\domain 
(u^h_x)^{2}(t) \,dx\right)^{\frac{1}{2}}
\sqrt{t_2-t_1}
\\&
\le
C\left(\sup_{t \in [0,T_{h}]} \left( C \int_{\domain} \uh(t) \, dx  + 2||\uh_{x}(t)||_{L^{2}(\domain)}  \right)^{n}\sup_{t \in [0,T_{h}]}\int_\domain 
(u^h_x)^{2}(t) \,dx\right)^{\frac{1}{2}}
\sqrt{t_2-t_1}
\\&
\le
C\sup_{t \in [0,T_{h}]} \left(  \left( \int_{\domain} \uh(t) \, dx \right)^{n+2}  + ||\uh_{x}(t)||^{n+2}_{L^{2}(\domain)} \right)^{\frac{1}{2}}
\sqrt{t_2-t_1}
\end{align*}
For the second term in \eqref{eq:hoelder4} we consider first the first summand in $\porad$ and estimate 
\begin{align*}
	&\left|-\tfrac{n-2}{6} \int_{t_1\wedge T_h}^{t_2\wedge T_h} \discint (\ui)^{n-3}\left(\left|\tfrac{\ui-\uim}{h}\right|^2+\left|\tfrac{\uip-\ui}{h}\right|^2\right)(\ui(t_{2}) - (\ui(t_{1})) \, ds \right|
	\\
	&\le
	C 
	\left(
	\int_{t_1\wedge T_h}^{t_2\wedge T_h} \discint (\ui)^{2(n-3)}\left(\left|\tfrac{\ui-\uim}{h}\right|^2+\left|\tfrac{\uip-\ui}{h}\right|^2\right)^{2} \, ds
	\right)^{\frac{1}{2}}
	\\
	&\quad \cdot \left(
	\int_{t_1\wedge T_h}^{t_2\wedge T_h} \discint (\ui(t_{2}) - \ui(t_{1}))^{2} \, ds
	\right)^{\frac{1}{2}}
	\\
	&\le 
	C \sqrt{t_{2} - t_{1}} 
	\left(
	\int_{t_1\wedge T_h}^{t_2\wedge T_h} \discint C_{osc}(\ui)^{2(n-3)}\left|\tfrac{\ui-\uim}{h}\right|^4 \, ds 
	\right)^{\frac{1}{2}}
	\cdot \left( \sup_{t \in [0, T_{h}]} 
	\discint (\ui)^{2}(t) 
	\right)^{\frac{1}{2}}
	\\
	&\le 
	C \sqrt{t_{2} - t_{1}} 
	\sup_{t \in [0,T_{h}] } ||\uh||^{\frac{n}{2}}_{H^{1}(\ort)}
	\left(\int_{t_1\wedge T_h}^{t_2\wedge T_h} \discint (\ui)^{n-4}\left|\tfrac{\ui-\uim}{h}\right|^4 \, ds 
	\right)^{\frac{1}{2}}
	\\ &\le 
	C \sqrt{t_{2} - t_{1}} 
             \sup_{t \in [0,T_{h}] } \left(  ||\uh_{x}(t) ||^{n}_{L^{2}(\ort)} +  \left(\int_{\ort} \uh(t) \, dx \right)^{n} \right)^{\frac{1}{2}}\cdot\\
  &\qquad\qquad\qquad\qquad\qquad\qquad\qquad\cdot \left(\int_{t_1\wedge T_h}^{t_2\wedge T_h} \discint (\ui)^{n-4}\left|\tfrac{\ui-\uim}{h}\right|^4 \, ds 
	\right)^{\frac{1}{2}} \, .
\end{align*}
Contributions of the other terms in $\porad$ are estimated similarly, using the Oscillation Lemma~\ref{lem:lowerbound} and $\frac{\uip -\uim}{2h} = \frac{\uip - \ui}{2h} + \frac{\ui - \uim}{2h}$ combined with Young's inequality.

The third term in \eqref{eq:hoelder4} gives 
\begin{align*}
	&\left| -\int_{t_1\wedge T_h}^{t_2\wedge T_h} \discint (\ui)^{n-2}\left(\tfrac{\uip-2 \ui+\uim}{h^2}\right) (\ui(t_{2}) - \ui(t_{1})) \, ds	\right|
	\\
	&\le
	C \sqrt{t_{2}-t_{1}} \sup_{t \in [0,T_{h}]}\left( ||\uh_{x}(t)||^{n} + \left( \int_{\ort} \uh(t) \, dx \right)^{n}\right)^{\frac{1}{2}} 
	\left(\int_{t_1\wedge T_h}^{t_2\wedge T_h} \discint (\ui)^{n-2}|(\Delta_{h}\uh)_{i}|^{2} \, ds \right)^{\frac{1}{2}} \, .
\end{align*}

Inequality \eqref{eq:hoelder4} then entails $\mathbb{P}$-almost surely
\begin{align}\label{eq:neu1414}
	\begin{split}
		&||u^h(t_2,\omega)-u^h(t_1,\omega)||_h^2
		\\&
		\leq C \sup_{t \in [0,T_{h}]}  \left(R(t)^{\frac{n+2}{2}} + \left(\intort \uh(t) \, dx \right)^{n+2}
		\right)^{\frac{1}{2}}
		\bigg(\int_{t_1\wedge T_h}^{t_2\wedge T_h}  \int_\domain M_h(u^h) |p^h_x|^2 \, dx \, ds \bigg)^{1/2} \sqrt{t_2-t_1}
		\\&
		~~~
		+C \sup_{t \in [0,T_{h}]}\left(R(t)^{\frac{n}{2}}+\left(\intort \uh(t) \, dx\right)^{n}\right)^{\frac{1}{2}}\sqrt{t_2-t_1} \\
		& \qquad \cdot\Bigg(\bigg(\int_{t_1\wedge T_h}^{t_2\wedge T_h}\discint (u^h_i)^{n-2}(\Delta_h u^h)_i^2 \, ds \bigg)^{1/2}+\bigg(\int_{t_1\wedge T_h}^{t_2\wedge T_h}\discint (u^h_i)^{n-4}\left|\tfrac{\uip-\ui}{h}\right|^4 \, ds\bigg)^{1/2}\Bigg)\\
		\\&~~~
		+C\mathcal{C}^2(\omega)\sqrt{t_2-t_1}.
	\end{split}
\end{align}

Dividing by $|t_2-t_1|^{1/2}$, taking the supremum with respect to $t_1$ and $t_2$, and taking expectations, we get

\begin{align*}
	&\mathbb{E}\Bigg[
	\sup_{t_1,t_2\in [0,T_{max}]}
	\frac{||u^h(t_2\wedge T_h)-u^h(t_1\wedge T_h)||_h^2}{|t_2-t_1|^{2/4}}
	\Bigg]
	\\&
	\leq C \mathbb{E}\Bigg[\sup_{s\in [0,T_h]} \bigg(R(s)^{\frac{n+2}{2}} + \left(\intort \uh(s) \, dx \right)^{n+2}\bigg)\Bigg]
	\\&~~~~
	+ C \mathbb{E}\Bigg[\int_0^{T_h}   \int_\domain M_h(u^h) |p^h_x|^2 \,dx \,ds\Bigg]
	+C \mathbb{E}[\mathcal{C}^2(\omega)]
	\\&~~~~
	+C\expect\Biggl[\int_0^{T_h}\discint (u^h_i)^{n-2}(\Delta_h u^h)_i^2ds \Bigg]
	+C\mathbb{E} \Bigg[\int_0^{ T_h}\discint (u_i^h)^{n-4}\left|\tfrac{\uip-\ui}{h}\right|^4ds \Bigg]
	\\&~~~~ +\expect\Bigg[\sup_{s\in [0,  T_h]}\bigg(R(s)^{\frac{n}{2}}+\bigg(\intort \uh(s) \, dx\bigg)^{n}\bigg)\Bigg] \, .
\end{align*}
By Proposition~\ref{prop:integral1}, the definition of the stopping times $T_{h}$ and since solutions $\uh$ are constant in the time interval $[T_{h}, T_{max}]$, the result follows, as the spatial H\"older property is a consequence of the standard embedding $H^1(\ort)\subset C^{1/2}(\ort)$.
\end{proof}

Similarly, as in \cite{FischerGruen2018} Lemma~4.11, we get
\begin{lemma}
\label{lem:hoelder3}
Under the assumptions of Lemma~\ref{lem:hoelder2}, solutions $u^h$ to \eqref{eq:1} are space-time H\"older-continuous almost surely. In particular, there is a positive constant $C_3$ independent of $h>0$ for which we have
\begin{align}
\mathbb{E} \big[ ||u^h||_{C^{1/2,1/8}(\domain\times [0,T_{max}])}^2 \big]
\leq C_3.
\label{eq:hoelder7}
\end{align}
\end{lemma}

\subsection{Estimates on pressure and evolution of discrete masses}
Finally, in our passage to the limit we need an uniform estimate on the pressures $p^h$ as well as on the evolution in time of discrete masses.
\begin{lemma}\label{lem:EstDiscMass}
	Let $h$ be such that 
	\begin{align}\label{lem:EstDiscMass1}
		\overline{\uh_{0}} \in (\overline{u_{0}}(1-\varepsilon), \overline{u_{0}}(1+\varepsilon))
	\end{align}
	for an $\varepsilon << 1$.
	Then there is a constant $C(T_{max})$ such that
	\begin{align}\label{eq:mass0}
		\mathbb{E} \left[ \sup_{t\in[0,T_{max}]} \left| \overline{\uh}(t) - \overline{{u}_{0}}	\right| \right] \le hC(T_{max}) \, .
	\end{align}
\end{lemma}
\begin{proof}
	Using Lemma~\ref{lem:MassOpAB} and \eqref{eq:1}, we find
	\begin{align}\notag
		&\tfrac{1}{(C_{Strat}+S)} \, \mathbb{E} \left[\sup_{t\in[0,T_{max}]} \left|	\intort \uh(t) \, dx  
		- \intort \uh_{0} \, dx \right| \right]
		\\ \notag
		&\le (n-2)\mathbb{E} \Bigg[ \sup_{t\in[0,T_{max}]} \bigg| \int_{0}^{t \wedge T_{h}} \discint \bigg|\frac{\ui -\uim}{h}\bigg|^{2}
		\bigg\{ 
		\uh(\theta(i-1,i))^{n-3} 
		\\ \notag 
		&\qquad \qquad \qquad \qquad \qquad \qquad- \tfrac{5}{12} ((\ui)^{n-3} + (\uim)^{n-3})
		\\ \notag 
		&\qquad \qquad \qquad \qquad \qquad \qquad- \tfrac{1}{12} ((\uip)^{n-3} + (\ui)^{n-3})
		\bigg\}	\, ds \bigg| \Bigg]
		\\ \notag
		&\quad + \tfrac{(n-2)}{12} \mathbb{E} \Bigg[
		\sup_{t\in[0,T_{max}]}
		\bigg| 
		\int_{0}^{t \wedge T_{h}}  \discint \bigg\{ \left((\uip)^{n-3} + 2(\ui)^{n-3} + (\uim)^{n-3}\right)
		\\ \notag
		&\qquad \qquad \qquad \qquad \qquad \qquad \bigg(  \frac{\uip -2\ui + \uim}{h^{2}}\bigg) (\ui - \uim) \bigg\} \, ds \bigg| \Bigg]
		\\
		&=:
		I + II \, .
	\end{align}
	for $\uh(\theta(i-1,i)) \in (\uim,\ui)$.
	
	As $x \mapsto x^{n-3}$ is continuous and monotone on $\mathbb{R}^{+}$, we have 
	\begin{align}
		\uh(\theta(i-1,i))^{n-3} \in [\max\{ \uim,\ui \}^{n-3}, \min \{ \uim,\ui\}^{n-3}] 
	\end{align} 
	and find $\tilde{\theta}_{i-1} \in [0,1]$ such that
	\begin{align}
		\uh(\theta(i-1,i)) = \tilde{\theta}_{i-1} (\uim)^{n-3} + (1-\tilde{\theta}_{i-1}) (\ui)^{n-3} \, .
	\end{align}
	Then we get
	\begin{align}\notag
		& 
		\uh(\theta(i-1,i))^{n-3} 
		- \tfrac{5}{12} ((\ui)^{n-3} + (\uim)^{n-3})
		- \tfrac{1}{12} ((\uip)^{n-3} + (\ui)^{n-3})
		\\ \notag
		&= 
		(\tilde{\theta}_{i-1} - \tfrac{5}{12}) (\uim)^{n-3}
		+ (1-\tilde{\theta}_{i-1} - \tfrac{6}{12}) (\ui)^{n-3} - \tfrac{1}{12}(\uip)^{n-3} 
		\\ 
		&=  
		-(\tilde{\theta}_{i-1} - \tfrac{5}{12}) \left((\ui)^{n-3} - (\uim)^{n-3}\right)
		- \tfrac{1}{12} \left((\uip)^{n-3} - (\ui)^{n-3}\right)
		 \, .
	\end{align}
	Inserting this in $I$, gives
	\begin{align}\notag 
		I   &= (n-2) \mathbb{E} \Bigg[ \sup_{t\in[0,T_{max}]} \bigg| \int_{0}^{t \wedge T_{h}} \discint \bigg|\frac{\ui -\uim}{h}\bigg|^{2}
		\bigg\{ 
			-(\tilde{\theta}_{i-1} - \tfrac{5}{12}) \left((\ui)^{n-3} - (\uim)^{n-3}\right)
			\\ \notag
			&\qquad \qquad \qquad \qquad - \tfrac{1}{12} \left((\uip)^{n-3} - (\ui)^{n-3}\right)
		\bigg\}	\, ds \bigg| \Bigg] \, .
	\end{align}
	Applying the mean-value theorem, the Oscillation Lemma~\ref{lem:lowerbound} and Hölder's inequality, we infer 
	\begin{align}\notag
		&\tfrac{n-2}{12}\, \mathbb{E} \Bigg[ \sup_{t\in[0,T_{max}]}
			\Big| \int_{0}^{t \wedge T_{h}}  \discint \bigg| \frac{\ui - \uim}{h} \bigg|^{2}  ((\uip)^{n-3} - (\ui)^{n-3}) \, ds \Big|
		\Bigg]
		\\ \notag
		&\le
		C \mathbb{E} \Bigg[  \sup_{t\in[0,T_{max}]}
		\Big| \int_{0}^{t \wedge T_{h}} \discint \bigg| \frac{\ui - \uim}{h} \bigg|^{2}  \frac{((\uip)^{n-3} - (\ui)^{n-3})}{\uip -\ui} (\uip -\ui) \, ds \Big|
		\Bigg]
		\\ \notag
		&\le
		C \mathbb{E} \Bigg[ \sup_{t\in[0,T_{max}]}
		\Big| \int_{0}^{t \wedge T_{h}} \discint \bigg| \frac{\ui - \uim}{h} \bigg|^{2}  \uh(\hat{\theta}(i,i+1))^{n-4} |\uip -\ui| \, ds \Big|
		\Bigg]
		\\ \notag
		&\le
		C \mathbb{E} \Bigg[ \sup_{t\in[0,T_{max}]}
		\Big| \int_{0}^{t \wedge T_{h}} \discint \bigg| \frac{\ui - \uim}{h} \bigg|^{2}  C_{osc}^{4-n} (\uip)^{n-4} |\uip -\ui| \, ds \Big|
		\Bigg]
		\\ \notag
		&\le
		C \mathbb{E} \Bigg[ 
		\int_{0}^{T_{h}} \discint \bigg| \frac{\ui - \uim}{h} \bigg|^{4}   C_{osc}^{2(4-n)}(\uim)^{n-4} \, ds 
		\Bigg]^{\frac{1}{2}}
		\\ \notag 
		&\qquad \cdot
		\mathbb{E} \Bigg[ 
		\int_{0}^{T_{h}} \discint  C_{osc}^{4-n}(\uip)^{n-4} |\uip -\ui|^{2} \, ds 
		\Bigg]^{\frac{1}{2}} \, .
	\end{align}
	Since the first term is controlled uniformly by Proposition~\ref{prop:integral1}, we estimate the second one and get
	\begin{align}\notag \label{lem:EstDiscMass2}
		&\mathbb{E} \Bigg[ 
		\int_{0}^{T_{h}} \discint  (\uip)^{n-4} |\uip -\ui|^{2} \, ds 
		\Bigg]^{\frac{1}{2}}
		\\ \notag
		&\le
		C h \mathbb{E} \Bigg[ 
		||\uh||_{L^{\infty}([0,T_{max}]\times\ort)}^{n+p-2} 
		\int_{0}^{T_{h}} \discint  (\uip)^{-p-2} \bigg|\frac{\uip -\ui}{h}\bigg|^{2} \, ds \, 
		\Bigg]^{\frac{1}{2}}
		\\ \notag 
		&\le
		C h \Bigg\{ \mathbb{E} \Bigg[ 
		\bigg(\int_{0}^{T_{h}} \discint  (\uip)^{-p-2} \bigg|\frac{\uip -\ui}{h}\bigg|^{2} \, ds \bigg)^{2}
		\Bigg]
		+
		\mathbb{E} \Bigg[  ||\uh||_{L^{\infty}([0,T_{max}]\times\ort)}^{2(n+p-2)}  
		\Bigg] \Bigg\}^{\frac{1}{2}}
		\\
		&< h C(T_{max}) \, .
	\end{align}
	For the last estimate we used $n+p-2 > 0$ as well as Proposition~\ref{prop:integral1}. Similarly, 
	\begin{align}\notag
		&(n-2)|\tfrac{5}{12} - \tilde{\theta}_{i-1}| \, \mathbb{E} \Bigg[ 
		\sup_{t\in [0,T_{max}]}
		\Big| \int_{0}^{t \wedge T_{h}} \discint \bigg| \frac{\ui - \uim}{h} \bigg|^{2}  ((\ui)^{n-3} - (\uim)^{n-3}) \, ds \Big|
		\Bigg]
		\\ \notag
		&\le h C(T_{max})
		\, .
	\end{align}
	Regarding $II$, \eqref{lem:EstDiscMass2} and again Proposition~\ref{prop:integral1} show
	\begin{align}\notag
		& \frac{n-2}{12} \, \mathbb{E} \Bigg[ \sup_{t \in [0,T_{max}]}\bigg|  \int_{0}^{t \wedge T_{h} }  \discint \bigg\{ ((\uip)^{n-3} + 2(\ui)^{n-3} + (\uim)^{n-3})
		\\ \notag
		&\qquad \qquad \qquad \qquad \qquad \qquad \bigg(  \frac{\uip -2\ui + \uim}{h^{2}}\bigg) (\ui - \uim) \bigg\} \, ds \bigg| \Bigg]
		\\ \notag
		&\le C \mathbb{E} \Bigg[ \sup_{t \in [0,T_{max}]} \int_{0}^{t \wedge T_{h}}  \discint  C_{osc}^{3-n}(\ui)^{n-3} \bigg|  \frac{\uip -2\ui + \uim}{h^{2}}\bigg| |\ui - \uim|  \, ds  \Bigg]
		\\ \notag
		&\le C \mathbb{E} \Bigg[ \int_{0}^{T_{h} }  \discint (\ui)^{n-2} |  (\Delta_{h} \uh)_{i} |^{2}  \, ds  \Bigg]^{\frac{1}{2}} \cdot
		\mathbb{E} \Bigg[ \int_{0}^{T_{h} }  \discint (\ui)^{n-4} |  \ui-\uim |^{2}  \, ds  \Bigg]^{\frac{1}{2}}
		\\ 
		&\le h C(T_{max})\, .
	\end{align}
	Combining the estimates above, we showed
	\begin{align}
		\mathbb{E} \left[ \sup_{t \in [0,T_{max}]} \left| \intort \uh(t)\, dx - \intort \uh_{0} \, dx	\right|	\right] \le h C(T_{max}) \, .
	\end{align}
	Then, 
	\begin{align}\notag
		\mathbb{E} \left[ \sup_{t \in [0,T_{max}]}| \overline{\uh}(t) - \overline{u_{0}}| \right] 
		&\le 
		\mathbb{E} \left[ \sup_{t \in [0,T_{max}]} | \overline{\uh}(t) - \overline{{u}^{h}_{0}}| \right] + \mathbb{E} \left[ |\overline{{u}^{h}_{0}} - \overline{u_{0}}| \right]
		\\
		&\le 
		h C(T_{max}) + O(h) \, , 
	\end{align}
	where we used \eqref{lem:EstDiscMass1}. Thus,
	we have shown \eqref{eq:mass0}.
\end{proof}

Similarly, as in \cite{FischerGruen2018} Lemma~4.12, we get
\begin{lemma}
\label{lem:pbound}
For any $q\in [1,2)$ there exists some $C>0$ such that
\begin{align*}
\mathbb{E}\left[\left(\int_0^{T_{max}}  \int_\domain |p^h_x|^2 + |p^h|^2 \,dx \,dt\right)^{q/2}\right] \leq C
\end{align*}
holds for all $h\in (0,1]$.
\end{lemma}

\section{Convergence of the scheme}
\label{sec:convergence}
\subsection{Compactness results}

We apply the Jakubowski-Skorokhod theorem \cite{Jakubowski} 
to identify a stochastic basis such that  a subsequence of the solutions to the semidiscrete scheme \eqref{eq:1}, \eqref{eq:2} almost surely converges in topologies which are appropriate for a passage to the limit in the nonlinearities of equation \eqref{eq:stfeSingPot}.

In the subsections to follow, we shall show that this limit is indeed a weak martingale solution to the stochastic thin-film equation \eqref{eq:stfeSingPot} in the sense of Definition~\ref{DefinitionMartingaleSolution}.

In our setting, we consider for $\gamma\in(0,1/2)$ the path spaces
\begin{align*}
\mathcal X_u&:=C^{\gamma,\gamma/4}(\domain \times [0,T_{max}]),
\\
\mathcal X_{u_{x}}&:=(L^2(\ort\times [0,T_{max}] ))_{weak},
\\ 
\mathcal X_{\Delta u}&:=(L^2(\ort\times [0,T_{max}] ))_{weak},
\\
\mathcal X_{p}&:=(L^2([0,T_{max}];H^1_{per}(\domain)))_{weak},
\\
\mathcal X_J&:=L^2(\domain \times [0,T_{max}])_{weak},
\\
\mathcal X_W&:=C([0,T];L^2(\domain)),
\\
\mathcal X_{u_0}&:= H^1_{per}(\ort),
\end{align*}
associated with the solutions to our semidiscrete scheme $u^h$, $p^h$, as well as with $u^{h}_{x}$, $\Delta_{h}\uh$, the corresponding pseudo-fluxes
\begin{align}
\label{DefinitionPseudoFlux}
J^h:=\chi_{T_h} \sqrt{M_h(u^h)} p^h_x,
\end{align}
and the Wiener process $W$,
respectively. 
By standard arguments (compare Lemma~5.2 in \cite{FischerGruen2018}) we get the following result.
\begin{lemma}
\label{lem:comp2}
On the path space $\mathcal X:=\mathcal X_u\times \mathcal X_{u_{x}}\times  \mathcal X_{\Delta u} \times \mathcal{X}_p \times \mathcal X_J \times \mathcal X_W\times \mathcal X_{u_0}  $, the joint laws $\mu_h$,
\begin{align*}
&\quad \mu_h(A\times B\times C\times D\times E \times F \times G)\\&:=\mathbb{P}\big[\{u^h\in A\}\cap \{u_x^h\in B\} \cap \{\Delta_{h} \uh \in C\}\cap \{p^h\in D\} \\ \notag 
&\qquad\cap \{J^h\in E\}
 \cap \{W \in F\} \cap \{u_0^{h}  \in G \}\big],
\end{align*}
for $h\in (0,1]$ are tight.
\end{lemma}

The application of Jakubowski's theorem (i.e. Theorem~5.1 in \cite{FischerGruen2018}) yields the following.
\begin{proposition}
\label{prop:conv1}
Let $\gamma\in (0,1/2) $ be given and assume $u^h$, $p^h$, $T_h$ to be a sequence of solutions to our semidiscrete scheme \eqref{FaedoGalerkinScheme} in the sense of Lemma~\ref{ExistenceDiscrete}, defined on the same stochastic basis $\Basis$ with respect to the Wiener process $W$. Then there exist a subsequence (not relabeled), a stochastic basis $(\tilde\Omega, \tilde{\mathcal F},  \tprobab)$, sequences of random variables
\begin{align*}
&\tilde u^h:\tilde\Omega \rightarrow \uspace,
\\
&\tilde w^h:\tilde\Omega \rightarrow L^{2}(\ort \times [0,T_{max}]),
\\
&\tilde z^h:\tilde\Omega \rightarrow L^{2}(\ort \times [0,T_{max}]),
\\
&\tilde p^h:\tilde\Omega \rightarrow \pspace,
\\
&\tilde J^h:\tilde\Omega \rightarrow L^{2}(\ort \times [0,T_{max}]),\\
&\tilde u_0^h:\tilde\Omega\rightarrow H^1_{per}(\ort),
\end{align*}
a sequence of $L^2(\domain)$-valued processes $\tilde W^h$ on $\tilde\Omega$, and random variables
\begin{align*}
\tilde u&\in L^2(\tilde\Omega;\uspace),
\\
\tilde{w}&\in L^{2}(\tilde\Omega;L^2(\domain \times [0,T_{max}])),
\\
\tilde{z}&\in L^{2}(\tilde\Omega;L^2(\domain \times [0,T_{max}])),
\\
\tilde p&\in L^{3/2}(\tilde\Omega;L^2([0,T_{max}];H^1_{per}(\domain))),
\\
\tilde J&\in L^{2}(\tilde\Omega;L^2(\domain \times [0,T_{max}])),
\\
\tilde u_0&\in L^2(\tilde\Omega; H^1_{per}(\ort)),
\end{align*}
as well as an $L^2(\domain)$-valued process $\tilde W$ on $\tilde\Omega$ such that the following holds:
\begin{itemize}
\item[i)] The law of $(\tilde u^h,\tilde w^h, \tilde{z}^{h}, \tilde p^h,\tilde J^h, \tilde W^h,\tilde u_0^h)$ on $\uspace \times \Jspace \times \Jspace \times \pspace \times \Jspace \times \wspace\times H^1_{per}(\ort)$ under $\tprobab$ coincides for any $h$ with the law of $(u^h,u^{h}_{x}, \Delta_{h}\uh, p^h,J^h, W, u_0^{h})$ under $\probab$.
\item[ii)] The sequence $(\tilde u^h, \tilde{w}^{h}, \tilde{z}^{h}, \tilde p^h,\tilde J^h, \tilde W^h,\tilde u_0^{h})$ converges $\tprobab$-almost surely to \newline $(\tilde u, \tilde{w}, \tilde{z}, \tilde p,\tilde J, \tilde W,\tilde u_0)$ in the topology of $\mathcal X$.
\end{itemize}
\end{proposition}
Furthermore, we introduce the random times
\begin{align*}
	\tilde T_h:=\tilde{T}^{E}_{h} \wedge \tilde{T}^{M}_{h},
\end{align*}
where
\begin{align*}
\tilde{T}^{E}_{h}:=T_{max}\wedge \inf\{t\geq 0: E_h[\tilde u^h(t)]\geq E_{max,h}\}
\end{align*}
and
\begin{align*}
	\tilde{T}_{h}^{M} := T_{max}\wedge\inf\{t\in [0,\infty):|\overline{\tuh}(t)-\overline{\tuh}(0)|\geq \frac{\overline{\tuh}(0)}{2}\} \, .
\end{align*}
Their behavior for $h\to 0$ is the content of the following lemma.

\begin{lemma}\label{lem:subsequence}
Along a subsequence, the convergence $\lim_{h\rightarrow 0} \tilde T_{h}=T_{max}$ holds ${\tprobab}$-almost surely.
\end{lemma}
\begin{proof}
We have for each $\tau\in (0,T_{max}]$ the estimate

\begin{align*}
&\tprobab(\{\tilde T_h<\tau\})=
\mathbb{P}(\{T_h<\tau\})
= \mathbb{P}(\{T_h ^{E}\wedge T_h ^{M} < \tau\})
\le \mathbb{P}(\{T_h ^{E} < \tau\}) + \mathbb{P}(\{T_h ^{M} < \tau\})\, .
\end{align*}
By Markov's inequality, \eqref{eq:proposition1}, and Lemma~\ref{lem:EstDiscMass}, we infer
\begin{align*}
	\probab(\{T^{E}_h<\tau\})
	=\mathbb{P}\Bigg(\bigg\{\omega\Big|\sup_{t\in[0,\tau)} E_h(u^h(\cdot,t))\geq \tfrac{c_{F} }{2}  h^{-\frac{p-2}{p+2}}\bigg\}\Bigg)\leq C h^{\tfrac{p-2}{p+2}}.
\end{align*}
and
\begin{align*}
	\mathbb{P}(\{T_h ^{M} < \tau\})
	&=
	\mathbb{P}\Bigg(\bigg\{\omega \Big| \sup_{t\in[0,\tau)} |\overline{\uh}(t)-\overline{\uh}(0)|\geq \frac{\overline{\uh}(0)}{2}\bigg\}\Bigg)
	\\
	&=
	\mathbb{P}\Bigg(\bigg\{\omega \Big|(\overline{\uh}(0))^{-1}\sup_{t\in[0,\tau)} |\overline{\uh}(t)-\overline{\uh}(0)|\geq \frac{1}{2}\bigg\}\Bigg)
	\\
	&\le 
	2\essup_{\omega \in \Omega}(\overline{\uh}(0))^{-1}
	  \expect \Bigg[ \sup_{t\in[0,T_{max}]}|\overline{\uh}(t)-\overline{\uh}(0)| \Bigg]  
	\\
	&\leq 2\essup_{\omega \in \Omega}(\overline{\uh}(0))^{-1} h \, C(T_{max})\, .
\end{align*}
By (H2) and the bound from (H3) which also holds for $\uh_{0}$, we infer $\tilde T_h\to T_{max} $ in probability for $h\to 0.$ The assertion  follows in a standard way.
\end{proof}
 
The relationship between  $J^{h}, p^{h}$, and $\uh$ is preserved for the  $\tilde{J}^{h}, \tilde{p}^{h}$, and $\tuh$.
It is not difficult to establish the following result.
\begin{lemma}
\label{lem:conv2}
Under the assumptions of Proposition~\ref{prop:conv1}, we identify 
$\tilde{w}^{h}$ and $\tilde{z}^{h}$ as 
\begin{align}\label{eq:path0}
	\tilde{w}^{h} &= \tilde{u}_{x}^{h}
	\\ \label{eq:path1}
	\tilde{z}^{h} &= \Delta_{h}\tilde{u}^{h}
\end{align}
and 
$\tilde{J}^{h}$ as
\begin{align}
\label{eq:path2}
\tilde{J}^h =\chi_{\tilde T_h} M_h(\tilde u^h)^{1/2} \tilde p^h_x
\end{align}
Furthermore, $\tilde p^h$ satisfies
\begin{align}
\label{eq:path3}
(\tilde p^h,\phi)_h =\chi_{\tilde T^h}\int_\domain \tilde u^h_x \phi_x \,dx + \chi_{\tilde T^h} \left(\mathcal{I}_h[F'(\tilde u^h)],\phi \right)_h 
\end{align}
for all $\phi\in X_h$ and $\tprobab$-almost all $\omega$.
\end{lemma}
The next step is to verify that $\tilde W$ and $\tilde W^h$ are $Q$-Wiener processes adapted to suitably defined filtrations $(\tilde {\mathcal{F}}_t)_{t\geq 0}$ and $(\tilde{\mathcal{F}}_t^h)_{t\geq 0}$:

We define $(\tilde {\mathcal{F}}_t)_{t\geq 0}$ to be the ${\tprobab}$-augmented canonical filtration associated with $(\tilde u,\tilde W, \tilde u_0)$, i.e.
\begin{align}
\label{eq:path4}
\tilde {\mathcal{F}}_t :=\sigma(\sigma(r_t\tilde u,r_t \tilde W)\cup \{N\in {\tilde{\mathcal{F}}}:\tprobab(N)=0\}\cup\sigma(\tu_0)).
\end{align}
Here, $r_t$ is the restriction of a function defined on $[0,T_{max}]$ to the interval $[0,t]$, $t\in [0, T_{max}]$.

Note that we do not need an explicit dependence of the filtration on $r_t \tilde J$ and $r_t \tilde p$, as the fluxes $\tilde J^h$ and the pressures $\tilde p^h$ depend in a measurable way on $\tilde u^h$ (cf.\ Lemma~\ref{lem:conv2}) and  -- later on -- we will identify
$\tilde J=\lim_{h\rightarrow 0} \tilde J^h = \lim_{h\rightarrow 0} \chi_{\tilde T_h} M_h^{1/2}(\tilde u^h) \tilde p^h_x= M^{1/2}(\tilde u) \tilde p_x$
and $\tilde p=\lim_{h\rightarrow 0}\tilde p^h = -\tilde u_{xx}+F'(\tilde u)$.

Analogously, we introduce the filtrations $(\tilde {\mathcal{F}}^h_t)_{t\geq 0}$ as the $\tprobab$-augmented canonical filtration associated with $(\tilde u^h,\tilde W^h, \tilde u_0^h)$
\begin{align}
\label{eq:path4h}
\tilde {\mathcal{F}}^h_t :=\sigma(\sigma(r_t\tilde u^h,r_t \tilde W^h)\cup \{N\in\tilde {\mathcal{F}}:\tprobab(N)=0\}\cup\sigma(\tilde u_0^h)).
\end{align}
Similarly, as in \cite{FischerGruen2018} Lemma~5.7, we obtain
\begin{lemma}
\label{lem:noise-ident}
The processes $\tilde W^h$ and $\tilde W$ are $Q$-Wiener processes  adapted to the filtrations $(\tilde {\mathcal F}^h_t)_{t\geq 0}$ and $(\tilde {\mathcal{F}}_t)_{t\geq0}$, respectively. They can be written as
\begin{align}
\label{eq:pathnew}
\tilde W^h(t)=\sum_{\ell\in \Z} \lambda_\ell \tilde \beta^h_\ell(t) g_\ell
\end{align}
and
\begin{align}
\label{eq:path5}
\tilde W(t)=\sum_{\ell\in \Z} \lambda_\ell \tilde \beta_\ell(t) g_\ell, 
\end{align}
respectively.
Here, $(\tilde\beta_\ell^h)_{\ell\in\Z}$ and
$(\tilde \beta_\ell)_{\ell\in \Z}$ are families of i.\,i.\,d.\ Brownian motions with respect to $(\tilde {\mathcal F}^h_t)_{t\geq 0}$ and $(\tilde {\mathcal{F}}_t)_{t\geq 0}$, respectively.
\end{lemma}

\subsection{Convergence of the deterministic terms}\label{convdeter}
In this subsection, we prove higher regularity of $\tu$ as well as the identification of $\tilde{w}=\tilde{u}_{x}$ and $\tilde{z}=\tilde{u}_{xx}$, the pseudo-flux $\tilde J$
\begin{align}
\tJ=\tu^{\frac{n}{2}}\tp_x\label{eq:ident1} \, ,
\end{align}
and the identification of the pressure $\tp$
\begin{align}
\tp=-\tu_{xx}+F'(\tu).\label{eq:ident2}
\end{align}
For the ease of presentation, let us collect the convergence and boundedness results established so far:
\begin{align}
\tuh &\rightarrow\tu &&\text{ in }\uspace~\tprobab\text{-almost surely},
\label{eq:ident3}
\\
\tilde{w}^{h}&=\tilde{u}^{h}_{x} \rightharpoonup \tilde{w}&&\text{ weakly in } \Jspace~\tprobab\text{-almost surely}
\label{eq:ident3a}
\\
\tilde{z}^{h}&=\Delta_{h}\tuh \rightharpoonup \tilde{z}&&\text{ weakly in } \Jspace~\tprobab\text{-almost surely}
\label{eq:ident3b}
\\
\tilde p^h &\rightharpoonup \tilde p &&\text{ weakly in }\pspace~\tprobab\text{-almost surely},
\\
\tJh&=\chi_{\tilde T_h} M_h(\tuh)^\frac{1}{2}\tp^h_x\rightharpoonup \tilde{J}&&\text{ weakly in } \Jspace~\tprobab\text{-almost surely},
\label{eq:ident4}
\end{align}
\begin{align}
&\E\left[\sup_{t\in [0,T_{max}]}\left(\intort{|\tu^h_x|^2(t)\,dx}+\intort{\Ih [F(\tuh)(t)]}\,dx\right)^{\overline{p}}\right]\leq C(\overline{p},u_0)<\infty\text{ for every } \overline{p}\geq 1,
\label{eq:ident5}
\\
&\E\left[\sup_{t\in [0,T_{max}]}\left(\intort{\tilde u^h(t) \,dx}\right)^{\overline{p}}+\sup_{t\in [0,T_{max}]} \left(\intort{\tilde u^h(t) \,dx}\right)^{-\overline{p}}\right]\leq C(\overline{p},u_0)<\infty\text{ for every } \overline{p}\geq 1,
\label{eq:ident5b}
\\
&\E \Bigg[ \int_0^{T_{max}}\norm {\Delta_h\tuh}{h}^2dt \Bigg]\leq C(u_0)<\infty,\label{eq:ident6}
\\
&\E \Bigg[\int_0^{T_{max}}\int_\domain |\tilde J^h|^2 \,dx\,dt \Bigg]\leq C(u_0)<\infty,\label{eq:ident7}
\\
&\E\Bigg[\bigg(\int_0^{T_{max}} \int_\domain |\tilde p^h_x|^2 + |\tilde p^h|^2 \,dx \,dt\bigg)^{3/4} \Bigg] \leq C(u_0)<\infty,
\label{eq:ident8}
\end{align}
where $\Delta_h\tuh$ satisfies the identity 
\begin{align}
\Delta_h\tuh = \partial_h^{+}(\partial_h^{-}\tuh)=-\tph+\Ih [F'(\tuh)] \label{eq:ident8a}
\end{align}
for $t\in[0,\tilde{T}_h]$, cf. \eqref{eq:2}. 
Moreover, we have for every $\bar p\geq 1$ the estimates
\begin{equation}
\label{eq:nDelta}
\E\Bigg[\bigg(\int_0^{\tilde T_h}\discint (\tui)^{n-2}|(\Delta_h \tuh)_i|^2ds\bigg)^{\bar p}\Bigg] \leq C(\bar p, u_0)
\end{equation}
and
\begin{equation}
\label{eq:nNabla}
\E\Bigg[\bigg(\int_0^{\tilde T_h}\discint (\tui)^{n-4}\left|\tfrac{\tuip-\tui}{h}\right|^4ds\bigg)^{\bar p}\Bigg]\leq C(\bar p, u_0).
\end{equation}

Similarly, as in \cite{FischerGruen2018} Lemma~5.8, we have 
\begin{lemma}
\label{LowerBound}
We have almost surely
\begin{align*}
\inf_{x\in \domain, t\in [0,T_{max}]} \tilde u(x,t)>0.
\end{align*}
\end{lemma}
We have the following result.
\begin{lemma}\label{lem:identificationJandp}
For the limits $\tilde{w}$, $\tilde{z}$, $\tilde J$, and $\tilde p$, we have the identification
\begin{align*}
\tilde{w} = \tilde{u}_{x} \, 
\end{align*}
\begin{align*}
	\tilde{z} = \tilde{u}_{xx} \, 
\end{align*}
\begin{align*}
\tilde J = \tilde u^{n/2} \tilde p_x \,
\end{align*}
and
\begin{align*}
\tilde p = - \tilde u_{xx} + F'(\tilde u)
\end{align*}
pointwise a.\,e.\ almost surely. Furthermore, we have $\tilde u_{xxx}\in L^2(\domain \times [0,T_{max}])$ almost surely.
\end{lemma}

\begin{proof}
	Due to Lemma~\ref{lem:conv2} and integration by parts, we find $\tilde{\mathbb{P}}$-almost surely for any $\phi\in C^\infty(\domain)$
	\begin{align*}
		\int_0^{T_{max}} \int_\domain \tilde{w}^{h} \phi \,dx\,dt 
		=
		\int_0^{T_{max}} \int_\domain \tuh_{x} \phi \,dx\,dt
		&\rightarrow
		\int_0^{T_{max}} \int_\domain \tu_{x} \phi \,dx\,dt
	\end{align*}
	which gives the identification of $\tilde{w}$.
	Regarding $\tilde{z}$ we note
	\begin{align*}
		\mathbb{E}\bigg[
		\bigg|\int_0^{T_{max}} \int_\domain \tilde{z}^h 	\phi^{h} \,dx\,dt
		-\int_0^{T_{max}} \int_\domain  	\Delta_{h}\tuh \phi^{h} \,dx\,dt \bigg|
		\bigg]=0
	\end{align*}
	for any $\phi^{h} \in X_{h}$ due to \eqref{eq:path1}.
	By Proposition~\ref{prop:conv1} we have for periodic $\phi\in C^\infty(\domain)$ and $\phi^h:=\mathcal{I}_h[\phi]\in X_h$
	\begin{align*}
		\int_0^{T_{max}} \int_\domain \tilde{z}^{h} \phi^{h} \,dx\,dt
		&\rightarrow
		\int_0^{T_{max}} \int_\domain \tilde{z} \phi\,dx\,dt \, .
	\end{align*}
	For the discrete Laplacian, we have for $\phi^{h}$ as above
	\begin{align*}
		\int_0^{T_{max}} (\Delta_{h}\tilde{u}^h, \phi^{h})_{h} \,dt
		&=
		- \int_0^{T_{max}} \intort \tuh_{x} \phi_{x}^{h} \,dx\,dt
		\\
		&\rightarrow
		-\int_0^{T_{max}} \int_\domain \tilde{u}_{x} \phi_{x} \,dx\,dt
		= 
		\int_0^{T_{max}} \int_\domain \tilde{u}_{xx} \phi \,dx\,dt.
	\end{align*}
	due to the strong convergence of $\phi_{x}^{h} \to \phi_{x}$ in $L^{\infty}(\ort)$. Moreover,
	\begin{align*}
		&\left|\int_{0}^{T_{max}} (\Delta_{h}\tilde{u}^h,\phi^h)_h \,dt
		-\int_{0}^{T_{max}} \int_\domain \Delta_{h}\tilde{u}^h \phi^{h} \,dx \,dt\right|
		\\
		&\le 
		\int_{0}^{T_{max}} h ||\phi^{h}_{x}||_{L^{2}(\ort)} ||\Delta_{h}\tuh||_{L^{2}(\ort)} \, ds 
		 \to 0 
	\end{align*}
	for $h \to 0$. Combined with Fatou's lemma, this provides 
	\begin{align*}
		\mathbb{E}\bigg[
		\bigg|\int_0^{T_{max}} \int_\domain \tilde z \phi \,dx\,dt
		-\int_0^{\tilde T_h} \int_\domain \tu_{xx} \phi \,dx\,dt \bigg|
		\bigg]=0
	\end{align*}
     for any $\phi\in C^\infty(\domain\times [0,T_{max}])$ which gives the identification of $\tilde{z}$.
	The remaining assertions follow in the same way as in Lemma~5.9 of \cite{FischerGruen2018}.
\end{proof}

\subsection{Convergence of the stochastic integral and proof of the main result} \label{sec:stochint}
Consider for $v\in \Hper{2}$  arbitrary, but fixed, the operator $\Mh:\Omega\times [0,T_{max}]\rightarrow\R$ defined by
\begin{align}
\Mh(t):=&(\uh(t)-u^h_0,\Ph v)_h+\int_0^{t\land T_h} \intort{(M_h(\uh))^{1/2} J^{h} (\Ph v)_x}\,dx\,ds \notag
\\ \notag
& + (C_{Strat}+S) \left(\tThInt\porad(u^h, \Ph v)\, ds+\tThInt \porb(u^h, \Ph v) \, ds\right)
\\
=&~\sum_{|\ell|\le N_h}\int_0^{t\land T_h}\intort{\left( \sqmob(\uh)\lambda_\ell g_\ell \right)_x \Ph v }\,dx\,d\beta_l(s). \label{eq:stident1}
\end{align}
Here, $\Ph:\Hper{1}\rightarrow X_h$ is a projection operator satisfying
\begin{align}
\lim_{h\rightarrow 0} \norm{\Ph v-v}{H^1_{per}}=&0\label{eq:stident2}
\end{align}
for all $v\in H^1_{per}(\domain)$. Observe that by the optional stopping theorem, $\Mh$ is a real valued martingale; that is, denoting by $r_s$ the restriction of a function on $[0,T_{max}]$ onto $[0,s]$, we have
\begin{align}
\E\left([\Mh(t)-\Mh(s)]\Psi(r_s\uh,r_sW)\right)=&0\label{eq:stident3}
\end{align}
for all $0\leq s\leq t\leq T_\text{max}$ and for all $[0,1]$-valued continuous functions $\Psi$ defined on $C^{\gamma,\gamma/4}(\domain\times [0,s])\times C([0,s];L^2(\domain))$.
\begin{lemma}\label{lem:stochident1}
For the quadratic variation of $\Mh$, we have
\begin{align}
&\crossl \Mh\crossr_t=\int_0^{t\land T_h}\sum_{|\ell|\le N_h}\lambda_\ell^2 \left( \intort  (\sqmob(\uh) g_\ell)_x\Ph v \,dx \right)^2\,ds
\notag
\\
&\leq C\chi_{T_h}(t) \norm{v}{H^1_{per}}^2T_{max} \Bigg\{\sup_{\tau\in[0,t\wedge T_h]}\norm{(u^h(\tau))_x}{L^2(\ort)}^2+\sup_{\tau\in[0,t\wedge T_h]}\left(\int_\ort \uh(\tau) \, dx\right)^{\tfrac{n+2}{4-n}}+1\Bigg\}.
\label{eq:stident4}
\end{align}
\end{lemma}
\begin{proof}
Consider $R(\uh,v):\Omega\times [0,T_{max}]\times L^2(\domain)\rightarrow\R$ defined by
\begin{align*}
(\omega,t,z)\mapsto \chitth (t,\omega) \intort{ (\sqmob(\uh) \mathcal P_{N_h}z)_x\Ph v}\,dx(t,\omega),
\end{align*}
where $\mathcal P_{N_h}$ denotes the orthogonal $L^2$-projection onto $\operatorname{span} \set{g_1,\cdots,g_{N_h}}$. For the Hilbert-Schmidt norm, we get using $||g_\ell||_{L^{\infty}(\cO)}\leq C(L)$ and $\sum_{\ell \in \Z} \lambda_\ell^2 < \infty$
\begin{align*}
	&\norm{R(\uh,v)(t,\omega)}{L_2(\WienerProcessSpace;\R)}^2=\chitth(t)\sum_{|\ell|\le N_h}\lambda_\ell^2 \left(\intort{(\sqmob(\uh) g_\ell)\partial_x\Ph v}\,dx\right)^2
	\\
	&\leq C\chitth(t) \norm{v}{H^1_{per}}^2\int_\ort\left(\sqmob(u^h)\right)^2dx\\
	&\leq C\chitth(t) \norm{v}{H^1_{per}}^2\Bigg\{\int_\ort|u^h_x|^2 \, dx+\left(\int_\ort u_h \, dx\right)^{\tfrac{n+2}{4-n}}+1\Bigg\}, 
\end{align*}
with the last step being based on equivalence of the lumped-masses scalar product and Lemma~\ref{lem:uHochN}.
By Lemma~2.4.3 in \cite{PreRoe}, the result is obtained.
\end{proof}
\begin{remark}
From \eqref{eq:proposition1}, \eqref{eq:stident4}, and our assumptions on initial data and stopping times, we infer that $\Mh$ is a square-integrable martingal.
\end{remark}
Similarly, $\mathcal{M}_{h,v}^2-\int_0^{(\cdot)\land T_h}\sum_{|\ell|\le N_h}\lambda_\ell^2 \left( \intort  (\sqmob(\uh) g_\ell)_x\Ph v \,dx \right)^2\,ds$ is a martingale. For the identification of the stochastic integral in the limit $h\rightarrow 0$, we will  study the processes
\begin{align}
\beta_\ell(t)=&\intort \int_0^t \frac{1}{\lambda_\ell}g_\ell \,dW \,dx \label{eq:stident5}
\end{align}
and their cross variations with $\Mh$.
By the same argumentation as in Lemma~5.12 of \cite{FischerGruen2018}, we get the following result.
\begin{lemma}\label{lem:stochident2}
For $\ell\in\N$ the cross variation $\crossl\Mh,\beta_\ell\crossr_t$ is given by the formula
\begin{align}
\crossl\Mh,\beta_\ell\crossr_t=&
\begin{cases}
	\lambda_\ell\int_0^{t\land T_h}  \intort{(\sqmob(\uh) g_\ell)_x\Ph v} \,dx\,ds, &\ell\leq N_h,
	\\
	0,&\ell>N_h.
\end{cases}\label{eq:stident6}
\end{align}
\end{lemma}

In particular, $\Mh \beta_\ell - \lambda_\ell\int_0^{(\cdot)\land T_h}  \intort{(\sqmob(\uh) g_\ell)_x\Ph v} \,dx\,ds$ for $\ell\leq N_h$ and $\Mh \beta_\ell$ for $\ell>N_h$ are martingales, too.

By equality of laws, we deduce that
\begin{align}
\label{eq:stident7}
\begin{split}
\tMh(t):=&(\tuh(t)-\tuh(0),\Ph v)_h
+\int_0^{t\land \tilde T_h}\intort{(M_h(\tuh))^{1/2} \tilde{J}^{h} (\Ph v)_x }\,dx\,ds\\
& +(C_{Strat}+S)\left(\int_0^{t\wedge \tilde T_h}\porad(\tuh,\Ph v) \, ds+\int_0^{t\wedge \tilde T_h}\porb(\tuh,\Ph v) \, ds \right),
\end{split}
\end{align}
\begin{gather}
\tMhq (t)
-\int_0^{t\land \tilde T_h}\sum_{|\ell|\le N_h}\lambda_\ell^2 \left( \intort  (\sqmob(\tuh) g_\ell)_x\Ph v \,dx \right)^2\,ds,
\label{eq:stident71}
\\
\label{eq:stident72}
\tMh (t)\tilde \beta^h_\ell(t)
-\lambda_\ell\int_0^{t\land \tilde T_h} \intort{(\sqmob(\tuh) g_\ell)_x\Ph v} \,dx\,ds
\quad\quad \text{for }\ell\leq N_h,
\\
\tMh (t)\tilde \beta^h_\ell(t)\quad\quad \text{for }\ell>N_h,
\end{gather}
are $(\tilde{\mathcal{F}_t^h})$-martingales, where $\tilde\beta^h_\ell(t):=\int_\ort \int_0^t\lambda_\ell^{-1} g_\ell d\tilde W^h \,dx$. In particular,
\begin{align}
\crossl \tMh \crossr_t =&\int_0^{t\land \tilde T_h}\sum_{|\ell|\le N_h}\lambda^2_\ell\left(\intort{ \sqmob(\tuh) g_\ell (\Ph v)_x }\,dx \right)^2 \,ds\label{eq:stident8}
\end{align}
and
\begin{equation}
  \crossl\tMh,{\tilde \beta^h}_\ell\crossr_t=\begin{cases}
  \lambda_\ell \int_0^{t\land \tilde T_h} \intort{(\sqmob(\tuh) g_\ell)_x \Ph v }\,dx \,ds & \text{ if } \ell\leq N_h
  \\
  0 & \text{ if } \ell> N_h.
  \end{cases}
  \label{eq:stident9}
\end{equation}

Starting point for the passage to the limit $h\rightarrow 0$ are the identities
{\footnotesize{
\begin{gather}
\E ((\tMh(t)-\tMh(s))\Psi(r_s\tuh,r_s\tilde W^h))=0,\label{eq:stident10}
\\
\E \left(\left( \tMhq(t)-\tMhq(s) -\int_{s\land \tilde T_h}^{t\land \tilde T_h}\sum_{|\ell|\le N_h}\lambda_\ell^2 \left( \intort (\sqmob(\tuh) g_\ell)(\Ph v)_x \,dx\right)^2 d\tau\right)\Psi(r_s\tuh,r_s\tilde W^h)\right)=0 \label{eq:stident11}
\end{gather}
}}
and
{\footnotesize{\begin{equation}
\E\left( \left(\tMh(t)\tilde \beta^h_\ell(t)-\tMh(s)\tilde \beta^h_\ell(s)-\int_{s\land \tilde T_h}^{t\land \tilde T_h}\lambda_\ell \intort (\sqmob(\tuh) g_\ell)_x \Ph v \,dx \,d\tau\right)\Psi(r_s\tuh,r_s\tilde W^h)\right)=0\label{eq_stident12}
\end{equation}}}
for all $s\leq t\in [0,T_{\mathrm{max}}]$ and for all $[0,1]$-valued continuous functions $\Psi$ defined on $C^{\gamma,\gamma/4}(\domain\times [0,s])\times C([0,s];L^2(\domain))$.

Let us pass to the limit in equation~\eqref{eq:stident10}.
\begin{lemma}\label{lem:stochconv1}
  For all $[0,1]$-valued continuous functions $\Psi$ defined on $C^{\gamma,\gamma/4}(\domain\times [0,s])\times C([0,s];L^2(\domain))$, we have
{\footnotesize{
\begin{align} \notag
& \E\bigg[\bigg( \intort (\tu(t)-\tu(s))v \,dx +\int_s^t\intort m(\tu)\tp_x v_x \,dx\,d\tau\\
& - (C_{Strat}+S)\left((n-2)\int_s^t\int_\ort\tu^{n-3}\tu_x^2v\,dx\,d\tau+\int_s^t\int_\ort \tu^{n-2}\tu_{xx}v \, dx \, d\tau\right)\bigg)\Psi(r_s\tu,r_s\tilde W)\bigg]=0\label{eq:stident13}
\end{align}}}
for all $0\leq s\leq t< T_{\mathrm{max}}$.
\end{lemma}
\begin{proof}
Note that due to Lemma~\ref{lem:subsequence}, we have $\chi_{\tilde T_h}=1$ on $[s,t]$ for $h$ sufficiently small depending on $\omega$.
 By definition \eqref{eq:stident7}, we first discuss the term $(\tuh(t)-\tuh(s),\Ph v)_h$. By the strong convergence of $\tuh$ in $\uspace~ \tprobab$-almost surely and of $\Ph v$ towards $v$ in $L^2(\domain)$, we readily identify
\begin{align}\label{eq:stident131}
\lim_{h\rightarrow 0}(\tuh(t)-\tuh(s),\Ph v)_h=\intort (\tu(t)-\tu(s))v\,dx.
\end{align}
In addition, $\Psi [r_s\tuh, r_s\tilde W^h]$ converges $\tprobab$-almost surely to $\Psi[r_s\tu,r_s\tilde W]$ in $\R$ by continuity of $\Psi$, the $\uspace$-convergence of $\tuh\rightarrow \tu$, and the $C([0,T_{max}];L^2(\ort))$-convergence of $\tilde W^h$.

To discuss the convergence behavior of
\begin{align}
\int_{s\land \tilde T_h}^{t\land \tilde T_h}\intort (M_h(\tuh))^{1/2} \tilde{J}^{h} (\Ph v)_x
\,dx\,d\tau\,
\Psi (r_s\tuh,r_s\tilde W^h),
\end{align}
we proceed as follows.

Due to \eqref{eq:ident3}, $M_{h}(\tuh)$ converges to $\tu^{n}$ in $L^\infty (\ort\times [0,T_{max}])~\tprobab$-almost surely. By \eqref{eq:ident5}, \eqref{eq:ident5b},
and Poincar\'{e}'s inequality, we have 
$\mathbb{E} [ |\sup_{(t,x)\in [0,T_{max}\times \ort]}\tuh|^{2\bar{p}}] < \infty$. 
By Vitali's theorem, 
\begin{align}
||M_h(\tuh)||_{L^{\infty}(\ort \times I)} \rightarrow ||\tu^n||_{L^{\infty}(\ort \times I)}\text{ in } L^{\overline{p}}(\tilde\Omega).
\label{eq:stident14}
\end{align}
Slightly modifying the corresponding results in \cite{FischerGruen2018}, p. 449, we convince ourselves that
\begin{align}
\sqrt{M_h(\tuh)} (\Ph v)_x \Psi (r_s\tuh,r_s\tilde W^h)\rightarrow
\tu^{n/2} v_x \Psi(r_s\tu,r_s \tilde W)
\label{eq:stident15}
\end{align}
strongly in $L^2(\tilde\Omega\times\domain\times [0,T_{max}])$. 

By the weak convergence of $\tilde{J}^{h}$ towards $\tJ$ in $L^2(\domain\times [0,T_{max}])$ $\tprobab$-almost surely, cf. \eqref{eq:ident4}, we infer
\begin{align}
&
\lim_{h\rightarrow 0}~ \int_{s\land \tilde T_h}^{t\land \tilde T_h}\intort \sqrt{M_h(\tuh)}\tilde{J}^{h} (\Ph v)_x \,dx\,d\tau \,\Psi(r_s\tuh,r_s\tilde W^h)\notag
\\
&=  \int_s^t\intort \tu^{n/2} v_x \tJ \,dx\,d\tau \, \Psi (r_s\tu,r_s\tilde W)\label{eq:stident16}
\end{align}
$\tprobab$-almost surely.
We note that the estimate 
$
	M_h(\tuh) \le C \left(\big|\sup_{[0,T_{max}]\times \ort} \tuh \big|^{n} + 1	\right) \, 
$
and the control of any moment of $\sup_{[0,T_{max}]\times \ort} \tuh$, cf. \eqref{eq:ident5b}, implies a uniform bound on moments of $\sup_{[0,T_{max}]\times \ort} M_h(\tuh)$.
\color{black}
Setting 
\begin{align*}
&B_h(\tuh,\tph,v,s,t):=\int_{s\land \tilde T_h}^{t\land \tilde T_h}\intort
\sqrt{M_h(\tuh)}\tilde{J}^{h} (\Ph v)_x \,dx\,d\tau \,\Psi(r_s\tuh,r_s\tilde W^h)
\end{align*}
and using the estimate
\begin{align*}
&|B_h(\tuh,\tph,v,s,t)|
\\
&\leq \left(\int_{s\land \tilde T_h}^{t\land \tilde T_h}\intort|\tilde{J}^{h}|^2\,dx\,d\tau\right)^{1/2}\left(\sup_{\ort\times [0,T_{max}]}M_h(\tuh)\right)^{1/2}\left(\int_0^{T_{max}}\int_{\ort}|(\Ph v)_x|^2\,dx\,d\tau\right)^{1/2}
\end{align*}
as well as 
\eqref{eq:proposition1}, we infer uniform integrability of a $q$-moment of $B_h(\tuh,\tph,v,s,t)$ for a number $q>1.$ Vitali's theorem then entails
\begin{align}\label{eq:stident16a}
\lim_{h\to 0}\expect\left(B_h(\tuh,\tph,v,s,t)\right) = 
\E\left( \int_s^t\intort \tu^{n/2} v_x \tJ \,dx\,d\tau \, \Psi (r_s\tu,r_s\tilde W)\right).
\end{align} 
By Lemma~\ref{lem:identificationJandp},  
we get, in particular, $\tJ=\tu^{n/2}\tp_x$. 
Let us now discuss the convergence of the correction term.
First, we show the convergence of 
\begin{align}\label{eq:convAh}
\int_{0}^{t} A_{\Delta}^{h}(\tuh, \Ph v ) \, ds \rightarrow -(n-2) \int_{0}^{t} \intort (\tu)^{n-3}(\tu_{x})^{2}v \, \,dx \,ds
\end{align} 
for $t \in [0,T_{max}]$.
According to the definition of $A_{\Delta}^{h}$, cf. \eqref{eq:definePorad}, and taking $(\tuip - \tuim ) = (\tuip - \tui)+(\tui - \tuim)$ into account, we see that it is sufficient to discuss terms of the form 
$$h \sum_{i=1}^{L_{h}} (\tuh_{i + k_{0}})^{n-3} \left(\tfrac{\tuip-\tui}{h}\right)\left(\tfrac{\tui-\tuim}{h}\right) \Ph v(ih)$$
for $k_{0} \in \{0,1,2,3\}$. We find
\begin{align}\notag
	&\left|\int_{0}^{t} \discint (\tuh_{i + k_{0}})^{n-3} \left(\tfrac{\tuip-\tui}{h}\right)\left(\tfrac{\tui-\tuim}{h}\right) \Ph v(ih) ds - \int_{0}^{t}\intort \tu^{n-3} (\tu_{x})^{2} v\,  \,dx \,ds \right|
	\\ \notag
	&=
	\left|\int_{0}^{t} \sum_{i=1}^{L_{h}} \int_{ih}^{(i+1)h} (\tuh_{i + k_{0}})^{n-3}  \tuh_{x} \, \tuh_{x}(\cdot -h) \Ph v(ih) dx  ds - \int_{0}^{t}\intort \tu^{n-3} (\tu_{x})^{2} v\,  \,dx \,ds \right|
	\\ \notag
	&\le
	\left|\int_{0}^{t} \sum_{i=1}^{L_{h}} \int_{ih}^{(i+1)h}
	\left((\tuh_{i + k_{0}})^{n-3} - \tu^{n-3}(x) \right)
	 \tuh_{x} \, \tuh_{x}(\cdot -h) \Ph v(ih) \,dx \,ds \right|
	\\ \notag
	&\quad+
	\left|\int_{0}^{t} \sum_{i=1}^{L_{h}} \int_{ih}^{(i+1)h} \tu^{n-3}(x)
	 \left(\tuh_{x} - \tu_{x}\right) \tuh_{x}(\cdot -h) \Ph v(ih) \,dx \,ds\right|
	 \\ \notag
	 &\quad+
	 \left|\int_{0}^{t} \sum_{i=1}^{L_{h}} \int_{ih}^{(i+1)h} \tu^{n-3}(x)
	 \tu_{x} \left(\tuh_{x}(\cdot -h) - \tu_{x} \right) \Ph v(ih) \,dx \,ds \right|
	\\ \notag
	&\quad+
	\left|\int_{0}^{t} \sum_{i=1}^{L_{h}} \int_{ih}^{(i+1)h} \tu^{n-3}  \left(\tu_{x}\right)^{2} \left(\Ph v(ih) -v(x) \right) \,dx \,ds \right|
	\\ \notag
	&=: 
	I_{1}+I_{2}+I_{3}+I_{4} \, .
\end{align}
Due to the positivity almost surely of $\tuh$ and $\tu$ and the uniform convergence $\tuh \rightarrow \tu$, also $(\tuh)^{n-3} \rightarrow \tu^{n-3}$ uniformly. Moreover, $\Ph v$ is uniformly bounded in  $L^{\infty}(\ort)$. Hence, we find the following convergence for $h \rightarrow 0$ almost surely, respectively.
\begin{align}\notag
	|I_{1}| \le &|| (\tuh)^{n-3}(\cdot + k_{0}h) - \tu^{n-3}||_{L^{\infty}([0,T] \times \ort)} ||\Ph v ||_{L^{\infty}(\ort)} 
	\\
	&\cdot \left(\int_{0}^{T_{max}}\intort|\tu_{x}|^{2} \,dx \,ds \right)^{1/2}
	\left(\int_{0}^{T_{max}}\intort|\tu_{x}(\cdot -h)|^{2} \,dx \,ds \right)^{1/2}
	\rightarrow 0 \, .
\end{align}
Using Lemma~\ref{lem:nablastrong}, we have
\begin{align}\notag
|I_{2}| \le &||\tu^{n-3}||_{L^{\infty}([0,T] \times \ort)} ||\Ph v ||_{L^{\infty}(\ort)}
\\ &\cdot  \left(\int_{0}^{T_{max}}\intort|\tu_{x} - \tu_{x}|^{2} \,dx \,ds\right)^{1/2}
\left(\int_{0}^{T_{max}}\intort|\tu_{x}(\cdot -h)|^{2} \,dx \,ds \right)^{1/2}
\rightarrow 0.
\end{align}
The Frechet-Kolmogorov theorem shows
\begin{align}\notag
|I_{3}| \le &||\tu^{n-3}||_{L^{\infty}([0,T] \times \ort)} ||\Ph v ||_{L^{\infty}(\ort)} 
\\ &\cdot 
\left(\int_{0}^{T_{max}}\intort|\tu_{x} |^{2} \,dx \,ds \right)^{1/2}
\left(\int_{0}^{T_{max}}\intort|\tu_{x}(\cdot -h)-\tu_{x}|^{2} \,dx \,ds \right)^{1/2}
\rightarrow 0,
\end{align}
and finally by the assumptions on $\Ph$ and $v$ 
\begin{align}
|I_{4}| \le ||\tu^{n-3}||_{L^{\infty}([0,T] \times \ort)} ||\Ph v -v ||_{L^{\infty}(\ort)} \left(\int_{0}^{T_{max}}\intort|\tu_{x} |^{2} \,dx \,ds \right)
\rightarrow 0
\end{align}
holds.
This shows \eqref{eq:convAh}. Let us show now that $\tilde{\mathbb{P}}$-almost surely
\begin{align}\label{eq:convBh}
	\int_{0}^{t} B_{\Delta}^{h}(\tuh,\Ph v)\, ds \rightarrow - \int_{0}^{t}\intort \tu^{n-2} \Delta \tu v  \,dx \,ds \, . 
\end{align}
Writing $B_{\Delta}^{h}(\tuh, \Ph v)$ as $-\left((\tuh)^{n-2} \Delta_{h} \tuh, \Ph v\right)_{h}$, we find
\begin{align}\notag
	&\left|\int_{0}^{t} B_{\Delta}^{h}(\tuh, \Ph v) ds - \int_{0}^{t} \intort \tu^{n-2} \Delta \tu \, v \,dx \,ds \right|
	\\ \notag 
	&\le 
	\left| - \int_{0}^{T_{max}}\left((\tuh)^{n-2} \Delta_{h} \tuh, \Ph v\right)_{h} ds + \int_{0}^{T_{max}} \intort (\tuh)^{n-2} \Delta_{h}\tuh \Ph \, v \,dx \,ds\right|
	\\\notag
	&\quad +
	\left| - \int_{0}^{T_{max}} \intort (\tuh)^{n-2} \Delta_{h}\tuh \Ph \, v \,dx \,ds + \int_{0}^{T_{max}} \intort (\tu)^{n-2} \Delta\tu \, v \,dx \,ds\right|
	\\
	&=: R_{1} + R_{2}  \,.
\end{align}

ad $R_{1}$:
\begin{align}\notag
	&|R_{1}| \le h \cdot \int_{0}^{T_{max}} ||(\tuh)^{n-2} \Delta_{h} \tuh ||_{L^{2}(\ort)} || \left(\Ph v\right)_{x} ||_{L^{2}(\ort)}\, ds
	\\ \notag 
	&\le
	h \cdot \int_{0}^{T_{max}} ||(\tuh)^{\frac{n-2}{2}}||_{L^{\infty}([0,T]\times \ort)} ||(\tuh)^{\frac{n-2}{2}} \Delta_{h} \tuh ||_{L^{2}(\ort)} || \left(\Ph v\right)_{x} ||_{L^{2}(\ort)}\, ds
	\\ \notag 
	&\le
	h\, C(T_{max})   ||(\tuh)^{n-2}||_{L^{\infty}([0,T]\times \ort)} \left(\int_{0}^{T_{max}} \intort(\tuh)^{n-2} |\Delta_{h} \tuh|^{2} \,dx \,ds\right)^{1/2} || \left(\Ph v\right)_{x} ||_{L^{2}(\ort)} 
	\\ \notag 
	&\rightarrow 0 
\end{align}
by the uniform bounds provided in Proposition~\ref{prop:integral1}.

Ad $R_{2}$:
The weak convergence of $\Delta_{h}\tuh \rightarrow \Delta \tu$ in $L^{2}([0,T_{max}] \times \ort)$ together with the uniform convergence $(\tuh)^{n-2} \rightarrow \tu^{n-2}$  shows $R_{2} \rightarrow 0$, $\tilde{\mathbb{P}}$-almost surely respectively.
Using additionally the arguments above regarding convergence of $\Psi$, we see
\begin{align}
	\int_{s\wedge T_{h}}^{t\wedge T_{h}} A_{\Delta}^{h}(\tuh, \Ph v )\, d\tau \, \Psi (r_s\tuh,r_s\tilde W^h)
	\rightarrow
	 -(n-2) \int_{s}^{t} \intort (\tu)^{n-3}(\tu_{x})^{2}v\,dx\,d\tau \, \Psi (r_s\tuh,r_s\tilde W^h)
\end{align}
and
\begin{align}
	\int_{s\wedge T_{h}}^{t \wedge T_{h}} B_{\Delta}^{h}(\tuh,\Ph v)\, d\tau \,\Psi (r_s\tuh,r_s\tilde W^h)
	\rightarrow 
	- \int_{s}^{t}\intort \tu^{n-2} \Delta \tu \,v \, dx\,d\tau \, \Psi (r_s\tuh,r_s\tilde W^h) \, 
\end{align}
$\tilde{\mathbb{P}}$-almost surely.
To find a bound of appropriate moments, we use Proposition~\ref{prop:integral1} and the Oscillation Lemma~\ref{lem:lowerbound} to obtain 
\begin{align}\notag
	&\left| \int_{s\wedge T_{h}}^{t\wedge T_{h}} A_{\Delta}^{h}(\tuh, \Ph v ) \, d\tau \, \Psi (r_s\tuh,r_s\tilde W^h) \right|
	\\ 
	&\le
	\left|\int_{s\wedge T_{h}}^{t\wedge T_{h}} \discint (\tuh)^{n-4} \left|\tfrac{\tuip-\tui}{h}\right|^{4} \, d\tau\right|^{1/2}
	\left|\int_{s\wedge T_{h}}^{t\wedge T_{h}} \intort (\tuh)^{n-2} \,dx\, d\tau\right|^{1/2}
	||\Ph v||_{L^{\infty}(\ort)} \, 
\end{align}
and
\begin{align}\notag
	&\left|\int_{s\wedge T_{h}}^{t \wedge T_{h}} B_{\Delta}^{h}(\tuh,\Ph v) \, d\tau \, \Psi (r_s\tuh,r_s\tilde W^h)\right|
	\\  
	&\le
	\left|\int_{s\wedge T_{h}}^{t\wedge T_{h}} \discint (\tui)^{n-2} \left|\Delta_{h}\tui\right|^{2} \, d\tau\right|^{1/2}
	\left|\int_{s\wedge T_{h}}^{t\wedge T_{h}} \intort (\tuh)^{n-2} \,dx \, d\tau\right|^{1/2} ||\Ph v||_{L^{\infty}(\ort)} \,.
\end{align}
Together with $n \in (2,3)$ and the energy-entropy estimate the integrability of a $q$-th absolute moment for $q>1$ follows. Thus, we may apply Vitali's theorem to derive the desired convergence properties for
\begin{align}
	\mathbb{E} \Big[\int_{s\wedge T_{h}}^{t\wedge T_{h}} A_{\Delta}^{h}(\tuh, \Ph v ) d\tau \, \Psi (r_s\tuh,r_s\tilde W^h) \Big]
\end{align}
and
\begin{align}
	\mathbb{E} \Big[\int_{s\wedge T_{h}}^{t \wedge T_{h}} B_{\Delta}^{h}(\tuh,\Ph v) d\tau\,\Psi (r_s\tuh,r_s\tilde W^h) \Big] \, .
\end{align}
Combining this with \eqref{eq:stident131} and \eqref{eq:stident16a}, the lemma is proven.
\end{proof}

By straightforward modifications in the proof of Lemma~5.14 in \cite{FischerGruen2018}, we get
\begin{lemma}\label{lem:stochconv2}
For all $[0,1]$-valued continuous functions $\Psi$ defined on $C^{\gamma,\gamma/4}(\domain\times [0,s])\times C([0,s];L^2(\domain))$, we have
\begin{align}
\E \left(\left( \tMq(t)-\tMq(s)-\int_s^t \sum_{|\ell|\le N_h} \lambda_\ell^2 \left(\intort \tu^{n/2} g_\ell v_x \,dx\right)^2 \,d\tau\right)\Psi(r_s\tu,r_s\tilde W)\right)=0
\label{eq:stident17}
\end{align}
for all $0\leq s\leq t<T_{\mathrm{max}}$, where
\begin{align*}
\tM(t):=\int_\ort \left(\tilde u(t)-\tilde u(0)\right)v \,dx +\int_0^t\int_\ort\tu^n\tilde p_x v_x \,dx \,d\tau.
\end{align*}
\end{lemma}

Furthermore, we have the following result on the cross-variation of $\tM$ and $\tilde \beta_\ell$, the proof of which we omit as it is similar to the preceding proofs.
\begin{lemma}\label{cor:stochconv3}
For all $[0,1]$-valued continuous functions $\Psi$ defined on $C^{\gamma,\gamma/4}(\domain\times [0,s])\times C([0,s];L^2(\domain))$, we have
\begin{align}
\E\left(\left(\tM(t)\tilde \beta_\ell(t)-\tM(s)\tilde \beta_\ell(s)-\int_s^t\lambda_\ell \intort (\tu^{n/2} g_\ell)_x v\,dx\,d\tau\right)\Psi(r_s\tu,r_s\tilde W)\right)=0 \label{eq:stident23}
\end{align}
for all $\ell\in \N$ and all $s\leq t\in [0,T_{\mathrm{max}})$.
\end{lemma}

Straightforward adaptation of Lemma~5.16 in \cite{FischerGruen2018} lead to the following result.
\begin{lemma}\label{lem:stochconv4}
We have 
\begin{align}
\tM(t)=\sum_{\ell \in \Z}\int_0^t\lambda_\ell\intort (\tu^{n/2} g_\ell)_x v\,dx\,d\tilde \beta_\ell.\label{eq:stident24}
\end{align}
\end{lemma}
It remains to establish Theorem~\ref{MainResult}.
\begin{proof}[Proof of Theorem~\ref{MainResult}]
From Proposition~\ref{prop:conv1}, Lemma~\ref{lem:noise-ident}, and Lemma~\ref{lem:identificationJandp}, we infer the existence of a stochastic basis $(\tilde\Omega, \tilde {\mathcal F}, (\tilde{\mathcal F}_t)_{t\geq 0}, \tprobab)$, of a Wiener process $\tilde W(t)=\sum_{|\ell|\le N_h} \lambda_\ell\tilde\beta_\ell(t)g_\ell$, and of random variables 
\begin{gather*}
	\tilde u\in L^2(\tilde\Omega; C^{\gamma, \gamma/4}(\ort\times [0,T_{max}])),
	\\
	\tilde p\in L^{3/2}(\tilde\Omega;L^2([0,T_{max}];H^1_{per}(\ort))),
	\\
	\tilde J\in L^2(\tilde\Omega\times\ort\times [0,T_{max}])
\end{gather*}
satisfying
\begin{align*}
	\tilde J&=\tu^{n/2}\tilde p_x \qquad &&\tprobab\text{-almost surely in } L^2(\ort\times [0,T_{max}]), \\
	\tilde p &=-\tu_{xx}+F'(\tu) \qquad &&\tprobab\text{-almost surely in } L^2(\ort\times [0,T_{max}]),\\
	\tu &\in L^2([0,T_{max}];H^3_{per}(\ort)) \qquad &&\tprobab\text{-almost surely.} 
\end{align*}
Furthermore, we have $\Lambda=\tprobab\circ \tu_0^{-1}$ by construction.
Lemma~\ref{lem:stochconv1} implies that for arbitrary $v\in  H^2_{per}(\ort)$ 
\begin{align*}
\tilde{\mathcal M}_v(t)=&\int_{\ort}(\tu(t)-\tu(0))v\,dx + \int_0^t\int_\ort \tu^n\tilde p_x v_x \,dx \,d\tau\\
&\quad - (C_{Strat}+S)\left((n-2)\int_0^t\int_\ort\tu^{n-3}\tu_x^2 v\,dx \,d\tau+\int_0^t\int_\ort\tu^{n-2}\tu_{xx} v \,dx \,d\tau\right)
\end{align*}
is an $(\tilde{\mathcal F}_t)_{t\geq 0}$-martingale. Due to Lemma~\ref{lem:stochconv4}, we conclude
\begin{align*}
	&\int_\ort \left(\tu(t)-\tu(0)\right)v \,dx
	+\int_0^t\int_\ort \tu^n \tilde p_x v_x \,dx \,d\tau\\
	&\qquad =(C_{Strat}+S)\left((n-2)\int_0^t\int_\ort\tu^{n-3}\tu_x^2 v\,dx \,d\tau+\int_0^t\int_\ort\tu^{n-2}\tu_{xx} v\,dx \,d\tau\right)\\
	&\qquad\qquad +\sum_{\ell\in \Z} \int_0^t\lambda_\ell\int_\ort (\tilde u^{n/2} g_\ell)_x v \,dx\,d\tilde\beta_\ell.
\end{align*}
Furthermore, Lemma~\ref{LowerBound} almost surely provides a positive lower bound for $\tilde u$.

Finally, by Fatou's lemma and Proposition~\ref{prop:integral1}, we have for any $\bar p\geq 1$
\begin{align*}
	&\mathbb{E} \left[\liminf_{h\rightarrow 0} \Big(\sup_{t\in [0,T_{max}]} E_h[\tu^h]^{\bar p}+ \left(\int_0^{T_{max}} \int_\ort |\tilde J^h|^2 \,dx\,dt\right)^{\bar{p}}\Big)\right]
	\\
	&\leq
	\liminf_{h\rightarrow 0}\mathbb{E} \left[\sup_{t\in [0,T_{max}]} E_h[\tu^h]^{\bar p}+ \left( \int_0^{T_{max}} \int_\ort |\tilde J^h|^2 \,dx\,dt\right)^{\bar{p}}\right]
	\\&
	\leq C(\bar p,u_0,T_{max}).
\end{align*}

By the almost-sure convergence of $\tu^h$ in $\uspace$ and the almost-sure strict positivity of the limit $\tu$, we deduce $\int_\ort \Ih[F(\tu^h)]\,dx \rightarrow \int_\ort F(\tu)\,dx$ in $L^\infty([0,T_{max}])$ almost surely. By the lower semicontinuity of the $L^2(\domain\times [0,T_{max}])$ norm with respect to weak convergence 
and the lower semicontinuity of the $L^\infty([0,T_{max}];H^1(\domain))$ norm with respect to convergence in the sense of distributions (for $\tu^h\rightarrow \tu$), we get  
\begin{align*}
&\mathbb{E} \left[\sup_{t\in [0,T_{max}]} E[\tu]^{\bar p}+\left(\int_0^{T_{max}} \int_\ort |\tilde J|^2 \,dx\,dt\right)^{\bar{p}}\right]
\leq C(\bar p,u_0,T_{max}).
\end{align*}
To use the same arguments as above regarding the other terms in estimate \eqref{mainresultest}, we need to derive estimates in suitable $L^{p}$-spaces. Let us for example consider the term $(\tuh)^{n-4} (\tuh)_{x}^{4}$. 
Using the positivity of $\tuh$ and convexity, we may estimate $\mathcal{I}_{h}[(\tuh)^{n-4}]\ge (\tuh)^{n-4} $ pointwise in $\ort$. Thus, from the a-priori estimates in Proposition~\ref{prop:integral1} and Lemma~\ref{LowerBound}, as well as the fact that $\tuh_{x}$ is constant on the intervals $(i-1,i)$, $i=1,\dots,L_{h}$, we deduce for any $q>1$
\begin{align}\notag \label{proofmainub}
	&\mathbb{E}\Bigg[ \left( \int_0^{T_{max}}\int_\domain  (\tuh)^{n-4} (\tuh)_{x}^{4} \, dx \, ds \right)^{q} \Bigg]
	\le
	\mathbb{E}\Bigg[ \left(\int_0^{t\wedge T_h} \int_\domain  \mathcal{I}_{h}[(\tuh)^{n-4}] (\tuh)_{x}^{4} \, dx \, ds \right)^{q} \Bigg]
	\\ \notag
	&=
	\mathbb{E}\Bigg[ \Bigg(\int_0^{t\wedge T_h} h \sum_{i=1}^{L_{h}} \frac{(\tuh_{i})^{n-4} + (\tuh_{i-1})^{n-4}}{2}  (\tuh)_{x}^{4}\vert_{((i-1)h,ih)} \, ds \Bigg)^{q}\Bigg]
	\\ 
	&\le
	C_{osc}
	\mathbb{E}\Bigg[ \bigg( \int_0^{t\wedge T_h} h \sum_{i=1}^{L_{h}} (\tuh_{i})^{n-4}  \bigg(\frac{\tuh_{i} - \tuh_{i-1}}{h}\bigg)^{4} \, ds \bigg)^{q} \Bigg] \le C(\bar{p}, u_{0}, T_{max}) 
\end{align}   
independently of $h$.
We deduce the weak convergence of $(\tuh)^{\frac{n-4}{4}} (\tuh)_{x}$ in $L^{4q}(\Omega ; L^{4}(\ort \times [0,T_{max}]))$ and identify the limit 
via the identity $(\tuh)^{\frac{n-4}{4}} (\tuh)_{x} = \frac{4}{n}\left((\tuh)^{\frac{n}{4}}\right)_{x}$. Then, we may use Fatou's lemma and lower-semi-continuity of the $L^{4q}(\Omega; L^{4}(\domain\times [0,T_{max}]))$ norm with respect to weak convergence to conclude.
As the other terms are treated similarly, we have finished the proof. 
\end{proof}

\section{Results of calculus related to the discrete Stratonovich correction term}\label{sec:CalcPorMed}

In the derivation of the energy-entropy estimate Proposition~\ref{prop:integral1}, we have used several identities and estimates regarding the operators $\porad$ and $\porb$, see \eqref{eq:definePorad} and \eqref{eq:definePorb}, i.e. the discrete version of the Stratonovich correction term.
In this section, we provide the corresponding arguments.
Moreover, we give a useful formula for the mass deviation due to the discretization of the Stratonovich correction term.
In order to simplify computations when testing \eqref{eq:definePorad} with powers of $u^{h}$, we introduce
\begin{align}\notag	\label{discgrad2}
	&A_{\nabla}^{h}(u^{h},v^{h}) 
	:= -\tfrac{(n-2)}{12} h \sum_{i=1}^{L_{h}}
	\left\{ \left((\ui)^{n-3} + (\uip)^{n-3}\right) \left|\tfrac{\uip - \ui}{h}\right|^{2} 
	\right. 
	\\ \notag
	& \left. +  \left((\ui)^{n-3} + (\uim)^{n-3}\right) \left|\tfrac{\ui - \uim}{h}\right|^{2}\right\} v_{i}^{h}
	\\ \notag
	&-\tfrac{(n-2)}{24}h \sum_{i=1}^{L_{h}} \left((\uip)^{n-3} + 2 (\ui)^{n-3} + (\uim)^{n-3}\right) \left|\tfrac{\uip-\uim}{h}\right|^{2} v_{i}^{h}
	\\ 
	&-\tfrac{(n-2)}{24}h \sum_{i=1}^{L_{h}} \left(2(\ui)^{n-3} - (\uip)^{n-3} - (\uim)^{n-3}\right) \left(\left|\tfrac{\uip-\ui}{h}\right|^{2}  + \left|\tfrac{\ui-\uim}{h}\right|^{2}\right)v_{i}^{h}  \,.
\end{align}
The following identity holds.
\begin{lemma}\label{lem:bilineq}
	For all $u^{h}, v^{h} \in X_{h}$, we have
	\begin{align}
		A_{\Delta}^{h}(u^{h}, v^{h}) = A_{\nabla}^{h}(u^{h}, v^{h}) \, . 
	\end{align}
\end{lemma}
\begin{proof}Starting with the second term of $A_{\Delta}^{h} (u^{h}, v^{h})$ and using the identities 
	\begin{align}\notag
		\left(\uip- \uim \right)=\left(\uip - \ui +\ui - \uim \right)
	\end{align}
	and
	\begin{align}\notag
		&(\uip - \ui)(\ui - \uim) \left(2(\ui)^{n-3} + (\uip)^{n-3} + (\uim)^{n-3}\right)  
		\\ \notag
		&
		=\tfrac{1}{2}\left(2(\ui)^{n-3} + (\uip)^{n-3} + (\uim)^{n-3}\right)  \left(\uip - \uim\right)^{2}
		\\ \notag
		&\quad
		- \tfrac{1}{2} \left(2(\ui)^{n-3} + (\uip)^{n-3} + (\uim)^{n-3}\right) \left(\left|\uip - \ui\right|^{2} + \left|\ui - \uim\right|^{2}\right)\, ,
	\end{align}
	a lengthy but straightforward computation gives the result.  
\end{proof}
\begin{lemma}\label{lem:discgradlapl}
	It holds
	\begin{align}\notag \label{lem:discgradlapl1}
		&A_{\Delta}^{h}(u^{h}, -\Delta_{h}u^{h}) = A_{\nabla}^{h}(u^{h}, -\Delta_{h}u^{h})
		\\ \notag
		&=
		\tfrac{1}{4}\tfrac{|(n-2)(n-3)|}{3} h \sum_{i=1}^{L_{h}} \theta_{i}(\ui,\uip)^{n-4} \left|\tfrac{\uip -\ui}{h} \right|^{2} \left[ \left|\tfrac{\ui -\uim}{h} \right|^{2} +
		2\left|\tfrac{\uip -\ui}{h} \right|^{2} +\left|\tfrac{u^{h}_{i+2} -\uip}{h} \right|^{2} \right]
		\\ 
		&\ge 	
		\tfrac{|(n-2)(n-3)|}{3} \tfrac{(1+C_{osc})^{n-4}}{2} h \sum_{i=1}^{L_{h}}  (\ui)^{n-4} \left( \left|\tfrac{\uip -\ui}{h} \right|^{2} \left|\tfrac{\ui -\uim}{h} \right|^{2} + \left|\tfrac{\uip -\ui}{h} \right|^{4} \right) 
	\end{align}
	with $\theta_{i}^{h}(\ui,\uip) = \sigma_{i}^{h} \ui + (1-\sigma_{i}^{h})\uip$ for $\sigma_{i}^{h} \in [0,1]$.
\end{lemma}
\begin{proof}
	Using periodicity as well as \eqref{eq:productrule}, we compute for the first term in \eqref{eq:definePorad}
	\begin{align}\notag \label{discgradlapl1}
		&\tfrac{n-2}{6} h \sum_{i=1}^{L_{h}}
		(\ui)^{n-3} \bigg(\left|\tfrac{\ui - \uim}{h}\right|^{2}+\left|\tfrac{\uip - \ui}{h}\right|^{2}\bigg) \Delta_{h}u^{h}
		\\ \notag
		&=
		\tfrac{n-2}{6} h \sum_{i=1}^{L_{h}} \Bigg\{ \tfrac{\uip - \ui}{h} \tfrac{1}{h} 
		\Bigg[(\ui)^{n-3} \bigg(\left|\tfrac{\ui - \uim}{h}\right|^{2}+\left|\tfrac{\uip - \ui}{h}\right|^{2}\bigg) 
		\\ \notag
		&\qquad \qquad	 - 
		(\uip)^{n-3} \bigg(\left|\tfrac{\uip - \ui}{h}\right|^{2}+\left|\tfrac{u_{i+2}^{h} - \uip}{h}\right|^{2}\bigg)
		\Bigg]\Bigg\}
		\\ \notag
		&=
		\tfrac{n-2}{6} h \sum_{i=1}^{L_{h}} \Bigg\{ \tfrac{\uip - \ui}{h}
		\Bigg[
		\tfrac{(\ui)^{n-3}- (\uip)^{n-3}}{h} 
		\tfrac{1}{2}
		\bigg(
		\left|\tfrac{
			\ui - \uim}{h}\right|^{2}
		+2\left|\tfrac{\uip - \ui}{h}\right|^{2}
		+\left|\tfrac{u_{i+2}^{h} - \uip}{h}\right|^{2}
		\bigg)
		\\ \notag
		&\qquad\qquad +\tfrac{1}{h} 
		\bigg(\tfrac{(\ui)^{n-3}+(\uip)^{n-3}}{2}\bigg)
		\bigg(\left|\tfrac{\ui - \uim}{h}\right|^{2}-\left|\tfrac{u_{i+2}^{h} - \uip}{h}\right|^{2}\bigg)  
		\Bigg] \Bigg\}
		\\ \notag
		&=
		\tfrac{1}{4}\tfrac{|(n-2)(n-3)|}{3} h \sum_{i=1}^{L_{h}} \theta_{i}^{h}(\ui,\uip)^{n-4} \left|\tfrac{\uip - \ui}{h}\right|^{2}
		\Bigg[
		\left|\tfrac{\ui - \uim}{h}\right|^{2}
		+2\left|\tfrac{\uip - \ui}{h}\right|^{2}
		+\left|\tfrac{u_{i+2}^{h} - \uip}{h}\right|^{2}
		\Bigg]
		\\ \notag
		&\quad+ \tfrac{n-2}{6}h \sum_{i=1}^{L_{h}} \tfrac{1}{h}
		\tfrac{\uip - \ui}{h} \cdot
		\tfrac{(\ui)^{n-3}+(\uip)^{n-3}}{2}
		\Bigg[
		\left|\tfrac{\ui - \uim}{h}\right|^{2} -\left|\tfrac{u_{i+2}^{h} - \uip}{h}\right|^{2}
		\Bigg] 
		\\ 
		&=: R_{1}+R_{2} \, .
	\end{align}
	Using  
	\begin{align}\notag
		&\left(u_{i+2}^{h}-\uip\right)^{2} - \left(\ui - \uim\right)^{2} 
		\\ \notag
		&=
		\left(u_{i+2}^{h} - 2 \uip +\ui\right)\left(u_{i+2}^{h}-\ui\right)
		+
		\left(\uip - 2 \ui +\uim\right)\left(\uip-\uim\right)\, 
	\end{align}
	as well as periodicity, we find for $R_{2}$
	\begin{align}\notag
		R_{2} = -\tfrac{n-2}{12} h \sum_{i=1}^{L_{h}}
		&\left\{\left((\ui)^{n-3} + (\uip)^{n-3}\right) \tfrac{\uip -\ui}{h} 
		+
		\left((\ui)^{n-3} + (\uim)^{n-3}\right) \tfrac{\ui -\uim}{h} 
		\right\}
		\\ \notag
		&\quad \cdot \bigg(\tfrac{\uip - \uim}{h}\bigg)\bigg( \tfrac{\uip - 2\ui +\uim}{h^{2}}\bigg)\, ,
	\end{align}
	which corresponds up to the sign to the second term on the right-hand side of \eqref{eq:definePorad}.
	Thus, the equality in \eqref{lem:discgradlapl1} is proven. 
	
	To show the estimate, we note that $n-4 < 0$ and thus
	\begin{align}\notag
		(\sigma_{i}^{h} \ui + (1-\sigma_{i}^{h})\uip)^{n-4} &\ge (\sigma_{i}^{h} \ui + (1-\sigma_{i}^{h}) C_{osc}\ui)^{n-4} 
		\\ \notag
		&\ge (\ui + C_{osc} \ui )^{n-4} 
		\\
		&= (1+C_{osc})^{n-4} (\ui)^{n-4}
	\end{align} 
	and similarly
	\begin{align}
		(\sigma_{i}^{h} \ui + (1-\sigma_{i}^{h})\uip)^{n-4} \ge (1 + C_{osc})^{n-4} (\uip)^{n-4} \, .
	\end{align}
	Then we have 
	\begin{align} \notag
		&\tfrac{1}{4}\tfrac{|(n-2)(n-3)|}{3} h \sum_{i=1}^{L_{h}} \theta_{i}(\ui,\uip)^{n-4} \left|\tfrac{\uip -\ui}{h} \right|^{2} \left[ \left|\tfrac{\ui -\uim}{h} \right|^{2} +
		2\left|\tfrac{\uip -\ui}{h} \right|^{2} +\left|\tfrac{u^{h}_{i+2} -\uip}{h} \right|^{2} \right]
		\\ \notag
		&\ge
		\tfrac{1}{4}\tfrac{|(n-2)(n-3)|}{3} \bigg\{
		h \sum_{i=1}^{L_{h}}  (1+C_{osc})^{n-4} (\uip)^{n-4} \left|\tfrac{\uip -\ui}{h} \right|^{2} \left|\tfrac{u^{h}_{i+2} -\uip}{h} \right|^{2}
		\\ \notag &\qquad \qquad \quad+ 2 h \sum_{i=1}^{L_{h}}  (1+C_{osc})^{n-4} (\ui)^{n-4} \left|\tfrac{\uip -\ui}{h} \right|^{4}
		\\ \notag 
		&\qquad \qquad \quad+ h \sum_{i=1}^{L_{h}}  (1+C_{osc})^{n-4} (\ui)^{n-4} \left|\tfrac{\uip -\ui}{h} \right|^{2} \left|\tfrac{\ui -\uim}{h} \right|^{2} \bigg\}
		\\ \notag
		&= 
		\tfrac{(1+C_{osc})^{n-4}}{2}\tfrac{|(n-2)(n-3)|}{3}
		\\ \notag  
		&\quad \cdot\bigg\{
		h \sum_{i=1}^{L_{h}}  (\ui)^{n-4} \left|\tfrac{\uip -\ui}{h} \right|^{4}	
		+ h \sum_{i=1}^{L_{h}}  (\ui)^{n-4} \left|\tfrac{\uip -\ui}{h} \right|^{2} \left|\tfrac{\ui -\uim}{h} \right|^{2} \bigg\} \, .
	\end{align}
\end{proof}

\begin{lemma}\label{lem:discsingular}
	For $n \in (2,3)$ there is a positive constant $c_{p}$ independent of $h>0$, such that
	\begin{align}\notag \label{eq:discsingular}
		A_{\Delta}^{h}(u^{h}, -\mathcal{I}_{h}[(u^{h})^{-p-1}]) 
		\ge
		c_{p} (n-2) h \sum_{i=1}^{L_{h}} (\ui)^{n-p-4} \left|\tfrac{\uip-\ui}{h}\right|^{2} \, .
	\end{align}
\end{lemma}
\begin{proof}
	Using $\mathcal{I}_{h}[(u^{h})^{-p-1}]$ as a testfunction in \eqref{discgrad2}, we obtain
	\begin{align}\notag
		&\tfrac{1}{n-2} A_{\Delta}^{h}(u^{h}, -\mathcal{I}_{h}[(u^{h})^{-p-1}]) = \tfrac{1}{n-2} A_{\nabla}^{h}(u^{h},-\mathcal{I}_{h}[(u^{h})^{-p-1}])
		\\ \notag	
		&=
		\tfrac{h}{12} \sum_{i=1}^{L_{h}}
		\left[ 
		\left((\ui)^{n-3}+(\uip)^{n-3}\right) \left|\tfrac{\uip -\ui}{h}\right|^{2}
		+
		\left((\ui)^{n-3}+(\uim)^{n-3}\right) \left|\tfrac{\ui -\uim}{h}\right|^{2} 
		\right] (\ui)^{-p-1}
		\\ \notag	
		&\quad+
		\tfrac{h}{24} \sum_{i=1}^{L_{h}}
		\left[ 
		\left((\uip)^{n-3}+2(\ui)^{n-3} + (\uim)^{n-3}\right) \left|\tfrac{\uip -\uim}{h}\right|^{2} 
		\right] (\ui)^{-p-1}
		\\ \notag	
		&\quad+
		\tfrac{h}{24} \sum_{i=1}^{L_{h}}
		\left[ 
		\left(2(\ui)^{n-3} -(\uip)^{n-3} - (\uim)^{n-3}\right) \left( \left|\tfrac{\uip -\ui}{h}\right|^{2} 
		+
		\left|\tfrac{\ui -\uim}{h}\right|^{2}\right) 
		\right] (\ui)^{-p-1}
		\\ \notag
		&=:
		\operatorname{I}+\operatorname{II}+\operatorname{III}
		\, .
	\end{align}
	While $\operatorname{I}$ and $\operatorname{II}$ have a good sign, the sign of $\operatorname{III}$ depends on the ordering of $\ui,\uip$, and $\uim$ according to size. We will discuss the six cases that may occur separately. In the following, the $i$-th summand of $\operatorname{I}$, $\operatorname{II}$, $\operatorname{III}$ is denoted by $\operatorname{I}_{i}$, $\operatorname{II}_{i}$, $\operatorname{III}_{i}$, respectively.
	
	\noindent
	\textbf{Case 1:} $\uim \le \ui \le \uip$
	
	Noting that 
	\begin{align}\notag
		\left(\uip - \uim \right)^{2} 
		=
		\left(\uip - \ui \right)^{2} + \left(\ui - \uim \right)^{2}
		+
		2\left(\uip - \ui \right)\left(\ui - \uim\right)\, , 
	\end{align} 
	we may rewrite $\operatorname{II}_{i}$ and $\operatorname{III}_{i}$
	\begin{align}\notag
		\operatorname{II}_{i}+\operatorname{III}_{i}
		&=
		\tfrac{h}{24} \left((\uip)^{n-3} +2(\ui)^{n-3} + (\uim)^{n-3}\right) 
		\\ \notag
		&\quad \cdot \left(\left|\tfrac{\uip -\ui}{h}\right|^{2} +\left|\tfrac{\ui -\uim}{h}\right|^{2} 
		+2\left(\tfrac{\uip -\ui}{h}\right)\left(\tfrac{\ui -\uim}{h}\right)  
		\right) (\ui)^{-p-1}
		\\ \notag
		&\quad 
		+ \tfrac{h}{24} \left(2(\ui)^{n-3} -(\uip)^{n-3} - (\uim)^{n-3}\right)
		\left(\left|\tfrac{\uip -\ui}{h}\right|^{2} +\left|\tfrac{\ui -\uim}{h}\right|^{2} \right) (\ui)^{-p-1}
		\\ \notag
		&=
		\tfrac{h}{6} (\ui)^{n-p-4}\left(\left|\tfrac{\uip -\ui}{h}\right|^{2} +\left|\tfrac{\ui -\uim}{h}\right|^{2} \right)
		\\ \notag 
		&\quad+  \tfrac{h}{12} \left((\uip)^{n-3} +2(\ui)^{n-3} + (\uim)^{n-3}\right)  
		\left(\tfrac{\uip -\ui}{h}\right)\left(\tfrac{\ui -\uim}{h}\right) (\ui)^{-p-1} \, ,
	\end{align}
	where the last line is positive by assumption. The term $\operatorname{I}$ remains unaffected.
	
	\noindent
	\textbf{Case 2:} $\uip \le \ui \le \uim$
	
	This case is symmetric to the first case and we obtain a similar positivity result.
	
	\noindent
	\textbf{Case 3:} $\ui \le \uim \le \uip$
	
	For the first factor in $\operatorname{III_{i}}$, we have
	\begin{align}\notag
		\left(2(\ui)^{n-3} -(\uip)^{n-3} - (\uim)^{n-3}\right)
		=
		(\ui)^{n-3} - (\uip)^{n-3} + (\ui)^{n-3} - (\uim)^{n-3} \ge 0 \, ,
	\end{align}
	due to $n\in(2,3)$. Hence, no absorption is necessary.
	
	\noindent
	\textbf{Case 4:} $\ui \le \uip \le \uim$
	
	The situation is analogous to Case 3.
	
	\noindent
	\textbf{Case 5:} $\ui \ge \uim \ge \uip$
	
	We have 
	$
	\ui - \uip \ge \ui - \uim \ge 0
	$
	and 
	$
	- (\uim)^{n-3} \ge -(\uip)^{n-3} \, ,
	$
	since $n-3 <0$.
	Therefore, we may estimate
	\begin{align}\notag
		\operatorname{III_{i}} 
		&\ge \tfrac{h}{12} \left((\ui)^{n-3} - (\uip)^{n-3}\right)
		\left(\left|\tfrac{\uip -\ui}{h}\right|^{2} 
		+
		\left|\tfrac{\ui -\uim}{h}\right|^{2}\right) 
		(\ui)^{-p-1}
		\\ 
		&\ge 
		\tfrac{h}{6} \left((\ui)^{n-3} - (\uip)^{n-3}\right)\left|\tfrac{\uip -\ui}{h}\right|^{2}(\ui)^{-p-1} 
	\end{align}
	and thus
	\begin{align}\notag \label{discsingular1}
		\operatorname{III_{i}} + \operatorname{I_{i}}  
		&\ge
		h \left(
		\left(\tfrac{3}{12}(\ui)^{n-3} - \tfrac{1}{12}(\uip)^{n-3}\right)
		\left|\tfrac{\uip -\ui}{h}\right|^{2} 
		\right. 
		\\
		&\left. \qquad+
		\tfrac{1}{12} \left((\ui)^{n-3} + (\uim)^{n-3}\right) \left|\tfrac{\ui -\uim}{h}\right|^{2}\right) 
		(\ui)^{-p-1} \, .
	\end{align}
	To absorb the negative term in \eqref{discsingular1}, we consider the index level $j=i+1$ and take the term $z_{j}=\tfrac{h}{12} (u_{j}^{h})^{n-3} \left|\tfrac{u_{j}^{h} - u_{j-1}^{h}}{h}\right|^{2}$ into account which is part of $\operatorname{I_{j}}$. Due to $\uip \le \ui$, we have $(\uip)^{-p-1} \ge (\ui)^{-p-1}$ and therefore 
	\begin{align}\notag
		\tfrac{h}{12} (u_{j}^{h})^{n-3} \left|\tfrac{u_{j}^{h} - u_{j-1}^{h}}{h}\right|^{2} (u_{j}^{h})^{-p-1} 
		\ge
		\tfrac{h}{12} (\uip)^{n-3} \left|\tfrac{\uip - \ui}{h}\right|^{2} (\ui)^{-p-1}\, .
	\end{align}
	Hence,
	\begin{align}\notag
		\operatorname{III_{i}} + \operatorname{I_{i}} +z_{j} >0 \, .
	\end{align}
	We note, that $z_{j}$ is not needed for absorption on level $j = i+1$. For $\ui \ge \uip$, we have for $j=i+1$
	one of the three cases
	$u_{j-1}^{h} \ge u_{j}^{h} \ge u_{j+1}^{h}$, $u_{j-1}^{h} \ge u_{j+1}^{h} \ge u_{j}^{h}$ or $u_{j+1}^{h} \ge u_{j-1}^{h} \ge u_{j}^{h}$, which correspond to the cases 2, 3 or 4 already discussed above. There no absorption was needed. 
	
	\noindent
	\textbf{Case 6:} $\ui \ge \uip \ge \uim$
	
	Taking into account the term $z_{j}$ at index level $j = i-1$, this case is treated very similarly to Case 5. We omit the details. 
	
	To conclude, we make use of periodicity and the Oscillation Lemma~\ref{lem:lowerbound} to bound the terms obtained above from below by the right-hand side of \eqref{eq:discsingular}.
\end{proof}

\begin{lemma}\label{lem:lemma68}
	There is a positive constant $c_{ent}$ independent of $h>0$ such that for any $\varepsilon >0$ 
	\begin{align}\notag \label{lem:discentr}
		&A_{\Delta}^{h}(u^{h}, \mathcal{I}_{h}[g_{h}(u^{h})]) = A_{\nabla}^{h}(u^{h}, \mathcal{I}_{h}[g_{h}(u^{h})])
		\\ \notag
		&\ge
		c_{ent}\Bigg(
		h \sum_{i=1}^{L_{h}} (\ui)^{-2} \left|\tfrac{\uip-\ui}{h}\right|^{2} 
		-\varepsilon
		h \sum_{i=1}^{L_{h}} (\ui)^{n-4} \left|\tfrac{\uip-\ui}{h}\right|^{4} 
		\\ 
		&
		\quad -h^{2} \sum_{i=1}^{L_{h}} (\ui)^{n-4} \left|\tfrac{\uip-\ui}{h}\right|^{4}
		-h^{2} \sum_{i=1}^{L_{h}} (\ui)^{n-4} \left|\tfrac{\uip-\ui}{h}\right|^{2}
		-C_{\varepsilon}\bigg(\discint \ui +1\bigg)
		\Bigg) \, .
	\end{align}
	
\end{lemma}
\begin{proof}
	Assuming an appropriate cut-off parameter in the discrete mobility combined with the criterion for stopping times, we may assume $g_{n}(u^{h}) = \tfrac{1}{1-n}((u^{h})^{1-n}-1)$.
	Writing this as
	\begin{align}\label{discentr1}
		g_{n}= \tfrac{1}{1-n} (u^{h})^{1-n} + \tfrac{1}{n-1} =: \hat{g}_{n}(u^{h}) + \tfrac{1}{n-1} \, ,
	\end{align}
	we find, using similar reasoning as in Lemma~\ref{lem:discsingular}, the estimate
	\begin{align}\notag
		& A^{h}_{\Delta}\left(u^{h},\mathcal{I}_{h}[\hat{g}_{n}(u^{h})]\right) 
		=
		A^{h}_{\nabla}\left(u^{h},\mathcal{I}_{h}[\hat{g}_{n}(u^{h})]\right)
		\\ \notag
		&\ge
		\tilde{c} \bigg(h \sum_{i=1}^{L_{h}} (\ui)^{-2} \left|\tfrac{\uip - \ui }{h}\right|^{2}\bigg)\, ,
	\end{align}
	$ \tilde{c} >0 $.
	For the second term in \eqref{discentr1}, we compute
	\begin{align}\notag
		&\tfrac{1}{n-2} A^{h}_{\nabla}\left(u^{h},\mathcal{I}_{h}[\tfrac{1}{n-1}]\right)
		\\ \notag
		&=
		- \tfrac{h}{12(n-1)} \sum_{i=1}^{L_{h}}
		\left\{ \left((\ui)^{n-3} + (\uip)^{n-3}\right) \left|\tfrac{\uip - \ui}{h}\right|^{2} 
		+
		\left((\ui)^{n-3} + (\uim)^{n-3}\right) \left|\tfrac{\ui - \uim}{h}\right|^{2} 
		\right\}
		\\ \notag
		&\quad-
		\tfrac{h}{24(n-1)} \sum_{i=1}^{L_{h}}
		\left((\uip)^{n-3} + 2(\ui)^{n-3} + (\uim)^{n-3}\right) \left|\tfrac{\uip - \uim}{h}\right|^{2} 
		\\ \notag
		&\quad-
		\tfrac{h}{24(n-1)} \sum_{i=1}^{L_{h}}
		\left(2(\ui)^{n-3} - (\uip)^{n-3} - (\uim)^{n-3}\right) \bigg(\left|\tfrac{\uip - \ui}{h}\right|^{2} +\left|\tfrac{\ui - \uim}{h}\right|^{2} \bigg)
		\\ 
		& =: \operatorname{I} + \operatorname{II} +\operatorname{III} \, .
	\end{align}
	
	Ad I: Using Lemma~\ref{lem:lowerbound}, periodicity, $n\in(2,3)$, and Young's inequality, we find for positive constants $\varepsilon, C_{\varepsilon}$ and $C$
	\begin{align}\notag
		|\operatorname{I}| 
		&\le C h \sum_{i=1}^{L_{h}} (\ui)^{n-3} \left|\tfrac{\uip -\ui}{h}\right|^{2}
		\\ \notag
		&\le
		\varepsilon h \sum_{i=1}^{L_{h}} (\ui)^{n-4} \left|\tfrac{\uip -\ui}{h}\right|^{4}
		+ C_{\varepsilon} h \sum_{i=1}^{L_{h}} (\ui)^{n-2}
		\\ \notag
		&\le
		\varepsilon h \sum_{i=1}^{L_{h}} (\ui)^{n-4} \left|\tfrac{\uip -\ui}{h}\right|^{4}
		+ C_{\varepsilon} \bigg(h \sum_{i=1}^{L_{h}} (\ui) + 1 \bigg) \, .
	\end{align} 
	Ad $\operatorname{II}$:
	We may estimate $\left|\uip - \uim\right|^{2} \le 2\left(\left|\uip - \ui\right|^{2}+ \left|\ui - \uim\right|^{2}\right)$. Then, with the help of Lemma~\ref{lem:lowerbound}, we may estimate $|\operatorname{II}|$ in a similar way as $|\operatorname{I}|$.
	
	Ad $\operatorname{III}$:
	The mean-value theorem combined with Lemma~\ref{lem:lowerbound} yields 
	\begin{align}
		\left|(\ui)^{n-3} - (\uip)^{n-3}\right| \le h C (\ui)^{n-4} \left|\tfrac{\ui - \uip}{h}\right|
	\end{align}
	and 
	\begin{align}
		\left|(\ui)^{n-3} - (\uim)^{n-3}\right| \le h C (\uim)^{n-4} \left|\tfrac{\ui - \uim}{h}\right| 
	\end{align}
	with a positive constant $C$.
	Then, using periodicity and $2(\ui)^{n-3} - (\uip)^{n-3} - (\uim)^{n-3} = (\ui)^{n-3} - (\uip)^{n-3} + (\ui)^{n-3} - (\uim)^{n-3}$, we have 
	\begin{align}\notag
		|\operatorname{III}| 
		&\le
		C  h^{2} \sum_{i=1}^{L_{h}} (\ui)^{n-4} \left|\tfrac{\uip - \ui}{h}\right| \left|\tfrac{\ui - \uim}{h}\right|^{2}
		\\ \notag
		&\quad+
		Ch^{2} \sum_{i=1}^{L_{h}} (\uim)^{n-4} \left|\tfrac{\uip - \ui}{h}\right|^{2} \left|\tfrac{\ui - \uim}{h}\right|
		\\ 	
		&\quad+
		Ch^{2} \sum_{i=1}^{L_{h}} (\ui)^{n-4} \left|\tfrac{\uip -\ui}{h}\right|^{3} 
	\end{align}
	Now, by means of Young's inequality, periodicity and Lemma~\ref{lem:lowerbound}, we get the assertion of the lemma.
\end{proof}

\begin{lemma}\label{lem:disclaplap}
	We have 
	\begin{align}
		B_{\Delta}^{h}(u^{h}, -\Delta_{h}u^{h}) = h \sum_{i=1}^{L_{h}} (\ui)^{n-2} \left|\tfrac{\uip - 2 \ui + \uim}{h^{2}}\right|^{2} \, .
	\end{align}
\end{lemma}
\begin{proof}
	We may choose $v_{i}^{h} = - \tfrac{\uip -2\ui + \uim}{h^{2}}$ in \eqref{eq:definePorb} to obtain the result.
\end{proof}

\begin{lemma}\label{lem:discLapPow}
	Let $\sigma \neq 2-n$. Then for $\alpha \in \R$
	\begin{align}\notag
		B_{\Delta}^{h}\left(u^{h},\alpha \mathcal{I}_{h}[(u^{h})^{\sigma}]\right) 
		= - 	\alpha h \sum_{i=1}^{L_{h}} \tfrac{(\uip)^{\sigma+n-2} - (\ui)^{\sigma+n-2}}{\uip - \ui} \left|\tfrac{\uip - \ui}{h}\right|^{2} \, .
	\end{align}
	In particular, if $\alpha(\sigma+n-2) >0$, then 
	\begin{align}\notag \label{eq:discLapPow}
		B_{\Delta}^{h}\left(u^{h},\alpha \mathcal{I}_{h}[(u^{h})^{\sigma}]\right) 
		&=
		\alpha(\sigma+n-2) h \sum_{i=1}^{L_{h}} (\sigma_{i}^{h} \ui + (1-\sigma_{i}^{h})\uip)^{\sigma+n-3} \left|\tfrac{\uip - \ui}{h}\right|^{2}
		\\ 
		&\ge 
		Ch\sum_{i=1}^{L_{h}}(\ui)^{\sigma+n-3} \left|\tfrac{\uip - \ui}{h}\right|^{2} \, ,
	\end{align}
	where $C>0$ and $\sigma_{i}^{h}\in [0,1], \forall i = 1,\dots,L_{h}$, are appropriate parameters.
\end{lemma}
\begin{proof}
	Using periodicity and the mean-value theorem, we have
	\begin{align}\notag
		B_{\Delta}^{h}\left(u^{h},\alpha \mathcal{I}_{h}[(u^{h})^{\sigma}]\right)
		&=
		-\alpha h \sum_{i=1}^{L_{h}} (\ui)^{\sigma+n-2} \left(\tfrac{\uip -2\ui +\uim}{h^{2}}\right)
		\\ \notag
		&=
		-\alpha h \sum_{i=1}^{L_{h}}  (\ui)^{\sigma+n-2}\left(\tfrac{\uip -\ui}{h^{2}} -\tfrac{\ui - \uim}{h^{2}}\right)
		\\ \notag
		&=
		-\alpha h \sum_{i=1}^{L_{h}}\left((\ui)^{\sigma+n-2} -(\uip)^{\sigma+n-2}\right) \left(\tfrac{\uip - \ui}{h^{2}}\right)
		\\ \notag
		&=
		\alpha h \sum_{i=1}^{L_{h}}\left(\tfrac{(\uip)^{\sigma+n-2} -(\ui)^{\sigma+n-2}}{\uip - \ui}\right) \left|\tfrac{\uip - \ui}{h}\right|^{2}
		\\ \notag
		&=
		\alpha (\sigma+n-2) h \sum_{i=1}^{L_{h}}
		\left(\sigma_{i}^{h} \ui + (1-\sigma_{i}^{h})\uip \right)^{\sigma+ n-3} \left|\tfrac{\uip - \ui}{h}\right|^{2} \, .
	\end{align}
	Then, \eqref{eq:discLapPow} follows using the assumption $\alpha (\sigma+n-2)>0$ and the Oscillation Lemma~\ref{lem:lowerbound}.
\end{proof}
\begin{remark}\label{lem:remark}
	The factor $\alpha(\sigma+n-2)$ in Lemma~\ref{lem:discLapPow} is positive when testing with the first derivative of the effective interface potential, since $-p(n-2-p-1) >0$. All the other terms, as destabilizing terms in the interface potential or derivatives of the entropy, are of lower order in the power of $u^{h}$. Therefore, those terms can be estimated against the former one and a Gronwall term $h \sum_{i=1}^{L_{h}} \left|\tfrac{\uip - \ui}{h}\right|^{2}$.  
\end{remark}

\begin{lemma}\label{lem:MassOpAB}
	It holds
	\begin{align}\notag \label{lem:MAssOpAB0}
		\porad(\uh,1) + \porb(\uh,1) = 
		&-\tfrac{5(n-2)}{12} \discint ((\ui)^{n-3} + (\uim)^{n-3}) \left| \frac{\ui - \uim}{h}\right|^{2}
		\\ \notag
		&- \tfrac{(n-2)}{12} \discint ((\uip)^{n-3} + (\ui)^{n-3}) \left| \frac{\ui - \uim}{h}\right|^{2}
		\\ \notag
		&-\tfrac{(n-2)}{12} \discint \bigg\{ \left((\uip)^{n-3} + 2(\ui)^{n-3} + (\uim)^{n-3}\right)
		\\\notag 
		&\qquad \qquad \qquad \qquad \left(  \frac{\uip -2\ui + \uim}{h^{2}}\right) (\ui - \uim) \bigg\}
		\\ 
		&+ (n-2) \discint \uh(\theta(i-1,i))^{n-3} \left| \frac{\ui -\uim}{h} \right|^{2}\, ,
	\end{align}
	where $\uh(\theta(i-1,i)) \in (\uim, \ui)$.
\end{lemma}
\begin{proof}
	Periodicity and the mean-value theorem yield  
	\begin{align}\notag \label{lem:MAssOpAB1}
		\porb(\uh, 1) 
		&=  \discint   \frac{(\ui)^{n-2} - (\uim)^{n-2}}{\ui - \uim}  \left|\frac{\ui -\uim}{h} \right|^{2}
		\\ 
		&= (n-2) \discint   \uh(\theta(i-1,i))^{n-3}  \left|\frac{\ui -\uim}{h} \right|^{2} \, ,
	\end{align}
	where $\uh(\theta(i-1,i)) \in (\uim, \ui)$.
	Regarding $\porad$, we write 
	\begin{align}\notag \label{lem:MAssOpAB2}
		\porad(u^h, 1)= -\tfrac{n-2}{6} 
		\bigg\{ 
		&\discint (\ui)^{n-3}\bigg(\left|\tfrac{\ui-\uim}{h}\right|^2+\left|\tfrac{\uip-\ui}{h}\right|^2\bigg)
		\\ \notag
		+&\discint\ \left((\ui)^{n-3}+(\uip)^{n-3}\right)\tfrac{\uip-\ui}{h} \tfrac{\uip-\uim}{2h}
		\\ \notag 
		+&\discint
		\left((\uim)^{n-3}+(\ui)^{n-3}\right)\tfrac{\ui-\uim}{h}
		\tfrac{\uip-\uim}{2h} \bigg\} 
		\\
		=:  -\tfrac{n-2}{6} &\left\{ A_{1} + A_{2} + A_{3}\right\}\, .
	\end{align}
	We have 
	\begin{align}
		A_{1} = \discint ((\uip)^{n-3}+(\ui)^{n-3}) \left|\frac{\uip - \ui}{h}\right|^{2}
	\end{align}
	by periodicity.
	For $A_{2}$ and $A_{3}$ we will make use of the formula 
	\begin{align}
		(\uip -\ui)(\ui - \uim) = (\ui -\uim)^{2} + (\uip - 2 \ui + \uim)(\ui - \uim) \, .
	\end{align}
	Then 
	\begin{align}\notag
		A_{2} 
		&= 
		\tfrac{1}{2} \discint ((\uip)^{n-3}+(\ui)^{n-3}) 
		\\ 
		&\qquad \quad \cdot\Bigg( \left|\frac{\uip -\ui}{h}\right|^{2} +
		\left|\frac{\ui -\uim}{h}\right|^{2}
		+ \frac{\uip - 2\ui +\uim}{h^{2}} (\ui - \uim)\Bigg) \, .
	\end{align}
	Similarly, we find 
	\begin{align} \notag
		A_{3} &= 
		\discint ((\uip)^{n-3}+(\ui)^{n-3}) \left|\frac{\uip -\ui}{h}\right|^{2}
		\\ 
		&\quad +\tfrac{1}{2} \discint ((\ui)^{n-3}+(\uim)^{n-3})  
		\frac{\uip - 2\ui +\uim}{h^{2}} (\ui - \uim) \, .
	\end{align}
	Hence, 
	\begin{align}\notag \label{lem:MAssOpAB3}
		A_{1} + A_{2} + A_{3} 
		&= 
		\tfrac{5}{2} \discint ((\uip)^{n-3}+(\ui)^{n-3}) \left|\frac{\uip - \ui}{h}\right|^{2}
		\\ \notag
		&\quad+ \tfrac{1}{2} \discint ((\uip)^{n-3}+(\ui)^{n-3}) \left|\frac{\ui - \uim}{h}\right|^{2}
		\\ 
		&\quad+\tfrac{1}{2} \discint ((\uip)^{n-3}+ 2 (\ui)^{n-3} + (\uim)^{n-3})  
		\frac{\uip - 2\ui +\uim}{h^{2}} (\ui - \uim) \, .
	\end{align}
	Combining \eqref{lem:MAssOpAB2} with \eqref{lem:MAssOpAB3} and adding \eqref{lem:MAssOpAB1}, we get \eqref{lem:MAssOpAB0}.
\end{proof}

\appendix
\section{Auxiliary results}
\label{sec:appendix}
\begin{lemma}\label{lem:appenHilfeSiebenA}
  Let the assumptions of Proposition~\ref{prop:integral1} be satisfied and assume $\delta$ to be a given positive number. Then there exist  positive constants $C_1$ and $C=C(\delta)$ such that
	\begin{align}
		\label{eq:neuAppen1}
		\begin{split}
			& \int_0^{t\wedge T_h}h \sum_{i=1}^{L_h}\left(\tfrac{1}{m_\sigma(\ui)}+\tfrac{1}{m_\sigma(\uim)}\right)\left|\frac{(\ui)^{n/2}-(\uim)^{n/2}}{h}\right|^2ds \\
			& \quad \leq C_1 \int_0^{t\wedge T_h}h \sum_{i=1}^{\Lh}\left(\frac{(u_i^h)^{n-2}}{m_\sigma(\ui)}+\frac{(\uim)^{n-2}}{m_\sigma(\uim)}\right)\left|\frac{\ui-\uim}{h}\right|^2ds\\
			& \quad \leq \delta \int_0^{t\wedge T_h}\sum_{i=1}^\Lh\dashint_{\uim}^{\ui}|\tau|^{-p-2}d\tau\int_{(i-1)h}^{ih}\left|\frac{\ui-\uim}{h}\right|^2\,dx \,ds 
			+ C(\delta) 
			\int_0^{t\wedge T_h}R(s) \,ds \,.
		\end{split}
	\end{align}
 \end{lemma}
  \begin{proof}
    The proof is based on a straightforward application of the mean-value equality, the Oscillation Lemma~\ref{lem:lowerbound} and the  fact that for every $\delta>0$ there is a constant $C=C(\delta)>0$ such that
    $$ \frac{s^{n-2}}{m_\sigma(s)}\leq \delta s^{-p-2} + C(\delta) $$
    for all $s>0$ independently of $\sigma \in (0,1).$
  \end{proof}

\begin{lemma}
\label{lem:uHochN}
Let $n\in [2,4).$
For every $\delta>0$,  there is a positive constant $C_\delta$ such that
\begin{equation}
\label{eq:uHochN}
\discint (u^h_i)^n \leq \delta \int_\ort |u^h_x|^2dx + C_\delta \Bigg(\left( \int_\ort u^hdx\right)^{\tfrac{n+2}{4-n}}+C \left(\int_{\ort}\uh \, dx \right)^{n}\Bigg).
\end{equation}
\end{lemma}
\begin{proof}
We have for $C>0$
\begin{align}\notag
	\discint (u^h_i)^n \le C \int_\ort (u^{h})^{n} dx 
\end{align}
due to a standard homogeneity argument.
Gagliardo-Nirenberg inequality shows for positive constants $C_{1}$ and $C_{2}$
\begin{align}\label{uHochN1}
	\int_\ort (u^{h})^{n} dx 
	\le 
	C_{1}
	\left( 
	\int_\ort |u_{x}^{h}|^{2} dx 
	\right)^{\frac{\beta n}{2}}
	\left(
	\int_\ort u^{h} dx 
	\right)^{(1-\beta)n}
	+
	C_{2}
	\left(
	\int_\ort u^{h} dx
	\right)^{n} \, ,   
\end{align}
where $\beta =\frac{2(n-1)}{3n}$.
Since $n<4$, this allows for the application of Young's inequality in the first term on the right-hand side of \eqref{uHochN1}, which gives the result.
\end{proof}
A combination of Young's inequality with Lemma~\ref{lem:uHochN} gives
\begin{lemma}
\label{lem:poroussquare}
Let $n\in [2,4).$
For every $\delta>0$, there is a positive constant $C_\delta$ such that
\begin{align}
\label{eq:poroussquare}
\begin{split}
\discint (u^{h}_i)^{n-2}\left|\tfrac{u^{h}_{i+1}-u^{h}_i}{h}\right|^2 & \leq \delta \discint \left| u^{h}_i\right|^{n-4}\left|\tfrac{u^{h}_{i+1}-u^{h}_i}{h}\right|^4 + \delta\int_\ort |u^h_x|^2dx \\
& \qquad + C_\delta\left(\left(\int_\ort u^hdx\right)^{\frac{n+2}{4-n}}+\left(\int_{\ort}\uh \, dx \right)^{n}\right)
\end{split}
\end{align}
\end{lemma}

\begin{lemma}
	\label{lem:aux1}
	We have the estimate
	\begin{align*}
		&\Big(\tfrac{1}{m_\sigma(u^h)},\Big(\sum_{i=1}^\Lh Z_i(g_\ell)e_i\Big)^2\Big)_h
		\\&
		\leq
		2\sum_{i=1}^\Lh \frac{1}{m_\sigma(u^h_i)}
		\bigg(
		\int_{(i-1)h}^{ih} \Big|\frac{\sqmob(u^h)(x+h)-\sqmob(u^h)(x)}{h}\Big|^2 g_\ell(x)^2 \,dx
		\\&~~~~~~~~~~~~~~~~~~~~~~~~
		+\int_{(i-1)h}^{ih} (\sqmob(u^h)(x+h))^2\Big|\frac{g_\ell(x+h)-g_\ell(x)}{h}\Big|^2 \,dx
		\bigg)
	\end{align*}
	for arbitrary $\ell\in\N$ and positive $u^h\in X_h.$
\end{lemma}
The proof of Lemma~\ref{lem:aux1} is very similar to that one of Lemma~4.5 in \cite{FischerGruen2018} and is therefore omitted.

\begin{lemma}\label{lem:nineAnew} 
	Let $\varepsilon$ and $\eta$ be arbitrary positive numbers. Then, there exist positive constants $\tilde{C}_{\eta,1}$ and $\tilde{C}_{\eta,2}$ such that
	\begin{align}\notag
		\label{eq:nineAnew}
		&\frac{1}{2h} \sum_{\ell \in \mathbb{Z}} \lambda_{\ell}^{2} \int_{0}^{t\wedge T_{h}} \sum_{i=1}^{L_{h}} \left( \int_\ort \partial_h^-((\sqmob(u^h)g_\ell)_x) e_{i+1} \dx \right)^2 \ds
		\\ \notag
		&\le 
		(1+\eta)  C_{Strat} \Bigg\{ \left(1 + \varepsilon \tfrac{n-2}{2} \right)
		C_{osc}^{n-2} \int_{0}^{t\wedge T_{h}} h \sum_{i=1}^{L_{h}}  (\ui)^{n-2}  
		|(\Delta_{h} \uh )_{i}|^{2} \, ds
		\\ \notag
		&\quad+ \left(\tfrac{(n-2)^{2}}{4} + \tfrac{n-2}{2\varepsilon } \right)
		C_{osc}^{4-n} 
		\\ \notag
		&\quad \qquad \cdot \int_{0}^{t\wedge T_{h}} \discint (\ui)^{n-4} 
		\bigg\{ \left|\tfrac{u^{h}_{i+1}-u^{h}_{i}}{h}\right|^{4} +2 \left| \tfrac{u^{h}_{i+1}-u^{h}_{i}}{h} \right|^{2}  \left| \tfrac{u^{h}_{i}-u^{h}_{i-1}}{h} \right|^{2} + \left|\tfrac{u^{h}_{i}-u^{h}_{i-1}}{h}\right|^{4} \bigg\} \, ds \Bigg\} 
		\\  
		&\quad+
		\tilde{C}_{\eta,1}
		\int_{0}^{t\wedge T_{h}}  \discint (u^{h}_{i})^{n-2} \left| \tfrac{u^{h}_{i+1} - u^{h}_{i}}{h} \right|^{2}  \, ds
		+
		\tilde{C}_{\eta,2}
		\int_{0}^{t\wedge T_{h}} 	\int_\ort (\uh)^{n}\, dx \, ds  \, .
	\end{align}
\end{lemma}
\begin{remark}
	Although $\varepsilon$ and $\eta$ are related in the proof of Lemma~\ref{lem:nineAnew} to a Young argument, the roles of $\eta$ on the one hand and of $\varepsilon$ on the other hand are different. 
	The parameter $\eta$ is used for interpolation estimates involving terms, which can be directly absorbed, or involving the last two terms in \eqref{eq:nineAnew}, respectively, which are estimated in the proof of Proposition~\ref{prop:integral1} against Gronwall terms.
	In contrast, the parameter $\varepsilon$ is used to balance the discrete versions of $\intort u^{n-2} u_{xx}^{2} \dx$ and of $\intort u^{n-4} u_{x}^{4} \dx$ against each other to determine the smallness condition on $S$, see the arguments starting with \eqref{est:Absorber}. 
\end{remark}
\begin{proof}
	We use the notation $I_{i} = [x_{i}, x_{i+1}]$, where $x_{i} := ih$, $i = 1, \dots, L_{h}$.
	Using integration by parts, we find
	\begin{align} \notag \label{nineAnew1}
		&\sum_{i=1}^{L_{h}} \left( \int_\ort \partial_h^-((\sqmob(u^h)g_\ell)_x) e_{i+1} \dx \right)^2 
		\\ \notag 
		&=
		\sum_{i=1}^{L_{h}} \Bigg( \int_{I_{i}} \frac{(\sqmob(u^h)g_\ell)_x(x) -(\sqmob(u^h)g_\ell)_x(x-h)}{h} \cdot \frac{x-x_{i}}{h} \dx
		\\ \notag
		&\qquad\qquad+
		\int_{I_{i+1}} \frac{(\sqmob(u^h)g_\ell)_x(x) -(\sqmob(u^h)g_\ell)_x(x-h)}{h} \cdot \frac{x_{i+2}-x}{h} \dx	 \Bigg)^2 
		\\ \notag 
		&=
		\sum_{i=1}^{L_{h}} \Bigg(
		- \int_{I_{i}} \frac{(\sqmob(u^h)g_\ell)(x) -(\sqmob(u^h)g_\ell)(x-h)}{h^{2}} \dx
		\\ \notag
		&\qquad\qquad \quad
		+
		\frac{(\sqmob(u^h)g_\ell)(x_{i+1}) -(\sqmob(u^h)g_\ell)(x_{i})}{h} 
		\\ \notag
		&\qquad \qquad +\int_{I_{i}}   \frac{(\sqmob(u^h)g_\ell)(x+h) -(\sqmob(u^h)g_\ell)(x)}{h^{2}}  \dx
		\\ \notag
		&\qquad\qquad \quad -
		\frac{(\sqmob(u^h)g_\ell)(x_{i+1}) -(\sqmob(u^h)g_\ell)(x_{i})}{h} \Bigg)^2 
		\\ \notag 
		&=
		\sum_{i=1}^{L_{h}} \Bigg(
		\int_{I_{i}}  \frac{(\sqmob(u^h)g_\ell)(x+h) -2(\sqmob(u^h)g_\ell)(x) + (\sqmob(u^h)g_\ell)(x-h) }{h^{2}}  
		\dx \Bigg)^2 
		\\ 
		&=
		\sum_{i=1}^{L_{h}} \Bigg(
		\int_{I_{i}}  \partial_h^- \partial_h^+  (\sqmob(u^h)g_\ell)(x) 
		\dx \Bigg)^2 \, .
	\end{align}
	We have the following product rule
	\begin{align}
		 \partial_h^- \partial_h^+ (uv) =  (\partial_h^- \partial_h^+ u) v(\cdot +h) +  \partial_h^+ v \, \partial_h^- u 
		 +  (\partial_h^- \partial_h^+ v) u
		 +  \partial_h^- u \,\partial_h^- v \,.
	\end{align}
	Hence, 
	\begin{align}\notag
		\partial_h^- \partial_h^+ (\sqmob(u^h)g_\ell) &=  (\partial_h^- \partial_h^+ (\sqmob(u^h))) g_\ell(\cdot +h) 
		+  \partial_h^+ g_\ell \, \partial_h^- (\sqmob(u^h))
		\\ 
		&\quad+  (\partial_h^- \partial_h^+ g_\ell) \sqmob(u^h)
		+  \partial_h^- (\sqmob(u^h)) \,\partial_h^- g_\ell \, .
	\end{align}
	Inserting this in \eqref{nineAnew1}, gives
	\begin{align}\notag
		&\sum_{i=1}^{L_{h}} \left( \int_\ort \partial_h^-((\sqmob(u^h)g_\ell)_x) e_{i+1} \dx \right)^2
		= 
		\sum_{i=1}^{L_{h}} \Bigg(
		\int_{I_{i}}  \partial_h^- \partial_h^+  (\sqmob(u^h)g_\ell)(x) 
		\dx \Bigg)^2
		\\ \notag
		&\le
		h\sum_{i=1}^{L_{h}} 
		\int_{I_{i}}  \left(\partial_h^- \partial_h^+  (\sqmob(u^h)g_\ell)(x) \right)^{2}
		\dx 
		\\ \notag
		&=
		h\sum_{i=1}^{L_{h}} \Bigg\{
		\int_{I_{i}}  
		|\partial_h^- \partial_h^+ (\sqmob(u^h))|^{2} g_\ell^{2}(\cdot +h) \dx
		+  \int_{I_{i}} |\partial_h^+ g_\ell|^{2} \, |\partial_h^- (\sqmob(u^h))|^{2} \dx
		\\ \notag
		&\qquad \qquad
		+ \int_{I_{i}}  |\partial_h^- (\sqmob(u^h))|^{2} \, |\partial_h^- g_\ell|^{2} \dx
		+ \int_{I_{i}}  |\partial_h^- \partial_h^+ g_\ell|^{2} (\sqmob(u^h))^{2}
		\dx 
		\\ \notag
		&\qquad \qquad
		+ 2
		\int_{I_{i}}  (\partial_h^- \partial_h^+ (\sqmob(u^h))) g_\ell(\cdot +h) 
		(\partial_h^+ g_\ell) \, (\partial_h^- (\sqmob(u^h)))
		\dx 
		\\ \notag
		&\qquad \qquad
		+ 2 \int_{I_{i}}
		 (\partial_h^- \partial_h^+ (\sqmob(u^h))) g_\ell(\cdot +h)
		  (\partial_h^- (\sqmob(u^h))) \, (\partial_h^- g_\ell)
		\dx 
		\\ \notag
		&\qquad \qquad
		+ 2\int_{I_{i}}
		(\partial_h^- \partial_h^+ (\sqmob(u^h))) g_\ell(\cdot +h)
		(\partial_h^- \partial_h^+ g_\ell) (\sqmob(u^h))
		\dx 
		\\ \notag
		&\qquad \qquad
		+ 2\int_{I_{i}}
		(\partial_h^+ g_\ell) \, (\partial_h^- (\sqmob(u^h)))
		(\partial_h^- (\sqmob(u^h))) \, (\partial_h^- g_\ell)
		\dx 
		\\ \notag
		&\qquad \qquad
		+ 2\int_{I_{i}}
		(\partial_h^+ g_\ell) \, (\partial_h^- (\sqmob(u^h)))
		(\partial_h^- \partial_h^+ g_\ell) (\sqmob(u^h))
		\dx 
		\\ \notag
		&\qquad \qquad
		+ 2\int_{I_{i}}
		(\partial_h^- (\sqmob(u^h))) \, (\partial_h^- g_\ell)
		(\partial_h^- \partial_h^+ g_\ell) (\sqmob(u^h))
		\dx \Bigg\}
		\\ &=: R_{1} + \dots + R_{10}
	\end{align}
	To discuss $R_{1}$, we use $\sqmob(\uh)$ to be continuous and elementwise linear by definition. If $a$ is such a function, we have, writing $a_{i} := a(x_{i})$,
	\begin{align}\label{newAnine2}
		\partial_h^- \partial_h^+  a \Big|_{I_{i}}
		=
		\frac{a_{i+1} - 2a_{i} + a_{i-1}}{h^{2}} \cdot \frac{x_{i+1} -x}{h} 
		+
		\frac{a_{i+2} - 2a_{i+1} + a_{i}}{h^{2}} \cdot \frac{x-x_{i}}{h} \, .
	\end{align}
	Using \eqref{newAnine2}
	and the properties of the basis function $(g_{\ell})_{\ell \in \Z}$, cf. \cite{GruenKlein22} Appendix~A, we find for $R_{1}$
	\begin{align}\notag \label{newAnine3}
		&\frac{1}{2h} \sum_{\ell \in \mathbb{Z}} \lambda_{\ell}^{2} h\sum_{i=1}^{L_{h}} 
		\int_{I_{i}}  
		|\partial_h^- \partial_h^+ (\sqmob(u^h))|^{2} g_\ell^{2}(\cdot +h) \dx 
		\\
		& =
		\frac{1}{2} \left(\frac{\lambda_0^{2}}{2} + \sum_{\ell=1}^{\infty} \frac{2\lambda_{\ell}^{2}}{L}\right)
		\sum_{i=1}^{L_{h}} 
		\int_{I_{i}}  
		\left( \Delta_{h}^{L}\sqmob(\uh) \lambda(x) +  \Delta_{h}^{R}\sqmob(\uh) (1-\lambda(x)) \right)^{2} \dx
		\, ,
	\end{align}
	where
	$\Delta_{h}^{L}\sqmob(\uh)\big|_{I_{i}} = \frac{\sqmob(u_{i+1}^{h})-2 \sqmob(u_{i}^{h}) + \sqmob(u_{i-1}^{h})}{h^{2}}$, 
	$\Delta_{h}^{R}\sqmob(\uh)\big|_{I_{i}} = \frac{\sqmob(u_{i+2}^{h})-2 \sqmob(u_{i+1}^{h}) + \sqmob(u_{i}^{h})}{h^{2}}$, and 
	$\lambda(x)\big|_{I_{i}} = \frac{x_{i+1}-x}{h}$ (which also implies $(1-\lambda(x))\big|_{I_{i}} = \frac{x-x_{i}}{h}$).
	From \eqref{newAnine3}, by convexity we infer
	\begin{align}\notag
		&\frac{1}{2h} \sum_{\ell \in \mathbb{Z}} \lambda_{\ell}^{2} h\sum_{i=1}^{L_{h}} 
		\int_{I_{i}}  
		|\partial_h^- \partial_h^+ (\sqmob(u^h))|^{2} g_\ell^{2}(\cdot +h) \dx  
		\\ \notag
		&\le
		\frac{1}{2}\left(\frac{\lambda_0^{2}}{2} + \sum_{\ell=1}^{\infty} \frac{2\lambda_{\ell}^{2}}{L}\right)
		\sum_{i=1}^{L_{h}} \Big\{
		\int_{I_{i}}  
		\lambda(x) |\Delta_{h}^{L}\sqmob(\uh)|^{2} \dx  
		+
		\int_{I_{i}} (1-\lambda(x)) |\Delta_{h}^{R}\sqmob(\uh)|^{2}\dx \Big\}
		\\
		&=
		\frac{1}{2}\left(\frac{\lambda_0^{2}}{2} + \sum_{\ell=1}^{\infty} \frac{2\lambda_{\ell}^{2}}{L}\right)
		h\sum_{i=1}^{L_{h}} 
		|(\Delta_{h}\sqmob(\uh))_{i}|^{2} \dx \, ,  
	\end{align}
	where in the last step we used $\lambda(x)\big|_{I_{i}} = \lambda(x)\big|_{I_{i+1}}$ and $\Delta_{h}^{L}\sqmob(\uh)\big|_{I_{i+1}} = \Delta_{h}^{R}\sqmob(\uh)\big|_{I_{i}}$.
	We infer for $\sigma_{i}^{h}, \sigma_{i-1}^{h} \in [0,1]$
	\begin{align}\notag \label{eq:nine2new}
		&h\sum_{i=1}^{L_{h}} 
		|(\Delta_{h}\sqmob(\uh))_{i}|^{2} \dx
				\\ \notag
		&=
		h \sum_{i=1}^{L_{h}} \tfrac{1}{h^{2}} \left(\tfrac{(u^{h}_{i+1})^{n/2} - (u_{i}^{h})^{n/2}}{u^{h}_{i+1} - u^{h}_{i}} \cdot \tfrac{u^{h}_{i+1}-u^{h}_{i}}{h} - \tfrac{(u_{i}^{h})^{n/2} - (u_{i-1}^{h})^{n/2}}{u^{h}_{i} - u^{h}_{i-1}} \cdot \tfrac{u^{h}_{i}-u^{h}_{i-1}}{h}\right)^{2}
		\\ \notag
		&=
		h \sum_{i=1}^{L_{h}} \tfrac{1}{h^{2}} \left(\tfrac{n}{2} \left(\sigma_{i}^{h} u^{h}_{i} + (1-\sigma_{i}^{h})u^{h}_{i+1}\right)^{\frac{n-2}{2}} \tfrac{u^{h}_{i+1}-u^{h}_{i}}{h} 
		\right.
		\\ \notag
		&\hspace{2cm} \left.
		- \tfrac{n}{2} \left((1-\sigma^{h}_{i-1}) u^{h}_{i} + \sigma^{h}_{i-1}u^{h}_{i-1}\right)^{\frac{n-2}{2}} \tfrac{u^{h}_{i}-u^{h}_{i-1}}{h}
		\right)^{2}
		\\ \notag
		&= h \sum_{i=1}^{L_{h}} \Bigg\{ 
		\vphantom{\tfrac{\left(\sigma_{i}^{h} u^{h}_{i} + (1-\sigma_{i}^{h})u^{h}_{i+1}\right)^{\frac{n-2}{2}} - \left((1-\sigma^{h}_{i-1}) u^{h}_{i} + \sigma^{h}_{i-1}u^{h}_{i-1}\right)^{\frac{n-2}{2}} }{h}}
		\tfrac{n}{4} \left[ \left(\sigma_{i}^{h} u^{h}_{i} + (1-\sigma_{i}^{h})u^{h}_{i+1}\right)^{\frac{n-2}{2}} + \left((1-\sigma^{h}_{i-1}) u^{h}_{i} +\sigma^{h}_{i-1}u^{h}_{i-1}\right)^{\frac{n-2}{2}} \right]
		\tfrac{u^{h}_{i+1} -2u^{h}_{i} + u^{h}_{i-1}}{h^{2}}
		\\ 
		&
		\qquad \qquad+  \tfrac{n}{4} \left( \tfrac{u^{h}_{i+1}-u^{h}_{i-1}}{h}\right)\cdot  \tfrac{\left(\sigma_{i}^{h} u^{h}_{i} + (1-\sigma_{i}^{h})u^{h}_{i+1}\right)^{\frac{n-2}{2}} - \left((1-\sigma^{h}_{i-1}) u^{h}_{i} + \sigma^{h}_{i-1}u^{h}_{i-1}\right)^{\frac{n-2}{2}} }{h}
		\Bigg\}^{2} \, ,
	\end{align}
	where we used 
	the formula 
	\begin{align}\label{eq:productrule}
		A_{i}B_{i}-A_{i+1}B_{i+1} = \tfrac{A_{i}+A_{i+1}}{2} \left(B_{i}-B_{i+1}\right)
		+
		\tfrac{B_{i}+B_{i+1}}{2} \left(A_{i}-A_{i+1}\right) \, .
	\end{align}
	To proceed, we note that
	\begin{align}\label{eq:nine31new}
		\sigma^{h}_{i} \ui + (1-\sigma^{h}_{i}) \uip 
		\le C_{osc}\ui \,
	\end{align}
	and similarly
	\begin{align}\label{eq:nine32new}
		(1-\sigma^{h}_{i-1}) \ui + \sigma^{h}_{i-1} \uip \le C_{osc}\ui \,.
	\end{align}
	Moreover,
	\begin{align}\notag \label{eq:nine3new}
		&\tfrac{\left(\sigma_{i}^{h} u^{h}_{i} + (1-\sigma_{i}^{h})u^{h}_{i+1}\right)^{\frac{n-2}{2}} - \left((1-\sigma^{h}_{i-1}) u^{h}_{i} + \sigma^{h}_{i-1}u^{h}_{i-1}\right)^{\frac{n-2}{2}}}{h}
		\\ 
		&=
		\tfrac{\left(\sigma_{i}^{h} u^{h}_{i} + (1-\sigma_{i}^{h})u^{h}_{i+1}\right)^{\frac{n-2}{2}} - \left((1-\sigma^{h}_{i-1}) u^{h}_{i} + \sigma^{h}_{i-1}u^{h}_{i-1}\right)^{\frac{n-2}{2}}}{\sigma_{i}^{h} u^{h}_{i} + (1-\sigma_{i}^{h})u^{h}_{i+1}  - (1-\sigma^{h}_{i-1}) u^{h}_{i} - \sigma^{h}_{i-1}u^{h}_{i-1}}
		\cdot	
		\tfrac{(1-\sigma_{i}^{h})(u^{h}_{i+1}-u^{h}_{i}) + \sigma^{h}_{i-1}(u^{h}_{i}-u^{h}_{i-1})}{h} \, .	
	\end{align}
	Using the mean-value theorem and $C_{osc}>1$, the first term in \eqref{eq:nine3new} gives for $\theta_{i}^{h}\in [0,1]$
	\begin{align}\notag
		&\tfrac{\left(\sigma_{i}^{h} u^{h}_{i} + (1-\sigma_{i}^{h})u^{h}_{i+1}\right)^{\frac{n-2}{2}} - \left((1-\sigma^{h}_{i-1}) u^{h}_{i} + \sigma^{h}_{i-1}u^{h}_{i-1}\right)^{\frac{n-2}{2}}}{\sigma_{i}^{h} u^{h}_{i} + (1-\sigma_{i}^{h})u^{h}_{i+1}  - (1-\sigma^{h}_{i-1}) u^{h}_{i} - \sigma^{h}_{i-1}u^{h}_{i-1}}
		\\ \notag
		&=
		\tfrac{n-2}{2}\left[\theta_{i}^{h} \left(\sigma_{i}^{h}u^{h}_{i} + (1-\sigma_{i}^{h})\uip \right) + (1-\theta_{i}^{h}) \left((1-\sigma^{h}_{i-1})u^{h}_{i}+\sigma^{h}_{i-1}\uim\right)\right]^{\tfrac{n-4}{2}}
		\\ \notag 
		&\le 
		\tfrac{n-2}{2}\left[\theta_{i}^{h} \ui \left(\sigma_{i}^{h} + (1-\sigma_{i}^{h})C_{osc}^{-1} \right) + (1-\theta_{i}^{h})\ui \left((1-\sigma^{h}_{i-1})+\sigma^{h}_{i-1}C_{osc}^{-1}\right)\right]^{\tfrac{n-4}{2}}
		\\ \notag
		&\le 
		\tfrac{n-2}{2}\left[\theta_{i}^{h} \ui C_{osc}^{-1} \left(C_{osc}\sigma_{i}^{h} + (1-\sigma_{i}^{h}) \right) + (1-\theta_{i}^{h})\ui C_{osc}^{-1} \left(C_{osc}(1-\sigma^{h}_{i-1})+\sigma^{h}_{i-1}\right)\right]^{\tfrac{n-4}{2}}
		\\ \notag
		&\le 
		\tfrac{n-2}{2}\left[\theta_{i}^{h} \ui C_{osc}^{-1} + (1-\theta_{i}^{h})\ui C_{osc}^{-1} \right]^{\tfrac{n-4}{2}}
		\\ 
		&=
		\tfrac{n-2}{2}C_{osc}^{\tfrac{4-n}{2}}(\ui)^{\tfrac{n-4}{2}} \, .
	\end{align}
	Finally, the second factor in \eqref{eq:nine3new}, multiplied with $\tfrac{u^{h}_{i+1}-u^{h}_{i-1}}{h} = \tfrac{u^{h}_{i+1}-u^{h}_{i}}{h}+\tfrac{u^{h}_{i}-u^{h}_{i-1}}{h}$ according to \eqref{eq:nine2new}, satisfies
	\begin{align}\notag \label{eq:nine33new}
		&\left(\tfrac{u^{h}_{i+1}-u^{h}_{i}}{h}+\tfrac{u^{h}_{i}-u^{h}_{i-1}}{h}\right)
		\left(\tfrac{(1-\sigma_{i}^{h})(u^{h}_{i+1}-u^{h}_{i}) + \sigma^{h}_{i-1}(u^{h}_{i}-u^{h}_{i-1})}{h}\right)
		\\ \notag
		&= 
		(1-\sigma_{i}^{h})\left|\tfrac{\uip-\ui}{h}\right|^{2} + \sigma_{i-1}^{h} \left|\tfrac{\ui-\uim}{h}\right|^{2} + 2 (\sigma_{i-1}^{h} + (1- \sigma^{h}_{i})) \left(\tfrac{\uip-\ui}{h}\right)\left(\tfrac{\ui-\uim}{h}\right) 
		\\ \notag
		&\le 
		\left|\tfrac{\uip-\ui}{h}\right|^{2} +  \left|\tfrac{\ui-\uim}{h}\right|^{2} + 2\left( \tfrac{1}{2} \left|\tfrac{\uip-\ui}{h}\right|^{2} + \tfrac{1}{2} \left|\tfrac{\ui-\uim}{h}\right|^{2} \right)
		\\
		&=
		2 \Bigg( \left|\tfrac{\uip-\ui}{h}\right|^{2} +  \left|\tfrac{\ui-\uim}{h}\right|^{2} \Bigg) \, .
	\end{align}
	Inserting \eqref{eq:nine31new},\eqref{eq:nine32new}, \eqref{eq:nine33new}, and \eqref{eq:nine3new} in \eqref{eq:nine2new}, gives
	\begin{align}\notag
		&h\sum_{i=1}^{L_{h}} 
		|(\Delta_{h}\sqmob(\uh))_{i}|^{2} \dx
		\\ \notag
		&\le
		\tfrac{n^{2}}{4} \discint \left[
		C_{osc}^{\frac{n-2}{2}} (\ui)^{\frac{n-2}{2}}  
		(\Delta_{h}\uh)_{i} \right]^{2}
		\\ \notag
		&\quad+ \tfrac{n^{2}}{2} \discint 
		C_{osc}^{\frac{n-2}{2}} (\ui)^{\tfrac{n-2}{2}}
		(\Delta_{h}\uh)_{i} 
		\\ \notag 
		&\quad \qquad \qquad \cdot
		\tfrac{n-2}{2}
		C_{osc}^{\frac{4-n}{2}} (\ui)^{\frac{n-4}{2}} 
		\bigg(\left|\tfrac{\uip-\ui}{h}\right|^{2} +  \left|\tfrac{\ui-\uim}{h}\right|^{2} \bigg)
		\\ \notag
		&\quad + \tfrac{n^{2}}{4} \discint \Bigg[\tfrac{n-2}{2}
		C_{osc}^{\frac{4-n}{2}} (\ui)^{\frac{n-4}{2}} 
		\cdot \bigg(\left|\tfrac{\uip-\ui}{h}\right|^{2} +  \left|\tfrac{\ui-\uim}{h}\right|^{2} \bigg) \Bigg]^{2}
		\\ \notag	
		&\le
		\tfrac{n^{2}}{4} \discint \left[
		C_{osc}^{\frac{n-2}{2}} (\ui)^{\frac{n-2}{2}}  
		(\Delta_{h}\uh)_{i} \right]^{2}
		\\ \notag
		&\quad+ \tfrac{n^{2}}{4} \discint\tfrac{\varepsilon(n-2)}{2}
		\left( C_{osc}^{\frac{n-2}{2}} (\ui)^{\tfrac{n-2}{2}}
		(\Delta_{h}\uh)_{i} \right)^{2}
		\\ \notag
		&\quad+ \tfrac{n^{2}}{4} \discint
		\tfrac{n-2}{2\varepsilon}
		\bigg(C_{osc}^{\frac{4-n}{2}} (\ui)^{\frac{n-4}{2}} 
		\cdot \bigg(\left|\tfrac{\uip-\ui}{h}\right|^{2} +  \left|\tfrac{\ui-\uim}{h}\right|^{2} \bigg) \bigg)^{2}
		\\ \notag
		&\quad+ \tfrac{n^{2}}{4} \discint \Bigg[\tfrac{n-2}{2}
		C_{osc}^{\frac{4-n}{2}} (\ui)^{\frac{n-4}{2}} 
		\cdot \bigg(\left|\tfrac{\uip-\ui}{h}\right|^{2} +  \left|\tfrac{\ui-\uim}{h}\right|^{2} \bigg) \Bigg]^{2}
		\\ \notag
		&=
		\tfrac{n^{2}}{4} \left(1+\tfrac{\varepsilon(n-2)}{2}\right) \discint 
		C_{osc}^{n-2} (\ui)^{n-2}  
		|(\Delta_{h}\uh)_{i}|^{2}
		\\ 
		&\quad+ \tfrac{n^{2}}{4} \left(\tfrac{(n-2)^{2}}{4} + 		\tfrac{n-2}{2\varepsilon} \right)\discint
		C_{osc}^{4-n} (\ui)^{n-4} 
		\cdot \left(\left|\tfrac{\uip-\ui}{h}\right|^{2} +   \left|\tfrac{\ui-\uim}{h}\right|^{2} \right)^{2} \, .
	\end{align}
	Recalling $\tfrac{1}{2} \tfrac{n^{2}}{4}\left(\frac{\lambda_0^{2}}{2} + \sum_{\ell=1}^{\infty} \frac{2\lambda_{\ell}^{2}}{L}\right) = C_{Strat}$, we get for $R_{1}$
	\begin{align}\notag
		&\frac{1}{2h} \sum_{\ell \in \mathbb{Z}} \lambda_{\ell}^{2} h\sum_{i=1}^{L_{h}} 
		\int_{I_{i}}  
		|\partial_h^- \partial_h^+ (\sqmob(u^h))|^{2} g_\ell^{2}(\cdot +h) \dx  		
		\\ \notag
		 &\le 
		C_{Strat} \Bigg\{(1 + \varepsilon \tfrac{n-2}{2} ) h \sum_{i=1}^{L_{h}} 
		C_{osc}^{n-2} (\ui)^{n-2}  
		|(\Delta_{h} \uh )_{i}|^{2} 
		\\ \notag
		&\quad + 
		\left(\tfrac{(n-2)^{2}}{4} + \tfrac{n-2}{2\varepsilon } \right)
		C_{osc}^{4-n}
		\\ 
		& \qquad \cdot \discint (\ui)^{n-4} 
		\bigg( \left|\tfrac{u^{h}_{i+1}-u^{h}_{i}}{h}\right|^{4} +2 \left| \tfrac{u^{h}_{i+1}-u^{h}_{i}}{h} \right|^{2}  \left| \tfrac{u^{h}_{i}-u^{h}_{i-1}}{h} \right|^{2} + \left|\tfrac{u^{h}_{i}-u^{h}_{i-1}}{h}\right|^{4} \bigg)  \Bigg\} \, .
	\end{align}
	Regarding $R_{2}$, we have
	\begin{align}\label{nineAnew4}
		\sum_{\ell \in \mathbb{Z}} \lambda_{\ell}^{2}
		 \sum_{i=1}^{L_{h}}\int_{I_{i}} |\partial_h^+ g_\ell|^{2} \, |\partial_h^- (\sqmob(u^h))|^{2} \dx
		 \le
		 \sum_{\ell =1}^{\infty}
		 \lambda_{\ell}^{2}\ell^{2}
		 \frac{8 \pi^{2}}{L^{3}} 
		 \sum_{i=1}^{L_{h}}
		 \int_{I_{i}}  |\partial_h^- (\sqmob(u^h))|^{2} \dx \, .
	\end{align}
	For a continuous and elementwise linear function $a$ we have
	\begin{align}
		\partial_{h}^{-} a \big|_{I_{i}} = \frac{a_{i+1} -a_{i}}{h} \lambda(x) + \frac{a_{i} -a_{i-1}}{h} (1-\lambda(x)) \, .
	\end{align}
	Then, similar as before
	\begin{align}\notag \label{nineAnew5}
	 &\sum_{i=1}^{L_{h}} \int_{I_{i}}  |\partial_h^-
	 (\sqmob(u^h))|^{2} \dx
	 \\ \notag
	 &=
	 \sum_{i=1}^{L_{h}}\int_{I_{i}}  \left(\frac{\sqmob(u^{h}_{i+1}) -\sqmob(u^{h}_{i})}{h} \lambda(x) + \frac{\sqmob(u^{h}_{i}) -\sqmob(u^{h}_{i-1})}{h} (1-\lambda(x))\right)^{2} \dx
	 \\ \notag
	 &\le
	 \sum_{i=1}^{L_{h}}\int_{I_{i}}  \left(\frac{\sqmob(u^{h}_{i+1}) -\sqmob(u^{h}_{i})}{h}\right)^{2} \lambda(x) + \left(\frac{\sqmob(u^{h}_{i}) -\sqmob(u^{h}_{i-1})}{h}\right)^{2} (1-\lambda(x)) \dx
	 \\
	 &=
	 \sum_{i=1}^{L_{h}}\int_{I_{i}}  (\sqmob(\uh))_{x}^{2} \dx \, .
	\end{align}
	Combining \eqref{nineAnew4} and \eqref{nineAnew5}, and using Lemma~\ref{lem:lowerbound} and the mean-value theorem, we find
	\begin{align}\notag
	\frac{1}{2h}\sum_{\ell \in \mathbb{Z}} \lambda_{\ell}^{2}	R_{2}
		&\le  \sum_{\ell =1}^{\infty} \lambda_{\ell}^{2}\ell^{2}\frac{8 \pi^{2}}{2L^{3}} 
		\sum_{i=1}^{L_{h}}
		\int_{I_{i}}  (\sqmob(\uh))_{x}^{2} \dx
		\\ \notag
		&=\sum_{\ell =1}^{\infty}  \lambda_{\ell}^{2}\ell^{2}\frac{8 \pi^{2}}{2L^{3}}  h \sum_{i=1}^{L_{h}} \left| \tfrac{(u^{h}_{i+1})^{n/2} - (u^{h}_{i})^{n/2}}{h} \right|^{2}
		\\ \notag
		&=\sum_{\ell =1}^{\infty}  \lambda_{\ell}^{2}\ell^{2}\frac{8 \pi^{2}}{2L^{3}}  \frac{n^{2}}{4} h \sum_{i=1}^{L_{h}} \left(\sigma_{i}^{h} u^{h}_{i} + (1-\sigma_{i}^{h}) u^{h}_{i+1}\right)^{n-2} \left| \tfrac{u^{h}_{i+1} - u^{h}_{i}}{h} \right|^{2}
		\\
		&\le
		\sum_{\ell =1}^{\infty}  \lambda_{\ell}^{2}\ell^{2}\frac{8 \pi^{2}}{2L^{3}} 
		C_{osc}^{n-2} \frac{n^{2}}{4} \discint (u^{h}_{i})^{n-2} \left| \tfrac{u^{h}_{i+1} - u^{h}_{i}}{h} \right|^{2}   \, .		
	\end{align}
	The Term $R_{3}$ can be estimated in a similar way as $R_{2}$. For $R_{4}$ we have 
	\begin{align}\notag
		\frac{1}{2h}\sum_{\ell \in \mathbb{Z}} \lambda_{\ell}^{2}  R_{4} &= \frac{1}{2}\sum_{\ell \in \mathbb{Z}} \lambda_{\ell}^{2} \sum_{i=1}^{L_{h}} \int_{I_{i}}  |\partial_h^- \partial_h^+ g_\ell|^{2} (\sqmob(u^h))^{2}
		\dx
		\le
		\sum_{\ell =1}^{\infty} 
		\lambda_{\ell}^{2}
		\ell^{4}
		\frac{32\pi^{4}}{2L^{5}} 
		\sum_{i=1}^{L_{h}} \int_{I_{i}} (\sqmob(u^h))^{2}
		\dx
		\\
		&\le \sum_{\ell =1}^{\infty} 
		\lambda_{\ell}^{2}
		\ell^{4}
		\frac{32\pi^{4}}{2L^{5}} \discint (\ui)^{n}\, .
	\end{align}
	The remaining terms $R_{5}, \dots,R_{10}$ are estimated via Young's inequality. 
	For arbitrary $\eta > 0$ we have 
	\begin{align}
		R_{5}&\le
		\frac{\eta}{3} R_{1} + C_{\eta} R_{2}
		\, ,
		\\
		R_{6}&\le
		\frac{\eta}{3} R_{1} + C_{\eta} R_{3}
		\, ,
		\\
		R_{7}&\le
		\frac{\eta}{3} R_{1} + C_{\eta} R_{4}
		\, ,
		\\
		R_{8}&\le
		R_{2} + R_{3}
		\, ,
		\\
		R_{9}&\le
		R_{2} + R_{4}
		\, ,
		\\
		R_{10}&\le
		R_{3} + R_{4}
		\, .
	\end{align}
	Estimating these terms as above, combining the constants, and integrating in time, the result follows. 
\end{proof}

\begin{lemma}\label{lem:nablastrong}
	We have $\tuh_{x} \rightarrow \tu_{x}$ strongly in $L^{2}(\ort\times [0,T_{max}])$ almost surely.
	\begin{proof}
In the light of the uniform bounds on $\tuh_{x}$ and $\Delta_{h}\tuh$ provided by Proposition~\ref{prop:integral1}, the claim follows by standard arguments.
	\end{proof}
\end{lemma}

\noindent
\textbf{Acknowledgment.}
L.K. has been supported by the Graduiertenkolleg 2339 \textit{IntComSin} ''Interfaces, Complex Structures, and Singular Limits'' of the Deutsche Forschungsgemeinschaft (DFG, German Research Foundation) -- Project-ID 321821685. The support is gratefully acknowledged.

\bibliographystyle{abbrv}
\bibliography{thinfilm}

\end{document}